\documentclass{amsart}

\usepackage{graphicx} 
\usepackage[pdfa]{hyperref}
\usepackage{hyperxmp}[2020/03/01]
\usepackage{embedfile}[2020/04/01] 
\usepackage{amsmath}
\usepackage{amsthm,amssymb}
\usepackage{thmtools}
\usepackage{thm-restate}
\usepackage[usenames,dvipsnames]{xcolor}
\usepackage{color,mathrsfs,stmaryrd,tikz-cd,mathabx}
\usepackage{subcaption} 
\usepackage{caption} 
\usepackage{floatrow}
\usepackage{mathtools}
\usepackage[shortlabels]{enumitem}
\usepackage{yfonts}
\usepackage[final]{showkeys}
\usepackage[curve,matrix,arrow]{xy}
\usepackage[export]{adjustbox}
\usepackage[
	top=31mm,
	textwidth=163mm,
	textheight=230mm,
	]{geometry}

\newtheorem*{theorem-nonumb}{Theorem}
\numberwithin{equation}{section}
\newtheorem{theorem}[equation]{Theorem}
\newtheorem{lemma}[equation]{Lemma}
\newtheorem{prop}[equation]{Proposition}
\newtheorem{cor}[equation]{Corollary}

\theoremstyle{definition}

\newtheorem{definition}[equation]{Definition}
\newtheorem{remark}[equation]{Remark}

\newtheorem{notation}[equation]{Notation}

\newtheoremstyle{subresult}
{}
{}
{\itshape}
{0.2in}
{}
{.}
{0.3in}
{\thmname{#1}	\thmnumber{#2}	\thmnote{#3}}

\theoremstyle{subresult}
\newtheorem{subresult}{}[equation]

\newtheoremstyle{caseresult}
{}
{}
{\itshape}
{0.2in}
{}
{.}
{0.3in}
{\thmname{#1}\thmnumber{#2}\thmnote{#3}}

\theoremstyle{caseresult}
\newtheorem{caseresult}{}[equation]

\newcounter{maintheorem}
\renewcommand{\themaintheorem}{}

\newenvironment{maintheorem}[1]
{%
	\def\customname{#1}%
	\renewcommand{\themaintheorem}{\customname}%
	\refstepcounter{maintheorem}%
	\trivlist
	\item[\hskip\labelsep\bfseries Theorem~\themaintheorem.]%
	\itshape
}
{%
	\endtrivlist
}

\renewcommand{\phi}{\varphi}

\newcommand{\bF}{\mathbb{F}}
\newcommand{\bZ}{\mathbb{Z}}

\newcommand{\bN}{\mathbb{N}}
\newcommand{\bK}{\mathbb{K}}

\newcommand{\W}{\mathbb{W}}

\newcommand{\A}{\mathcal{A}}

\newcommand{\F}{\mathcal{F}}
\newcommand{\M}{\mathcal{M}}

\newcommand{\J}{\mathcal{J}}
\newcommand{\gM}{\mathfrak{M}}

\newcommand{\N}{\mathcal{N}}

\renewcommand{\L}{\mathcal{L}}
\newcommand{\cL}{\textswab{L}}
\renewcommand{\S}{\mathcal{S}}
\newcommand{\C}{\mathcal{C}}
\renewcommand{\H}{\mathcal{H}}
\newcommand{\K}{\mathcal{K}}
\newcommand{\Z}{\mathcal{Z}}
\renewcommand{\P}{\mathcal{P}}

\newcommand{\cU}{\mathcal{U}}
\newcommand{\V}{\mathcal{V}}

\newcommand{\coll}{\mathcal{O}}

\newcommand{\Hom}{\operatorname{Hom}}
\newcommand{\Aut}{\operatorname{Aut}}
\newcommand{\Out}{\operatorname{Out}}
\newcommand{\Inn}{\operatorname{Inn}}
\newcommand{\End}{\operatorname{End}}
\newcommand{\Syl}{\operatorname{Syl}}

\newcommand{\id}{\operatorname{id}}
\newcommand{\SL}{\operatorname{SL}}
\newcommand{\PSL}{\operatorname{PSL}}
\newcommand{\Sp}{\operatorname{Sp}}
\newcommand{\SU}{\operatorname{SU}}
\newcommand{\SO}{\operatorname{SO}}
\newcommand{\Sym}{\operatorname{Sym}}
\newcommand{\Alt}{\operatorname{Alt}}

\newcommand{\GL}{\operatorname{GL}}

\newcommand{\Spin}{\operatorname{Spin}}

\renewcommand{\O}{\operatorname{O}}

\newcommand{\wcl}{\operatorname{wcl}}

\newcommand{\Sub}{\operatorname{Sub}}

\makeatletter
\newcommand{\extp}{\@ifnextchar^\@extp{\@extp^{\,}}}
\def\@extp^#1{\mathop{\bigwedge\nolimits^{\!#1}}}
\makeatother

\newcommand{\ov}{\overline}
\newcommand{\norm}{\trianglelefteq}
 
\newcommand{\inv}{^{-1}}
\newcommand{\wde}{\widetilde}

\newcommand{\wht}{\widehat}
\newcommand{\fto}{\longrightarrow}

\newcommand{\<}{\langle}
\renewcommand{\>}{\rangle}
\newcommand{\fieldchar}{\mathrm{char}}

\newcommand{\h}{\textswab{h}}

\title[A local approach to a programme of Meierfrankenfeld I]{A local approach to a programme of Meierfrankenfeld: initial setting and the symmetric case}
\author{Edoardo Salati}
\address{Department of Mathematics,	RPTU University Kaiserslautern-Landau in Kaiserslautern, 67663 Kaiserslautern, Germany}
\email{edoardo.salati@rptu.de}

\begin{document}

\maketitle

\begin{abstract}
	\noindent
    A large $p$-subgroup of a group $G$ is a self-centralizing $p$-subgroup $Q \le G$ whose normalizer controls the normalizers of all the non-trivial, central subgroups of $Q$. In 2016 Meierfrankenfeld, Stellmacher and Stroth produced a result describing the $p$-local structure of a finite group having a large $p$-subgroup (the main examples arising from finite groups of Lie type in defining characteristic $p$). This result is a major success within the wider framework of studying groups of local characteristic $p$.\\
    We attempt to produce a result analogous to that of Meierfrankenfeld, Stellmacher and Stroth, but for fusion systems and localities. Previous work of Ellen Henke and the author shows that reasonable generalizations can be formulated and solved equivalently either within the realm of fusion systems or in the world of localities. In particular, in the present paper we set the stage for our analysis, showing how working within a locality grants a clear advantage: it allows to follow the same lines of reasoning as for a group. We therefore produce analogous reduction results and case subdivision as those in the 2016 result of Meierfrankenfeld et al. and, proceeding according to the analogy, we deal with occurrences of certain natural orthogonal modules and with the first of the cases that are to be studied, the so-called symmetric case.\\
    The remaining cases will appear in future publications.
\end{abstract}

\setlength{\parindent}{0pt}

\section{Introduction}
\subsection{Context}
Given a finite $p$-group on $S$, a saturated fusion system on $S$ is a category with objects the subgroups of $S$ and injective group homomorphisms as maps, subject to certain axioms; it models the conjugacy between subgroups of $S$ induced by the elements of a group $G$ containing $S$ as a Sylow $p$-subgroup. Indeed, any such group $G$ induces a saturated fusion system $\F_S(G)$ built by choosing $\Hom_{\F_S(G)}(P,Q) := \{c_g \colon P \to Q \mid g \in G \textit{ s.t. } P^g \le Q \}$. Due to the symmetry among Sylow $p$-subgroups of $G$, $\F_S(G)$ does not depend on the choice of $S$ up to a suitable notion of isomorphism of fusion systems, hence $\F_S(G)$ is known as the $p$-fusion system of $G$ and sometimes more simply denoted $\F_p(G)$. Fusion systems of the form $\F_S(G)$ are called realizable; those saturated fusion systems that instead admit no group $G$ realizing them are called exotic.\\

The success of fusion systems is probably due to the role they play in connecting different areas of mathematics, from finite group theory to modular representation theory and to certain parts of algebraic topology.\\
Their study in connection to finite group theory is related to the classification of the finite simple groups (CFSG): this is undoubtedly one of the most successful mathematical achievements of the past century, nevertheless its proof still remains quite inaccessible. An important part of the theory of fusion systems is indeed being developed for providing different approaches to the CFSG.   Two of them are relevant for this paper: the programme of Aschbacher, aiming at reclassifying the so-called finite simple groups of component type via their $2$-fusion systems, and the programme initiated by Meierfrankenfeld, with the goal of describing the $p$-local structure of the finite groups of so-called characteristic $p$-type (also named of local characteristic $p$). In the realm of finite simple groups these two programs may be considered complementary, but we need to provide some context to justify this.\\

In a finite group $G$ the $p$-local subgroups are the normalizers of the non-trivial $p$-subgroups of $G$; by referring to the $p$-local structure of $G$ we mean the set of $p$-local subgroups of $G$ and of their mutual relationships, which may be modeled by the $p$-fusion system of $G$.\\
The Gorenstein-Walter Dichotomy Theorem states that, apart from some small $2$-rank exceptions, every finite simple group falls in one of two classes: the groups of component-type, that is groups containing an involution centralizer which, after modding out the largest normal subgroup of odd order, admits a component, and the groups of characteristic $2$-type, which are groups whose $2$-local subgroups $L$ all satisfy $C_L(O_2(L)) \le O_2(L)$ (are of characteristic $2$). When studying the groups of component type, the presence of a normal odd-order subgroup in the involution centralizer is the major source of complexity, forcing a long and detailed case-analysis. On the other hand, for a finite group $G$ with Sylow $p$-subgroup $S$ and $p'$-normal subgroup $N$ one can show that $\F_S(G) \cong \F_S(G/N)$ (where we are identifying $S \cong SN/N$), hence the $2$-fusion system of $G$ avoids the obstructions coming from the odd-order normal subgroups. Based on this idea Aschbacher proposed a new way of classifying the finite simple groups of component type by first classifying their corresponding $2$-fusion systems.\\
On the other hand, the groups of characteristic $2$-type require a different approach; the techniques used for studying the involution centralizers no longer work when applied to involutions, but are still often useful after switching to a suitable odd prime. Unfortunately, the choice of such odd prime is not uniform and yields another long and complicated case analysis. Since the prototypes of groups of local characteristic $p$ are the finite groups of Lie type in their defining characteristic, the idea of Meierfrankenfeld et al. consists in detecting and making use of properties which, although they are a reflection of properties shared by the corresponding algebraic groups over algebraically closed fields, are expressible in the language of finite groups (of Lie type).\\

In analogy with the finite simple groups, also saturated fusion systems can be categorized into two groups (for all primes and with no exceptions): those of component type and those of characteristic $p$-type (where the fusion system is on a $p$-group). Aschbacher's programme requires precise knowledge about the simple $2$-fusion systems of component type before even attempting to rebuild the groups from them. The known examples are indeed those coming from finite simple groups together with the Benson-Solomon fusion systems, the only (infinite) family of exotic $2$-fusion systems that was found. It is a conjectured that no other simple exotic $2$-fusion system exists.\\
On the other hand, in this paper we work on the opposite spectrum, dealing exclusively with fusion systems of characteristic $p$-type (or slight variations of them) and following the ideas and, when possible, the methods that pertain to the work of Meierfrankenfeld et al. For this reason one can think of it as a complement of Aschbacher's work at the prime $2$ and, even though prior work of Aschbacher has already produced results about fusion system of characteristic $2$-type, our approach is substantially different and allows us to deal with all primes at the same time.\\

Moreover, in contrast with the general belief about the rarity of exotic $2$-fusion systems, there is a relative abundance of exotic fusion systems at odd primes. The search for exotic fusion systems and, more generally, questions about exoticity have propelled research on fusion systems for many years. Two long standing conjectures are indeed related to the exotic behavior of certain fusion systems. We have already encountered the conjecture regarding the Benson-Solomon systems; in addition, it is possible to build saturated fusion system also starting from certain blocks of the group algebra over an algebraically closed field of characteristic $p$, known as block fusion systems. It is well known that every realizable fusion system may also be realized as a block fusion system by taking the principal block; whether the converse also holds is however still an open problem.\\
The advantage of working uniformly with all prime that our approach offers is therefore not negligible. Gaining enough control over the local structure often allows to rebuild the fusion systems; because of the very general nature of our results, they would be of great use in future classification programs, helping shedding light into new exotic patterns, or excluding them otherwise.\\

There is one extra tool that needs to be mentioned, since it plays a fundamental role within this paper: this is \emph{localities}. They are a subclass of the more general class of \emph{partial groups} defined by Chermak in \cite{Chermak:2013} in a solution to the last open problem that proved the Martino-Priddy conjecture for all saturated fusion systems. The Martino-Priddy conjecture relates the $p$-fusion system of a finite group $G$ with a certain (homotopy class of a) topological space, namely the Bousfield-Kan $p$-completion of the classifying space $BG$ of $G$. This is normally denoted by $BG^\wedge_p$; it is uniquely determined by, as well as it uniquely determines, $\F_p(G)$, and it is therefore called the classifying space of $\F_p(G)$.\\
Partial groups and localities were used by Chermak in a proof of an extension of the Martino-Priddy conjecture to all saturated fusion systems: any saturated fusion system $\F$ builds a unique (up to homotopy) $p$-complete space $|\L^c(\F)|$ (where the vertical bars denote nerve composed with the geometric realization functors).\\
Part of Chermak's work was dedicated to prove that, up to a certain choice of subgroups of the $p$-group $S$, there is a one-to-one correspondence between saturated fusion systems on $S$ and localities on $S$ satisfying certain properties. These are nowadays known as linking localities. By the work of Gonz\'alez \cite{Gonzalez:2015}, Broto--Gonz\'alez \cite{Broto-Gonzalez:2021} and Hackney--Lynd \cite{Hackney-Lynd:2025} partial groups may be naturally regarded as simplicial sets and, if $\L$ is the the linking locality associated to $\F$, the geometric realization of $\L$ has the same homotopy type of $|\L^c(\F)|$. This correspondence comes with its own set of questions, often having more of an algebraic topology or homotopical algebra flavour. For example, it still remains unclear what topology-theoretic properties detect the $p$-complete spaces that are the classifying space of some saturated fusion systems. In addition, the introduction of localities provides an extra layer of complexity between groups and fusion systems. For example, let $S \in \Syl_p(G)$ and $Q \le S$ be a subgroup with the property that $Q$ is its unique $G$-conjugate contained in $S$. Then in the group $G$ there are as many conjugates of $Q$ as $|\Syl_p(G)|$, in $\F_S(G)$ there is only $Q$ and no other conjugate, but in an associated linking locality $\L$ there can be multiple conjugates of $Q$. It remains unclear how many; moreover, not all such conjugates share the same status within $\L$, some having better properties than others. This precise situation will be encountered in this paper, and it will somewhat be both an advantage (with respect to working with the fusion system) and a nuisance at the same time.\\
Behavior like that described above is one of the reasons leading the author to also investigate the more general class of partial groups, with the hope of then specializing results to the better behaved class of localities. This is better done through the formalism introduced by Hackney--Lynd in \cite{Hackney-Lynd:2025} and in \cite{Hackney-Lynd:2025bis}. With regard to the example above, the locality $\L$ induces a partial action (of a partial group) as per \cite[Definition 5.4]{Hackney-Lynd:2025bis} on the (transitive closure of the) set of $\L$-conjugates of $Q$. The behavior of such partial actions remains nowadays still obscure and is therefore object of current studies.\\

Nevertheless, the possibility of switching from fusion systems to linking localities will be an important card in our hand. Their intrinsic structure, much more affine to a classical algebraic structure, in particular to groups, than to a category, allows for a more efficient translation of the language and results of group theory. During the development of the results, after showing that these can be equivalently formulated either in terms of fusion systems or of localities (most of the work here coming from \cite{Henke-Salati:2026}), we will quickly abandon fusion systems.  The ideas and methods that appear in the work of Meierfrankenfeld et al. are much easier to replicate when working with a locality and may appear cumbersome when translated into the language of fusion systems.\\

\subsection{Organization of the paper}
This paper, together with \cite{Henke-Salati:2026} and with subsequent ones yet to be published, reorganizes and adds to the results that were part of the PhD thesis of the author. It can be equivalently thought of in different ways: as a complementary side of Aschbacher's programme, as a natural generalization of the goal of Meierfrankenfeld et al. from groups to fusion systems or, per se, as a part of the currently active research involving fusion systems. The objective in mind is indeed to investigate the ideas and results from the programme of Meierfrankenfeld in order to replicate them in the context of fusion systems. In this way one could obtain fairly detailed descriptions of the structure of the fusion systems of characteristic $p$-type, or at least of a large portion of them. \\
Ellen Henke, supervisor of the author's PhD thesis, identified \cite[Theorem A]{Meier-Stell-Stroth:2012} as a good candidate to try to understand how results and techniques from the work of Meierfrankenfeld, Stellmacher and Stroth transfer to fusion systems. We will refer to \cite{Meier-Stell-Stroth:2012} as the Local Structure Theorem or with the acronym LST; the authors here investigate the $p$-local structure of a finite group $G$ mainly under two assumptions: the existence of a large $p$-subgroup $Q$ and the existence of a maximal, $p$-local subgroup not normalizing $Q$. For context we provide the definition of a large $p$-subgroup of a group.
\begin{definition}
	\label{def: large subgroup of a group}
	A $p$-subgroup $Q$ of a finite group $G$ is \emph{large} if the following hold:
	\begin{itemize}
		\item[$(SC)$] $Q$ is self-centralizing, that is $C_G(Q) \le Q$,
		\item[$(Q!)$] for every non-trivial subgroup $U \le Z(Q)$, $N_G(U) \le N_G(Q)$.
	\end{itemize}
\end{definition}
Their main goal, that is a description of the $p$-local structure of $G$, reflects into our results: the focus of this paper is on the inner structure of a fusion system satisfying hypothesis analogous to those of $G$ in \cite[Theorem A]{Meier-Stell-Stroth:2012}, i.e. on the automorphism groups of the objects of the fusion system. The problem of rebuilding the fusion system from knowledge about its local structure is, instead, left for future studies.\\

In section \ref{background} we begin by collecting the notation and conventions in use together with the names and descriptions of certain recurring modules. We then turn our attention to establishing the background knowledge that is necessary. We will assume the reader to be familiar with the basic knowledge about fusion systems, since they are more widely known within the community and play a marginal role in the proofs; we focus instead on some more specific definitions and results about finite group theory. We also provide the necessary tools for understanding and working with localities, which is fundamental knowledge for this paper, especially since intuition coming from group theory, something the reader might be much more familiar with, can be misleading when applied to localities.\\

Section \ref{large subgroups} is mainly dedicated to setting up the stage relative to groups, fusion systems and localities with a large subgroup. A large portion of this has already been worked out in \cite{Henke-Salati:2026}: after defining large subgroup of fusion systems and localities, we provide a summary of the results that make clear the relationships among fusion systems, localities and groups under the assumption of having a large subgroup. Here is where the one-to-one correspondence between a saturated fusion system and its so-called subcentric linking locality is described. Moreover, any finite group with a large $p$-subgroup yields a fusion system and a linking locality with a large $p$-subgroup, even though the viceversa is shown to be false in general in \cite{Henke-Salati:2026}; thus we can expect the results in this paper to encompass those of Meierfrankefeld et al. and possibly be more general.\\
The chapter also contains some additional results that are more specific to the goal of this paper.\\

Section \ref{section: strategy of LST for groups} is a necessary addition for the reader who is not familiar with the work of Meierfrankenfeld, Stellmacher and Stroth, more specifically with the LST: we provide a rather detailed description of their strategy for the study of a group $G$ with a large $p$-subgroup $Q$. This is particularly important to us since it is possible to adapt large portions of their techniques to the locality we focus on.\\
Condensing in only few lines, one would start by recalling that a $p$-parabolic subgroup of $G$ is a subgroup of $G$ containing some Sylow $p$-subgroup. A fundamental result in the LST shows that one can restrict the attention to parabolic subgroups $M$ that satisfy certain additional properties, among which such that $MC_G(Y_M)$ is a maximal, parabolic, local subgroup (namely $N_G(Y_M)$), with $Y_M$ a specific elementary abelian $p$-subgroup, normal in $M$ and that will be defined further on. Note however that $Y_M$ becomes then an $\bF_p M$-module. By \cite[]{Meier-Stell-Stroth:2012} every other parabolic, local subgroup of $G$ that does not normalize a conjugate of $Q$ is then a $G$-conjugate of some subgroup of $MC_G(Y_M)$ and the conjugation map realizes a module-isomorphism between $Y_M$ and its corresponding conjugate.\\
In this same chapter we also explain the case subdivision that was utilized in \cite{Meier-Stell-Stroth:2012} for the identification of the pairs $(M,Y_M)$ for $M$ as above. The general idea consists in detecting some convenient non-trivial offender on $Y_M$ (or some other related subgroup); theorems that classify modules with offenders provide then lists of possible pairs for $(M,Y_M)$, which then need to be sifted through to exclude those that are in conflict with the general hypothesis on the group $G$. In this paper we restrict to the initial set-up, the analysis of the orthogonal modules (these need a separate, lengthy argument) and to the symmetric case, the first of the various cases treated in the LST.\\

Section \ref{basic property in localit} is the core of the paper; we translate the reduction to the pairs $(M,Y_M)$ from the setting involving a group $G$ to our situation, working with a locality $\L$ with a large subgroup. Here is where it first becomes clear how avoiding fusion systems is a smart option: localities provide centralizers, which may be used in analogy to what is done for the LST to define a poset structure on certain subgroups. The groups that correspond to $N_G(Y_M)$ (and contain $M$) are then shown to exist (as groups) also within $\L$ and are identified as the maximal elements of such poset, and the subgroups analogous to $M$ are then searched within these. This is mostly collected in Theorem \ref{theorem: main reduction}, which provides the actual reduction.\\
At this point one is finally able to formalize the hypothesis on the locality $\L$ that are a reasonable translation of those satisfied by a group in the LST; we provide them here for referencing, even though not all the notation has been defined yet.

\begin{restatable}{hyp}{hypothesisLocality} \label{hyp on the locality}
	$(\L,\Delta,S)$ is a $\K_p$-locality of objective characteristic $p$, $Q \le S$ a large subgroup of $\L$ such that $\L$ is $Q$-replete and there exists a parabolic subgroup $M$ of $\L$ such that $M \not \le N_\L(Q)$.
\end{restatable}

One can then also give an explicit characterization of the fusion systems that correspond to such a locality; in detail we prove the following.

\begin{maintheorem}{B}
	\label{theorem:translation of hyp on local to fus systems}
	Let $(\L,\Delta,S)$ be a locality with associated fusion system $\F:=\F_S(\L)$. Then $\L$ satisfies hypothesis \ref{hyp on the locality} (possibly with the exclusion of the $\K_p$-hypothesis) if and only if $\F$ is a saturated fusion system with large subgroup $Q$ and such that there exists an $\F$-centric subgroup $Y$ normal in $S$ and satisfying $\xi(Q) \not \norm \Aut_\F(Y)$, where $\xi \colon S \fto \Aut_\F(Y)$ maps each element to the conjugation map induced by it.
\end{maintheorem}
The $\K_p$-hypothesis for the locality is the natural translation of the same hypothesis for groups; both are declared together with the general notation and conventions adopted. It is only meaningful in a CFSG-free setting and trivial if one assumes the CFSG.\\
The author believes that Theorem~\ref{theorem:translation of hyp on local to fus systems} may help those researchers, who are versed in fusion systems but not so much in localities, at understanding the scope and application of the results we develope, especially since the fusion system community is quite larger than the community actively working with localities.\\

Theorem \ref{theorem: main reduction} also suggests how to further proceed, and this is explained and taken care in section \ref{case history and orthog case}: one attempts to split the analysis into cases that resemble as close as possible those appearing in the LST, hoping to find the same offenders and apply in each case the same classification results for the possible pairs $(M,Y_M)$. When successful, this strategy is very effective and convenient: one is then only left to sieve through the same lists of possible modules, excluding those whose properties are incompatible with the general assumptions on the locality, just as in the LST. Most of the time the conflict is related to the presence of the large subgroup $Q$ and here is were the analysis in the locality can become harder. Conjugates of $Q$ play an important role in this process, but conjugation in a locality is definitively more delicate to handle, as explained earlier; it is, however, still better than working directly with the fusion system, since this contains no conjugate of $Q$. Unfortunately there is at least one case where this strategy does not work and a slightly stronger hypothesis may be required. However, we won't deal with it in the present paper and will be a topic for future publications.\\
Note also that one of the cases is some sort of a ``black box" in the LST: the authors cannot classify the possible modules for $(M,Y_M)$ (this is the only place where it happens) and can only prove some property of $M$. In the PhD thesis of the author it is, instead, shown that no locality appears in such case, suggesting that it is some exceptional behavior of finite groups depending not exclusively on their $p$-local structure. This will also be part of some future publication.\\
We close the chapter following the idea of dealing separately with the occurrences of the orthogonal modules for $\O_{2n}^{\epsilon}(2)$ for the pairs $(M,Y_M)$, which need a special and quite lengthy reasoning. This is one of those situations where the corresponding analysis undertaken in the LST adapts in a quite straightforward way to our locality. The outcome is Theorem \ref{theorem CL}, which clearly translates \cite[Theorem C]{Meier-Stell-Stroth:2012}.\\

Finally, we conclude this paper with section \ref{symm case}, in which we deal with the symmetric case for the locality. The first technical difference in the locality $\L$ over the $p$-group $S$, with respect to the same case for groups, consists in showing that it is always possible to find a conjugate $Y_M^u$ of $Y_M$ that is a non-trivial offender on $Y_M$ and is contained in $S$ (note that $Y_M$ is always assumed to be a subgroup of $S$). After that we get to apply the \nameref{thm:Q!FF} just as in the LST and obtain the same list of modules to check. Some of the modules need some extra work and here becomes evident how a clever way of choosing conjugates of the large subgroup $Q$ is very important in the locality, but the final result, Theorem \ref{theorem DL}, is very close to \cite[Theorem D]{Meier-Stell-Stroth:2012}. Indeed, the same list of possible modules for $(M,Y_M)$ is found.\\

Versions of Theorems~\ref{theorem CL} and \ref{theorem DL} for fusion systems may be obtained by means of Theorem~\ref{theorem:translation of hyp on local to fus systems} and by the usual translation machinery between localities and fusion systems, see \cite{ChermakHenke:2021}.

\subsection{Notes to the reader}
Due to the nature of the goal itself, as well as the nature of localities, which are locally groups, see \ref{lemma:properties of localities}(a), the choices made about the exposition of the results need a little explanation.\\
The strategy followed can be summarized as a two steps process, similarly to what done by Meierfrankenfeld, Stellmacher and Stroth: first, one tries to reduce the analysis of the locality to a local analysis, preserving as much of the local conditions that one finds in the LST as possible, then one studies the local configurations, which now involve groups. In particular this second step, with respect to the analogous process in \cite{Meier-Stell-Stroth:2012}, can range from needing a new argument, to requiring a simple adaptation of the methods in use in the LST, to be identical (or almost identical) to analogous results appearing in \cite{Meier-Stell-Stroth:2012}. To complicate matters, some intermediate results can show all of these behaviors in different sections all at once.\\
For this paper the author has decided to favor readability above everything else, paying particular attention to two main traits: easiness in the understanding of proofs and faithfulness with the notation and results from the LST.
Since the analysis of the local configuration often follows a line of reasoning close to that in the LST and the reduction to the local analysis can also be sometimes managed similarly, the author has decided to split and organize the results as close as possible to the choices made in \cite{Meier-Stell-Stroth:2012}. This is made clear by our notation: thus in ``\textbf{Lemma 3.4} (2.2)" 3.4 refers to the numbering within the current paper, whereas 2.2 refers to the numbering of a statement in \cite{Meier-Stell-Stroth:2012} that closely resembles our Lemma 3.4.\\
There are also situations where a large enough part of our analysis ends up coinciding with that undertaken in the LST and, in such cases, we often refer the reader to the specific section, unless it impairs readability.\\
More complicated are the situations where only minor changes or small additions are needed. Here readability and easiness in the understanding of the arguments are again the key aspects considered. We do sometimes refer the reader to the corresponding results, noting the needed changes in a separate remark when these are scarce and minor, otherwise we write down again the argument in its entirety, with all the needed adjustments, to allow for a better reading flow.

\vspace{1cm}

\section*{Acknowledgments}
The author is grateful to Ellen Henke for introducing him to the world of localities and for suggesting such a beautiful topic as the main focus for his PhD thesis. Many thanks also to Bernd Stellmacher, who has been a wonderful, even though unofficial, second supervisor and to Gemma Parmeggiani, who welcomed the author in Padova multiple times, allowing for many hours of fruitful conversation with her and Bernd, all of it contributing to the result of this article.

\newpage

\section{Notation and necessary background}
\label{background}

\subsection{Notation and conventions}
We collect some of the notation that will be used throughout the paper, together with some conventions that we will adopt in order to keep the notation not excessively heavy.\\
In order to facilitate the reader we also summerize here the naming of the modules that will arise throughout this paper as well in following ones. We also provide some additional context for each module in order to facilitate the reader in those situations where such modules occur.\\

Let $G$ be a group, $S$ a finite $p$-group ($p$ a prime) and $\F$ a fusion system over $S$.
\begin{itemize}
	\item Set $|G|_p := p^k$ for $k \in \bN$ such that $p^k \divides |G|$ but $p^{k+1} \not \divides |G|$.
	\item We denote by $\Sub(G)$ the set of subgroups of $G$, and we will often consider it as a poset ordered by inclusion.
    \item If $g \in G$, then for any $x \in G$ we set $x^g := g\inv x g$; the induced conjugation map is denoted by
    \[
    c_g : G \fto G, \quad x \fto x^g = g\inv x g.
    \]
    Moreover, if $A,B$ are subsets of $G$ we set
    \[
    \Hom_G(A,B) := \{c_g :A \fto B \mid g \in G, \, A^g \subseteq B \}.
    \]
    In a similar fashion, if $A$ is a subgroup of $G$, we write
    \begin{align*}
        \Aut_G(A) &:= \{c_g \in \Hom_G(A,A) \mid \textit{$c_g$ is an automorphism of $A$} \},\\
        \Out_G(A) &:= \Aut_G(A) / \Inn(A).
    \end{align*}
    \item If $A$ is a subset of $G$, then we write $A^\natural := A \setminus \{1_G\}$.
    \item $G$ is a $\K_p$-group if every $p$-local subgroup of $G$ has simple sections that are known finite simple groups; similarly, a $\K_p$-locality is a locality over a $p$-group such that every local subgroup has simple sections that are known finite simple groups. We refer to \ref{def:local and parabolic of localities} for the definition of local subgroups of localities. The $\K_p$-hypothesis is only necessary if one needs to use these results in a proof of the CFSG by a minimal counterexample argument.
    \item $Y_G$ denotes the (unique) largest, elementary abelian, normal, $p$-reduced, $p$-subgroup of $G$. See \ref{definition of p reduced} and \ref{lemma:unique largest p-reduced} for more details.
    \item $\F^{cr}$ is the set of $\F$-centric, radical subgroups of $S$; \\
     $\F^c$ denotes the set of $\F$-centric subgroups of $S$; \\
      $\F^s$ is the set of $\F$-subcentric subgroups of $S$ (see Definition \ref{def:subcentric}).
    \item $\F^f$ is the set of fully normalized subgroups of $S$ and, if $P \le S$, then $P^\F$ is the set of $\F$-conjugates of $P$.
\end{itemize}

\vspace{0.8cm}
Let now $\L=(\L,D,\Pi,i)$ be a partial group as per Definition~\ref{def:partial group}.

\begin{itemize}
	\item Let $x,y \in \L$, then conjugation of $x$ by $y$ is denoted as for groups:
	\[
		c_y \colon x \fto x^y,
	\]
	where we implicitly assume that $(y\inv,x,y) \in D$ and set $x^y := \Pi(y\inv,x,y)$.\\
	If $\H$ be a subset of $\L$ we write
	\[
		\H^y := \{ h^y \mid h \in \H\}
	\]
	and again implicitly assume that $(y\inv,h,y) \in D$ for all $h \in \H$.
	\item We reserve the writing $\H \le \L$ for the cases when $\H$ forms a subgroup of $\L$, that is $\W(\H) \subseteq D$.
	\item $\< \H \> \subseteq \L$ is the partial subgroup of $\L$ generated by $\H$, that is taking as domain for $\H$ the entire set $\W(\H) \cap \L$. It is often possible to form ``smaller" partial groups embedded in $\L$ and containing $\H$ by taking smaller sets of words as domain; these are often referred to in the literature as \emph{impartial subgroups} and play in this paper a marginal role.
\end{itemize}

	\vspace{0.8cm}
	Let $V$ be a finite dimensional vector space over the finite field $\bK$, $f$ a bilinear form on $V$ and $h$ a quadratic form on $V$ such that $f$ is the associated bilinear form. If $\bK$ is a quadratic extension of the subfield $\bF$ and $\sigma$ is the unique $\bF$-automorphism of $\bK$ of order $2$, we also allow $f$ to be the hermitian form on $V$ associated with $\sigma$.\\
	Let also $U$ denote a subspace of $V$.
	\begin{itemize}
		\item If $f=0=h$, we say that $(V,f,h)$ is a \emph{linear space}.
		\item $h(v) = \displaystyle \frac{1}{2} f(v,v)$ whenever $\fieldchar(\bK) \ne 2$; in even characteristic $f$ is given by $h(u+v) = h(u) +h(v) + f(u,v)$.
		\item $v,u \in V$ are \emph{orthogonal} (or perpendicular), and we write $v \perp u$ if $f(u,v)=0$; note that the cases above are precisely the cases where orthogonality is a symmetric relation.
		\item $v \in V$ is \emph{isotropic} if $f(v) = 0$; $U$ is \emph{isotropic} if $f|_{U \times U} =0$.
		\item $v \in V$ is \emph{singular} if $h(v) = 0$; $U$ is \emph{singular} if isotropic and it consists of only singular vectors.
		\item $U^\perp := \{v \in V \mid f(v,u)=0 \textit{ for all } u \in U\}$; the \emph{bilinear radical} of $V$ is $V^\perp$.
		\item The radical of $U$ is $\mathrm{rad}(U) := \{u \in U \cap U^\perp \mid h(u)=0 \}$; we say that $U$ is \emph{non-degenerate} if $\mathrm{rad}(U) =0$.
		\item $\S(U) := \{ \textit{$1$-dimensional, singular subspaces of $U$} \}$.
		\item An \emph{isometry} of $V$ is a $\bK$-automorphism $\phi$ of $V$ such that $h(\phi(v))=h(v)$ for every $v \in V$.
	\end{itemize}
	We will restrict our attention to \emph{classical spaces}, that is triples $(V,f,h)$ which are either linear or non-degenerate. Then we always have $V^\perp = \mathrm{rad}(V)$ unless $\bK$ has characteristic $2$, $f \ne 0$ and $\dim_{\bK}(V)$ is odd, in which case $0 \le \mathrm{rad}(V) < V^\perp$, the latter subspace being $1$-dimensional and non-singular. The group $Cl(V)$ of all isometries of a classical space $V$ is known as a \emph{classical group}: here we denote the general orthogonal groups by $\O_n^\epsilon(q)$, where $n$ is the dimension of the vector space and $\epsilon \in \{\emptyset,+,-\}$. We suggest the reader to compare with \cite[Appendix B]{Meier-Stell-Stroth:2012}; indeed our notation was inspired by that in use for the Local Structure Theorem.\\
	We also suggest \cite{Grove:2001} as a source of information for the finite classical groups.\\
	
	We can now switch to the naming of certain recurring modules; we will follow the same conventions used in \cite[Appendix A.2]{Meier-Stell-Stroth:2012}, but will provide some extra details.\\
	If $V$ is a $\bK$-vector space and $G$ a group acting linearly, on the right, on $V$, then there is induced left action of $G$ on the dual space $V^*:= \Hom_\bK(V,\bK)$; by twisting via inverse elements we can obtain again a right action, which then works as follows
	\[
		(f \cdot g)(v) := f(v \cdot g\inv) \qquad \forall f \in V^*, \, \forall g \in G, \, \forall v \in V.
	\]
	We denote by $\Lambda^2(V)$, $S^2(V)$ and $U^2(V)$ the exterior, respectively symmetric and unitary, squares of $V$. Rcall that the first two can be constructed as quotients of $V \otimes V$ and are the homogeneous component of degree $2$ of the exterior, respectively symmetric, tensor algebra. A similar construction is possible for $U^2(V)$ after introducing the conjugate vector space $\bar{V}$ that twists scalar multiplication by a field isomorphism; $U^2(V)$ becomes then the subalgebra of fixed points of $V^* \otimes (\bar{V}^*)$ under twisted flip operation (which enforces the hermitian symmetry condition). If $\dim_{\bK}(V)=m$, then one sees that
	\[
		\dim_{\bK}(\Lambda^2(V)) = \binom{m}{2}, \; \dim_{\bK}(S^2(V)) = \binom{m+1}{2} \; \text{and} \; \dim_{\bF}(U^2(V)) = m^2,
	\]
	where $\bK$ is a quadratic extension of $\bF$.\\
	Let now $V$ be an $\bF_pH$-module, $K$ a group and $W$ and an $\bF_p K$-module; assume that there exist a group epimorphism
	\[
	\tau \colon H \fto K/C_K(W), \quad h \mapsto \tau_h,
	\]
	and a \emph{$\tau$-equivariant} $\bF_p$-isomorphism $\phi  \colon V \fto W$, that is such that for every $h \in H$ the following square commutes 
	\begin{center}
		\begin{tikzcd}
			V \ar[r,"\phi"] \ar[d, "h"] & W \ar[d,"\tau_h"] \\
			V \ar[r,"\phi"]		& W
		\end{tikzcd}
	\end{center}
	\begin{itemize}
		\item If $K \le \GL_{\bF_p}(W)$, then $V$ is a \emph{natural $K$-module} for $H$.
		\item Let $V$ be a faithful $G$-module and $\K$ a non-empty and $G$-invariant set of subgroups of $G$; then $V$ is a \emph{natural $\SL_2(q)$-wreath product module} for $G$ with respect to $\K$ if $\displaystyle V = \bigoplus_{K \in \K} [V,K]$, $\displaystyle \<\K\> = \bigtimes_{K \in \K} K$ and for each $K \in \K$, $K \cong \SL_2(q)$ and $[V,K]$ is a natural $\SL_2(q)$-module.
		\item We will often encounter the case where $K=\SL_\bK(W_0)$ for some $\bK$-vector space $W_0$ and $W$ is one of $\Lambda^2(W_0)$, $S^2(W_0)$ or $U^2(W_0)$. Then $W$ is the exterior, respectively symmetric or unitary, square of a natural $\SL_\bK(W_0)$-module for $H$.
		\item If $K \le \Sym(I)$ for a finite set $I$, then $\bF_p^I$ becomes a $K$-module via the action $(w_i) \cdot \pi := (w_{\pi\inv(i)})$.
		\begin{itemize}
			\item For $W = \bF_p^I$ we call $V$ a \emph{$\bF_pK$-permutation module} for $H$.
			\item Assume for simplicity that $|I|=: n \ge 5$ and suppose that $K \in \{\Sym(I), \Alt(I)\}$. Inside $\bF_p^I$ one can detect two $K$-submodules: the trivial module 
			\[
			 \< (1,\dots,1) \> := T,
			\]
			and the augmentation module
			\[
				A := \{(w_i) \in \bF_p^I \mid \sum w_i =0 \}.
			\]
			It is not difficult to see that
			\begin{align*}
				\text{if $p \not \divides n$, then} \qquad & \bF_p^I = T \oplus A & \text{and $A$ is the unique, simple, non-central section,} \\
				\text{if $p \divides n$, then} \qquad & T \le A \le \bF_p^I  & \text{and $A/T$ is the unique, simple, non-central section}. 
			\end{align*}
		\end{itemize}
		If $W$ is such unique, simple, non-central section, we say that $V$ is a \emph{natural $\bF_p K$-module} for $H$.\\
		Note that exceptions to the existence and uniqueness of such simple section may arise for $n \in \{2,3,4\}$.
		\item The group $G_2(q)$ acts on a $7$-dimensional vector space (spanned by the octonions $\{i,j,ij,l,il,jl,(ij)l\}$) and such action preserves an symmetric bilinear form, allowing to embed $G_2(q) \le \SO_7(q)$. If $q$ is odd, such $7$-dimensional module is simple. If $q$ is even, such module has instead a bilinear radical; quotienting over it we obtain a $6$-dimensional simple module over which $G_2(q)$ acts preserving a symplectic form. In particular, $G_2(q) \le \Sp_6(q)$ for $q$ even.\\
		If $K=G_2(q)$ and $W$ is the $7$-dimensional simple module, if $q$ is odd, or the $6$-dimensional simple module, if $q$ is even, as above described, then $V$ is a \emph{natural $G_2(q)$-module} for $H$.
		\item Let $(U,f,h)$ be a quadratic space; the corresponding Clifford algebra $C$ is the quotient of the tensor algebra $T(U)$ modulo the two-sided ideal generated by the tensors $u \otimes u - h(u)$ and is a $\bZ / 2\bZ$-graded algebra, so $C = C_0 \oplus C_1$. $C_0$ is a subalgebra, known as the even Clifford algebra. The Spin group $\Spin(U)$ is a certain subgroup of $C_0^\times$. A minimal right ideal of $C$ is called a \emph{spin module}; a minimal right ideal of $C_0$ is called a \emph{half spin module}.\\
		Note that, when $U$ is odd-dimensional, the spin module splits as a sum of isomorphic copies of simple $C_0$-modules; when $U$ is even-dimensional, the spin module splits, generating two non-isomorphic, simple $C_0$-modules, which are isomorphic to the half spin modules (in case the form is of minus type, this happens after suitably extending the field).
	\end{itemize}

\subsection{Groups}
Let $G$ be a finite group and $p$ be a prime.
\begin{definition}
    \begin{itemize}
    \item $G$ is \emph{of characteristic $p$} if $C_G(O_p(G)) \le O_p(G)$; $G$ is \emph{almost of characteristic $p$} (also known in the literature as \emph{$p$-constrained}) if $G/O_{p'}(G)$ is of characteristic $p$.
    \item The $p$-local subgroups of $G$ are the subgroups of the form $N_G(P)$ for some non-trivial $p$-subgroup $P \le G$.    $G$ is \emph{of local characteristic $p$} (also \emph{of characteristic $p$-type} in the literature) if every $p$-local subgroup of $G$ is of characteristic $p$.
    \item The \emph{$p$-parabolic subgroups} of $G$ are those subgroups of $G$ containing some Sylow $p$-subgroup. $G$ is \emph{of parabolic characteristic $p$} if all the parabolic, $p$-local subgroups of $G$ are of characteristic $p$.
    \end{itemize}
\end{definition}
Note that $G$ is of characteristic $p$ if and only if $F^*(G) = O_p(G)$. The most important examples are the finite simple groups of Lie type in their defining characteristic.

\vspace*{0.6cm}
\begin{definition}
	\label{definition of p reduced}
    A normal subgroup $A$ of $G$ is $p$-reduced if $O_p(G/C_G(A)) = 1$.
\end{definition}

\begin{lemma}\label{lemma:unique largest p-reduced}
    There exists a unique, maximal (with respect to inclusion), elementary abelian, normal and $p$-reduced $p$-subgroup $Y_G$ of $G$.    If $G$ is of characteristic $p$ and $H \le G$ is $p$-parabolic, then $Y_H \le Y_G$.
\end{lemma}
\begin{proof}
    This is \cite[Lemma 2.2]{MS:2009} with $G$ in place of $L$ and $H$ in place of $P$. Note that in the paper by Meierfrankenfeld and Stellmacher a $p$-reduced subgroup is assumed to also be elementary abelian and normal.
\end{proof}

We will often call a subgroup $A \le G$ $p$-reduced for $G$ implicitly assuming that $A$ is an elementary abelian, normal $p$-subgroup of $G$, at least whenever the context leaves no ambiguity.\\

The following properties about groups of characteristic $p$ will often be used without further notice.

\begin{lemma}\label{lemma:properties groups of char p}
    Let $G$ be a group and $P \le G$ a non-trivial $p$-subgroup. Then $N_G(P)$ is of characteristic $p$ if and only if $C_G(P)$ is of characteristic $p$. Moreover, if $G$ is of characteristic $p$ and $T \in \Syl_p(G)$ then the following hold:
    \begin{enumerate}[(a)]
        \item $O_{p'}(G)=1$,
        \item $N_G(P)$ and $C_G(P)$ are of characteristic $p$ for all the non-trivial $p$-subgroups $P \le G$ ($G$ is of local characteristic $p$),
        \item every subnormal subgroup of $G$ is of characteristic $p$,
        \item $G=N_G(C_T(Y_G)) C_G(Y_G)$, $C_T(Y_G) = O_p(N_G(C_T(Y_G))$ and $Y_G=\Omega Z(C_G(Y_G))$,
        \item $Y_T = \Omega Z(T)$, $Z_G := \< \Omega Z(T)^G \>$ is $p$-reduced for $G$ and $\Omega Z(T) \le Z_G \le Y_G$.
    \end{enumerate}
\end{lemma}

Note that, whenever $G$ is of characteristic $p$ and $T \in \Syl_p(G)$, Lemmas~\ref{lemma:properties groups of char p}(e) and \ref{lemma:unique largest p-reduced} yield $1 \ne Y_T \le Y_G$.

\begin{definition}\label{def: p-irreducible, thompson and baumann subgroups}
    Recall that a finite group $G$ is \emph{$p$-closed} if $|\Syl_p(G)|=1$.
    \begin{enumerate}
        \item $G$ is \emph{$p$-irreducible} if $G$ is not $p$-closed and for every normal subgroup $N \norm G$ which is not $p$-closed one has $O^p(G) \le N$.
        \item $G$ is \emph{strongly $p$-irreducible} if $G$ is not $p$-closed and for every normal subgroup $N \norm G$ such that $[N,G] \not \le O_p(G)$ one has $O^p(G) \le N$.
        \item $Z_G := \< \Omega Z(T) \mid T \in \Syl_p(G) \>$.
        \item Set $\A_G := \{ \textit{elementary abelian $p$-subgroups of $G$ of maximal order} \}$.\\
        The \emph{Thompson subgroup} of $G$ is defined as 
        \[
        J(G) := \< \A_G \>;
        \]
        the \emph{Baumann subgroup} of $G$ is defined by
        \[
        B(G) := \begin{cases}
                    C_G(\Omega Z(J(G)))     & \textit{if $G$ is a $p$-group,} \\
                    \<B(T) \mid T \in \Syl_p(G) \>  & \textit{else.}
                \end{cases}
        \]
        \item Let $Y \le G$. $G$ is \emph{$Y$-minimal} if:
    \begin{itemize}
        \item $G=\< Y^G \>$,
        \item there exists a unique maximal subgroup of $G$ containing $Y$.
    \end{itemize}
        We say that $G$ is $p$-minimal if it is $Y$-minimal for $Y \in \Syl_p(G)$.
    \end{enumerate}
\end{definition}

\vspace{0.4cm}
\begin{definition}
    Let $H \le K \le G$. We say that $H$ is \emph{weakly closed} in $K$ with respect to $G$ if
    \[
    H^G \cap \Sub(K) = \{H\}.
    \]
    Set $ \wcl_G(H,K) := \< H^g \mid g \in G, H^g \le K \>$, called the \emph{weak closure} of $H$ in $K$ with respect to $G$. Clearly $H$ is weakly closed in $K$ if and only if $H=\wcl_G(H,K)$.\\
    If $H \le K \in \Syl_p(G)$, we simply refer to $\wcl_G(H,K)$ as the weak closure of $H$ with respect to $G$ (and the prime $p$).
\end{definition}

\vspace{0.4cm}
We have already given the definition of a large subgroup \ref{def: large subgroup of a group} in the Introduction.\\
Note that, if $G$ is a finite simple group of Lie type in defining characteristic $p$ with $S \in \Syl_p(G)$, then $O_p(N_G(\Omega(Z(S)))$ is large in $G$ if and only if $\Omega(Z(S))$ is a root subgroup, which fails only when $G$ is one of $\Sp_{2n}(2^m)$, $n \ge 2$, or $F_4(2^m)$, or $G_2(3^m)$.\\

We collect here some important results about large subgroups.

\begin{prop}\label{prop:properties of large in a group}
    Let $G$ be a group and $Q$ a large $p$-subgroup contained in $S \in \Syl_p(G)$. Then the following hold:
    \begin{enumerate}[(a)]
        \item $N_G(Q)$ is $p$-parabolic and $O_p(N_G(Q))$ is large in $G$;
        \item if $P \le S$ is another large $p$-subgroup of $G$, then $PQ$ is large in $G$; in particular, there exists a unique largest large subgroup of $G$ in each Sylow $p$-subgroup;
        \item $Q$ is weakly closed with respect to $G$, hence $Q \norm S$;
        \item $Z(Q) \cap Z(Q)^g = 1$ for any $g \in G \setminus N_G(Q)$;
        \item $G$ has parabolic characteristic $p$.
    \end{enumerate}
\end{prop}
\begin{proof}
    Since $1 \ne Z(S) \le C_G(Q)=Z(Q)$ we get $S \le N_G(Z(S)) \le N_G(Q)$. Now (a), (c) and (d) follow from \cite[Lemma 1.52(b) and (e)]{Meier-Stell-Stroth:2012}.\\
    Point (e) is \cite[Lemma 1.55]{Meier-Stell-Stroth:2012}, whereas (b) is \cite[Lemma 4.2(c)]{Henke-Salati:2026} after noticing that $P \le S$ implies that $P$ normalizes $Q$.
\end{proof}

Note that there are groups which possess a large $p$-subgroup even though not being of local characteristic $p$; for example, from \cite[Table 1 in Chapter 7]{Meier-Stell-Stroth:2012} we see that $G=\Alt(10)$ contains a natural $\Omega_4^-(2)$-module, which does not satisfy the condition $(char \, Y_M)$, i.e. it contains a non-trivial element $y$ such that $C_G(y)$ is not of characteristic $2$.

\vspace{0.5cm}
Finally, we need to define and provide some properties concerning offenders.\\
Let $\alpha$ be a positive, real number and $H$ be a group acting on the group $V$; define the map
\begin{center}
	\begin{tikzcd}
		\nu_\alpha: &[-25pt] \Sub(H) \ar[r] & \bN \\ [-20pt]
			&	L \ar[r,mapsto] & \null |L|^\alpha \cdot |C_V(L)|
	\end{tikzcd}
\end{center}

\begin{definition}
	\label{def:offenders}
Following the notation above, a subgroup $A \le H$ is an \emph{$\alpha F$-offender} on $V$ if $\nu_\alpha(A) \ge \nu_\alpha(C_A(V))$. The $\alpha F$-offender $A$ in \emph{non-trivial} if $[A,V] \ne 1$. If $\alpha =1$ we simply call $A$ an offender and write $\nu$ for $\nu_1$.\\
Assume now that $\alpha = 1$, that $V$ is an $\bF_p H$-module and $A \le H$ is such that $A/C_A(V)$ is an elementary abelian $p$-group.
\begin{itemize}
	\item If $H$ contains some non-trivial offender on $V$, then $V$ is called an $FF$-module.
	\item $A$ is an \emph{over-offender} on $V$ if $\nu(A) \gneq \nu(C_A(V))$.
	\item $A$ is a \emph{best offender} on $V$ if $\nu |_{ \Sub(A) }$ is monotone.
	\item $A$ is a \emph{strong offender} on $V$ if it is an offender and $C_V(A) = C_V(a)$ for all $a \in A \setminus C_A(V)$.
	\item $A$ is a \emph{root offender} on $V$ if it is a strong offender and also $[V,A]=[V,a]$ for all $a \in A \setminus C_A(V)$.
	\item $A$ is a \emph{strong dual offender} on $V$ if $A$ acts nilpotently on $V$ and $[V,A]=[v,A]$ for all $v \in V \setminus C_V(A)$.
\end{itemize}
Set $J_H(V) := \< A \le H \mid \textit{$A$ is a best offender on $V$} \>$, a normal subgroup of $H$.\\
A subgroup $1 \ne K \le J_H(V)$ such that $K \not \le C_H(V)$ and it is minimal satisfying $[K,J_H(V)]=K$ is called a \emph{$J_H(V)$-component}. Let $\J_H(V)$ to be the set of all $J_H(V)$-components and $J^*_H(V) := \< K \le H \mid K \in \J_H(V) \>$, also a normal subgroup of $H$.
\end{definition}
This definition matches the more general one given in \cite[Section 1]{MPS:2018} for $\beta=1$ and with $\nu_\alpha$ corresponding to the measure $\lVert - \rVert_V$ as well as \cite[Definition A.7]{Meier-Stell-Stroth:2012} (which is our standard and we have here expanded to encompass the case of $2F$-offenders).\\
Note that \cite[Lemma 1.3]{MS:2012a} and \cite[9.2.4]{Kurzweil/Stellmacher:2004} together give that the minimal offenders $B$ on $V$ that satisfy $[B,V]\ne 1$ are non-trivial, quadratic, best offenders on $V$. It is therefore sufficient to detect an offender $A$ to also find a quadratic, best offender and one needs only to check that it acts non-trivially. We refer to \cite[Sections 1 and 2]{MS:2012a} for more properties of offenders and of the groups $J_H(V)$ and $J^*_H(V)$. 

\subsection{Localities}
For the convenience of the reader and for future reference we summarize the main properties of localities.

The following definition was originally developed by Chermak in \cite{Chermak:2013}. We start with a piece of notation: if $X$ is a set, we denote by $\W(X)$ the free monoid on $X$, hence with elements words (of finite lengths) on the alphabet $X$ and concatenation as operation; here the unit of the monoid is the empty word $(\,)$. We denote the concatenation of two words $u,v \in \W(X)$ by $u \circ v$.

\begin{definition}
	\label{def:partial group}
    A \emph{partial group} consists in the data of a quadruple $(\L,D,\Pi,i)$ where:
    \begin{itemize}
        \item $\L$ is a non-empty set,
        \item $\L \subseteq D \subseteq \W(\L)$, where the first inclusion is realized by embedding in words of length $1$,
        \item $\Pi: D \fto \L$, the \emph{multivariable product} of the partial group,
        \item $i: \L \fto \L$ is a bijection such that $i \circ i = \id_\L$,
    \end{itemize}
    satisfying the following axioms:
    \begin{enumerate}[(i)]
        \item $\Pi |_\L = \id_\L$,
        \item if $u \circ v \in D$, then $u,v \in D$ (in particular, the empty word is always in $D$),
        \item if $u \circ v \circ w \in D$, then $u \circ \Pi(v) \circ w \in D$ and $\Pi(u \circ v \circ w) = \Pi(u \circ \Pi(v) \circ w)$ (associativity axiom),
        \item if $u=(u_1, \dots, u_k) \in D$, set $i(u):= (i(u_k), \dots, i(u_1))$; then $u \circ i(u) \in D$ and $\Pi(u \circ i(u)) = \Pi((\,)) =: 1_\L$.
    \end{enumerate}
\end{definition}

Most often we will denote the inversion map $i$ by $(-)\inv$; moreover, we will often denote a partial group simply by specifying the underling set, unless otherwise needed.\\
Note that all groups are partial groups: indeed, a partial group $(\L,D,\Pi,i)$ is a group if and only if $D=\W(\L)$.\\

When it comes to substructures, the following notions are important to us.
\begin{itemize}
\item Given a non-empty subset $\H$ of a partial group $\L=(\L,D,\Pi,(-)\inv)$ such that $h\inv \in \H$ for any $h \in \H$ and $\Pi(h_1,\dots,h_k) \in \H$ for every $(h_1,\dots,h_k) \in \W(\H)\cap D$, then the quadruple $(\H, \W(\H) \cap D, \Pi|_{\W(\H) \cap D},(-)\inv)$ is called a \emph{partial subgroup} of $\L$.\\
\item If $\N=(\N,D,\Pi,i)$ is a partial group and $\H, \K$ are subsets of $\N$, we define
\[
\H \K := \{ \Pi(h,k) \mid h \in \H, k \in \K \textit{ and } (h,k) \in D \};
\]
note that $\H \K$ needs not be a partial subgroup of $\L$ even when $\H$ and $\K$ are so.
\item A \emph{partial normal subgroup} of $\L$ is a partial subgroup $\N$ of $\L$ such that for any $n \in \N$ and $x \in \L$, if $(x\inv,n,x) \in D$, then $n^x := \Pi(x\inv, n,x) \in \N$.
\end{itemize}

In a partial group $\L$ there is a sensible notion of conjugation: if $g \in \L$ we set
\[
D(g) := \{x \in \L \mid (g\inv, x,g) \in D \}.
\]
A word of caution: a subset $U$ of a partial group $\L$ might itself be a group (with respect to the product of $\L$), hence we call it a subgroup of $\L$. It might also happen that $U \subseteq D(g)$ for some $g \in G$, but this fact doesn't imply that the well-defined set-map given by conjugation $c_g: U \fto U^g$ is a group homomorphism. We will see that such a feature is partially retrieved in the context of localities. \\
For a subset $U \subseteq \L$, by writing $U^g$ we will always implicitly assume that $U \subseteq D(g)$; if $U$ is a subgroup, we will additionally assume that the conjugation map is a homomorphism.\\

The following definition is easily seen to be equivalent to \cite[Definition 2.9]{Chermak:2013}.
\begin{definition}
    A (finite) locality over the $p$-group $S$ consists of the data of a triple $(\L,\Delta,S)$ where $\L=(\L,D,\Pi,i)$ is a finite partial group, $S$ is a subgroup of $\L$ and $\Delta$ is a non-empty set of subgroups of $S$ satisfying:
    \begin{enumerate}[(i)]
        \item $S$ is a maximal $p$-subgroup of $\L$ with respect to inclusion,
        \item a word $u=(u_1, \dots, u_k) \in \W(\L)$ is in $D$ if and only if there exist subgroups $P_0, \dots, P_k \in \Delta$ such that 
        \[
        P_{i-1}^{u_i} = P_i \quad \forall i \in \{1, \dots, k\},
        \]
        \item $\Delta$ is closed under taking $\L$-conjugates and overgroups.
    \end{enumerate}
    We call $\Delta$ the \emph{set of objects} of $\L$.
\end{definition}

In a locality the following properties hold.

\begin{lemma}\label{lemma:properties of localities}
    Let $(\L,\Delta,S)$ be a locality with domain $D$. Then:
    \begin{enumerate}[(a)]
        \item for every $P \in \Delta$, the set $N_\L(P):= \{ x \in \L \mid P^x = P\}$ is a subgroup of $\L$;
        \item if $P \in \Delta$ and $x \in \L$ is such that $P^x \subseteq S$, then the induced conjugation map $c_x: P \fto P^x$ extends to a homomorphism of groups
        \[
        c_x : N_\L(P) \fto N_\L(P)^x = N_\L(P^x);
        \]
        \item every subgroup $H$ of $\L$ is contained in some $N_\L(P)$ for some $P \in \Delta$;
        \item every $p$-subgroup of $\L$ is an $\L$-conjugate of some subgroup of $S$,
        \item for every $x \in \L$ the set $S_x:= \{ s \in S \mid s^x \in S \}$ is a subgroup of $S$ contained in $\Delta$;
        \item for any word $w:=(w_1,\dots,w_k) \in \W(\L)$ set $S_w := \{s_0 \in S \mid s_i := s_{i-1}^{w_i} \in S \; \forall i \in \{1,\dots,n\} \}$, then $S_w$ is a subgroup of $S$ and $S_w \in \Delta$ if and only if $w \in D$.
    \end{enumerate}
\end{lemma}
\begin{proof}
    Points (a) and (b) are proved in \cite[Lemma 2.7]{Chermak:2013}; the argument is provided for the more general setting of \emph{objective} partial groups, to which all localities belong. In particular, in (b) the conditions $P \in \Delta$ and $P^x \subseteq S$ imply $P^x \in \Delta$, since $\Delta$ contains $S$ and is closed under $\L$-conjugations.\\
    (c) and (d) are \cite[Proposition 2.22]{Chermak:2013}.\\
    (e) and (f) are \cite[Lemma 2.14]{Chermak:2013}; note that (e) is a special case of (f), obtained by taking $w$ of length $1$.
\end{proof}

In analogy with the terminology for groups we give the following definition.
\begin{definition}
	\label{def:local and parabolic of localities}
	Let $(\L,\Delta,S)$ be a locality.
	\begin{itemize}
		\item The groups of the form $N_\L(P)$ for some $P \in \Delta$ are the \emph{local subgroups} of $\L$.
		\item The subgroups of $\L$ that contain $S$ are the \emph{parabolic subgroups} of $\L$.
	\end{itemize}
\end{definition}

\vspace{0.4cm}
In \cite[Lemma 2.13]{Chermak:2022:I} Chermak also defines an analogue of the $p$-radical $O_p(G)$ of a group for localities; we will need here the following easy generalization of his idea.

	\begin{definition}\label{def:O-p of a partial subgr of local}
		Let $(\L, \Delta, S)$ be a locality and $\M \subseteq \L$ a partial subgroup of $\L$. Set
		\[
		\coll_p(\M) := \bigcap_{w \in W(\M)} S_w \qquad \text{and} \qquad \hat{\coll}_p(\M) := \bigcap_{w \in W(\M)} (S_w \cap \M)
		\]
	\end{definition}
	
	\begin{lemma}\label{lemma:Op-of-part-subgr}
		Let $\M \subseteq \L$ be a partial subgroup of $\L$. Then 
		\begin{itemize}
			\item $\coll_p(\M)$ is the unique largest subgroup of $S$ normalized by $\M$;
			\item $\hat{\coll}_p(\M) = \coll_p(\M) \cap \M$ and it is the unique largest subgroup of $S \cap \M$ normal in $\M$.
		\end{itemize}
	In particular, $\displaystyle \coll_p(\L) = \bigcap_{w \in W(\L)} S_w$ is the largest subgroup of $S$ normal in $\L$.
	\end{lemma}
	\begin{proof}
		By \ref{lemma:properties of localities}(f) each $S_w$ is a subgroup of $S$, thus also $\coll_p(\M)$ is a subgroup of $S$.\\
		For any element $m \in \M$ and word $w \in \W(\M)$, $\coll_p(\M) \le S_m \cap S_{(m) \circ w}$. This means that also $\coll_p(\M)^m \le S_w$. As this holds for all $w \in \W(\M)$ and every $m \in \M$, we get $\coll_p(\M)$ is normalized by $\M$.\\
		Suppose now that $X \le S$ is normalized by $\M$. Since for every $m \in \M$, $N_\L(S_m)^m$ is well defined by Lemma~\ref{lemma:properties of localities}(e) we get that $N_X(S_m)^m$ is well-defined and, since $X$ is normalized by $\M$, it is contained in $X$. This implies $N_X(S_m)^m \le S$, so $N_X(S_m) \le S_m$; in particular $N_{\<S_m, X\>}(S_m) \le S_m$. Since $\<S_m,X\>$ is a subgroup of $S$ and, therefore, it is a $p$-group, we get $\<S_m,X \>= S_m$, i.e. $X \le S_m$. Hence $X$ is conjugable by $m$ and, being normalized by $\M$, $X=X^m$; then $X \le S_w$ for every $w \in \W(\M)$, hence $X \le \coll_p(\M)$.\\
		
		Since $\M$ is a partial subgroup, $D(\M) = \W(\M) \cap D(\L)$ and therefore each $S_w \cap \M$ is a subgroup of $S$.  Then we get
		\[
			\hat{\coll}_p(\M) \le \coll_p(\M) \cap \M = \left( \bigcap_{w \in W(\M)} S_w \right) \cap \M = \bigcap_{w \in W(\M)} (S_w \cap \M) = \hat{\coll}_p(\M)
		\]
		If now $X \le S \cap \M$ is normal in $\M$, then $X \le \coll_p(\M)$ and $X \le \M$ by definition, hence $X \le \hat{\coll}_p(\M)$ and all is proven.
	\end{proof}	

\begin{remark}
	The subgroups $\coll_p(\M)$ and $\hat{\coll}_p(\M)$ capture something analogous to the $p$-radical $O_p(G)$ of a group $G$, but there are some important differences.\\
	For example, assume that $\M \subseteq \L$ is a partial subgroup of the locality $(\L,\Delta,S)$ which is also a group; then clearly $\hat{\coll}_p(\M) \le  O_p(\M)$. The relationship between $\coll_p(\M)$ and $O_p(\M)$ is, however, less trivial; for example, a similar inclusion does not hold if $S \cap \M = 1$ and both $\coll_p(\M)$ and $O_p(\M)$ are non-trivial. It is nonetheless true that $\hat{\coll}_p(\M) = O_p(\M) \le \coll_p(\M)$ whenever $\M \cap S$ is a Sylow $p$-subgroup of $\M$, with all equalities if $S \le \M$ holds.
\end{remark}
The next result may be helpful in understanding how $\coll_p(\M)$ fits within $\M$ and $\L$.

\begin{lemma}
	Let $\M \subseteq \L$ be a partial subgroup of $\L$. Then $\M$ is a group if and only if $\coll_p(M) \in \Delta$.
\end{lemma}
\begin{proof}
	Assume first that $\M$ is a group; then by \ref{lemma:properties of localities}(a) and (c) $\M \le N_\L(P)$ for some $P \in \Delta$, where $N_\L(P)$ is a group; thus $\M$ normalizes $P$, so $P \le \coll_p(\M)$, showing that $\coll_p(\M) \in \Delta$.\\
	Assume now that $\coll_p(\M) \in \Delta$; then always \ref{lemma:properties of localities} shows that $N_\L(\coll_p(\M))$ is a group. Since $\M \subseteq N_\L(O_p(M))$ and it is a partial subgroup of $\L$, also $\M$ is a group.
\end{proof}

\vspace{0.6cm}
Given a locality $(\L,\Delta,S)$ we associate to $\L$ a (not necessarily saturated) fusion system $\F_S(\L)$ over $S$, namely the one generated by the maps $c_g: S_g \fto S$ for all $g \in \L$. We say that $\L$ is a locality over $\F$, or that $\F$ is the fusion system of $\L$, if $\F \cong \F_S(\L)$.\\
The main goal of the paper \cite{Chermak:2013} by Chermak was to reverse the statement above, proving that for every saturated fusion system $\F$ there exists a locality (unique under certain additional conditions) over $\F$.\\

The following definition, originally given by Henke, is crucial for further results.
\begin{definition}\label{def:subcentric}
Let $\F$ be a fusion system on the $p$-group $S$. A subgroup $P \le S$ is \emph{subcentric} if, for every fully normalized $\F$-conjugate $Q$ of $P$, $O_p(N_\F(Q))$ is centric in $\F$. The set of all $\F$-subcentric subgoups of $S$ is denoted by $\F^s$.
\end{definition}

We also recall the following important notions.
\begin{definition}\label{def:linking localities}
    Let $(\L,\Delta,S)$ be a locality over the fusion system $\F$.
    \begin{itemize}
        \item $\L$ is of \emph{objective characteristic $p$} if for every $P \in \Delta$ the group $N_\L(P)$ is of characteristic $p$.
       \item $\L$ is a \emph{linking locality} if $\F_S(\L)$ is saturated, $\F^{cr} \subseteq \Delta$ and $\L$ is of objective characteristic $p$.
       \item $\L$ is a \emph{centric linking locality} (resp., \emph{subcentric linking locality}) over $\F$ if $\L$ is a linking locality over $\F$ and $\Delta = \F^c$ (resp., $\Delta = \F^s$).
    \end{itemize}
\end{definition}

\begin{lemma}\label{lemma:parabolic implies char p}
    Let $(\L,\Delta,S)$ be a locality of objective characteristic $p$. Then every parabolic subgroup of $\L$ has characteristic $p$.
\end{lemma}
\begin{proof}
    Let $S \le H \le \L$ be a parabolic subgroup; by \ref{lemma:properties of localities}(c) there exists $P \in \Delta$ such that $H \le N_\L(P)$. Set $N:=N_\L(P)$ and note that, $H$ being parabolic, $O_p(N) \le O_p(H)$; since $\L$ is of objective characteristic $p$ we obtain
    \[
        C_H(O_p(H)) \le C_N(O_p(N)) \le O_p(N) \le O_p(H).
    \]
\end{proof}

We recall now some fundamental results on localities over a given fusion system $\F$.
\begin{prop}
    Let $(\L,\Delta,S)$ be a locality over a (non necessarily saturated) fusion system $\F$. If $\F^c \subseteq \Delta$, then $\F$ is saturated.   Conversely, assume that $\F$ is a saturated fusion system on the $p$-groups $S$. Then:
    \begin{enumerate}[(a)]
        \item there exists a unique (up to rigid isomorphism) centric linking locality over $\F$;
        \item for every set $\Gamma$ of subgroups of $S$ which is closed under $\L$-conjugation and overgroups and such that $\F^c \subseteq \Gamma \subseteq \F^s$ there exists a unique (up to rigid isomorphism) linking locality over $\F$ having $\Gamma$ as set of objects.
    \end{enumerate}
\end{prop}
See \cite[Definition 3.5]{Chermak:2013} for the definition of rigid isomorphisms of localities.
\begin{proof}
    The first statement is \cite[Proposition 2.18(a)]{Chermak:2013}.\\
    Now (a) is the Main Theorem in \cite{Chermak:2013}; note that the Appendix A in the same paper shows how to obtain a one-to-one correspondence between the centric linking systems in the Main Theorem and centric linking localities according to \ref{def:linking localities}.\\
    Point (b) is instead shown in \cite[Theorem A]{Henke:2015}.
\end{proof}

As a consequence, since $\F^{cr} \subseteq \F^c$, one easily sees that a locality $(\L,\Delta,S)$ over $\F$, of objective characteristic $p$, is a linking locality whenever the condition $\F^c \subseteq \Delta$ is satisfied.\\

Finally we can make more precise the dichotomy for fusion systems that was already described in the Introduction. We need a preliminary definition that extends \cite[Definition 14.1]{Aschbacher/Kessar/Oliver:2011}, along ideas of Ellen  Henke, in order to suit some of our future needs. Recall that there is a theory of components and the general Fitting subsystem of fusion systems in analogy to the corresponding theory for finite groups, and a corresponding theory for localities. References may be found in \cite[Chapter II.9]{Aschbacher/Kessar/Oliver:2011} and \cite{Henke:2025}.

\begin{definition}\label{def:fusion systesms of char p}
    Let $\F$ be a saturated fusion system on the $p$-group $S$.
    \begin{itemize}
        \item $\F$ is {of characteristic $p$-type} if $N_\F(U)$ is constrained for every $1 \ne U \in \F^f$.
        \item $\F$ is {of parabolic characteristic $p$} if $N_\F(U)$ is constrained for every $1 \ne U \norm S$.
        \item $\F$ is {of component type} if there exists $X \in \F^f$, $X$ of order $p$, such that $C_\F(X)$ has a component.
    \end{itemize}
\end{definition}

\begin{theorem}[Dichotomy theorem for saturated fusion systems] \label{thm:dichotomy for fus syst}
Let $\F$ be a saturated fusion system on the $p$-group $S$. Then either $\F$ is of characteristic $p$-type, or $\F$ is of component type.
\end{theorem}
\begin{proof}
    This is \cite[Theorem 14.3]{Aschbacher/Kessar/Oliver:2011}.
\end{proof}

The following result due to Henke gives us a way to relate the properties of fusion systems in \ref{def:fusion systesms of char p} to corresponding properties of the associated linking systems.

\begin{lemma}\label{lemma:equivalence subcentric sbgrps}
    Let $\F$ be a saturated fusion system on $S$. Then for every $Q \le S$ the following are equivalent:
    \begin{itemize}
        \item[(a1)] $Q$ is subcentric in $\F$,
        \item[(a2)] for some $P \in \F^f \cap Q^\F$, $P$ is subcentric in $\F$,
        \item[(b1)] for every $P \in \F^f \cap Q^\F$, $N_\F(P)$ is constrained,
        \item[(b2)] for some $P \in \F^f \cap Q^\F$, $N_\F(P)$ is constrained,
        \item[(c1)]  for every $P \in \F^c \cap Q^\F$, $C_\F(P)$ is constrained,
        \item[(c2)] for some $P \in \F^c \cap Q^\F$, $C_\F(P)$ is constrained.        
    \end{itemize}
\end{lemma}
\begin{proof}
    This is \cite[Lemma 3.1]{Henke:2015}
\end{proof}

As a consequence we obtain the following straightforward result.

\begin{cor}
	\label{cor:locality transl of char p type}
    Let $\F$ be a saturated fusion system on $S$ and $\L=(\L,\Delta,S)$ be the subcentric linking locality over $\F$. Then
    \begin{itemize}
        \item[(a)] $\F$ is of parabolic characteristic $p$ if and only if $\Delta$ contains all the non-trivial, normal subgroups of $S$;
        \item[(b)] $\F$ is of characteristic $p$-type if and only if $\Delta$ is the set of all non-trivial subgroups of $S$.
    \end{itemize}
\end{cor}
\begin{proof}
    Point (a) is a clear consequence of \ref{lemma:equivalence subcentric sbgrps}(b1) and (b2).\\
    Point (b) follows easily from the fact that $\Delta$ is closed under $\L$-conjugation, hence under $\F$-conjugation.
\end{proof}

\section{Fusion systems and localities with a large subgroup}
\label{large subgroups}
This section is devoted to providing the fundamental results necessary for pursuing our goal. In particular, here we set the stage for the entire future analysis.\\
The fundamental results connecting groups, fusion systems and localities with a large $p$-subgroup were obtained in \cite[Section 4]{Henke-Salati:2026}; here we report what will be needed for our specific goal and focus thereafter on properties of weakly closed subgroups. Localities introduce an extra layer of conjugation, which remains invisible when dealing with groups and fusion systems. In a fusion system $\F$ over $S$ there exist only conjugations within $S$; in a group $G$ any conjugation between $p$-subgroups of $G$ lifts to a conjugation between Sylow $p$-subgroups. In a locality $\L$ over the $p$-group $S$, instead, there are conjugations within subgroups of $S$ as well as conjugation carrying subgroups of $S$ out of $S$, (or viceversa, or also conjugating $p$-subgroups not contained in $S$) and these may have no well-defined lift to maximal $p$-subgroups containing them. This lifting problem will be a major concern in many situations and naturally appears when working with weakly closed subgroups.\\

The definition of a large $p$-subgroup of a group extends in a very natural way to fusion systems and localities; these are \cite[Definition 1.1 and 1.2]{Henke-Salati:2026} and are reported below.

\begin{definition}
    Let $\F$ be a fusion systems on $S$ and $Q \le S$. The subgroup $Q$ is \emph{large} in $\F$ if:
    \begin{itemize}
        \item[$(SC\F)$] $C_S(Q) \le Q$;
        \item[($Q!\F)$] $N_\F(U) \subseteq N_\F(Q)$ for all non-trivial subgroups $U \le Z(Q)$.
    \end{itemize}
\end{definition}

\begin{definition}
    Let $(\L,\Delta,S)$ be a locality and $Q \le S$. $Q$ is large in $\L$ if:
    \begin{itemize}
        \item[$(SC\L)$] $C_\L(Q) \subseteq Q$;
        \item[$(Q!\L)$] $N_\L(U) \subseteq N_\L(Q)$ for all non-trivial subgroups $U \le Z(Q)$.
    \end{itemize}
\end{definition}

The next lemma is a locality version of \ref{prop:properties of large in a group}.

\begin{lemma}\label{lemma:properties of large in fus system}
    Let $\F$ be a fusion system on $S$ and $Q \le S$ be large in $\F$. Then:
    \begin{enumerate}[(a)]
      \item $O_p(N_\F(Q)$ is large in $\F$ and $Q \norm S$;
      \item if $R \le S$ is large in $\F$, then $RQ$ is large in $\F$;
      \item $Q \norm N_\F(U)$ for all $1 \ne U \le Z(Q)$.
    \end{enumerate}
    If, in addition, $\F$ is saturated, then:
    \begin{enumerate}
        \item[(d)] $Q$ is weakly closed, thus $Q \in \F^c$;
        \item[(e)] $P \in \F^s$ for every $1 \ne P \le S$ such that $Q \le N_S(P)$; in particular, $\F$ is of parabolic characteristic $p$.
    \end{enumerate}
\end{lemma}
\begin{proof}
    Points (a), (b) and (c) are \cite[Lemma 4.3]{Henke-Salati:2026}; (d) and (e) are instead \cite[Lemma 4.5]{Henke-Salati:2026}.
\end{proof}

\vspace{0.4cm}
\begin{definition}
    Let $(\L,\Delta,S)$ be a locality. A subgroup $Q \le S$ is \emph{weakly closed} (in $S$) with respect to $\L$ if for every $g \in \L$ the following implication holds:
    \[
    Q^g \le S \implies Q^g =Q.
    \]
    If $K \le \L$ is a subgroup containing $Q$ we also set
    \[
        \wcl_\L(Q,K) := \< Q^g \mid g \in \L, \, Q^g \le K \>,
    \]
    where, as usual, we tacitly assume that $c_g: Q \fto Q^g$ is a group homomorphism.
\end{definition}
A word of caution is needed here. In the locality $(\L,\Delta,S)$ there might exist elements $g,h \in \L$ such that, for $Q \le S$, $Q \to Q^g$ and $Q^g \to (Q^g)^h$ are well-defined group homomorphisms, $(Q^g)^h \le S$ and $Q^g \not \le S$. There is nothing forcing $(g,h) \in D$ (the domain of $\L$), hence it might be the case that $Q$ is weakly closed in $\L$, but $Q \ne (Q^g)^h$ (this may clearly also happen over longer chains of conjugations). Controlling this situation is one of the major differences between working with a large subgroup of a group and working in a locality.\\
The next lemma, when compared with \cite[Lemma 1.46]{Meier-Stell-Stroth:2012}, already provides an example of the analogies and differences stemming from the situation described above.

\begin{notation} 
Let $Q \le S \le K \le \L$ be subgroups of a locality $(\L,\Delta,S)$ with $Q$ weakly closed in $\L$. Define
\[
K^\circ := \< Q^K \>, \qquad K^\flat := \wcl_\L(Q,K) \qquad \text{and} \qquad K_\circ := O^p(K^\circ).
\]
\end{notation}

\begin{lemma}[1.46]\label{lemma:weakly closed subgroups in localities}
	Let $Q \le S$ be a weakly closed subgroup of a locality $(\L,\Delta,S)$ and $Q \le K$ for a subgroup $K$ of $\L$. Then the following hold.
	\begin{enumerate}[(a)]
		\item $Q$ is a weakly closed subgroup of $K$ and, therefore, of any subgroup of $K$ containing $Q$.
		\item $K^\circ \norm K^\flat \le K$. Moreover, if for every $g \in \L$ such that $c_g : Q \fto Q^g$ is a group isomorphism with $Q^g \le K$ there exists $k \in K$ such that $Q^g=Q^k$, then $K^\circ = K^\flat$.
		\item $Q^{K^\circ} = Q^{K^\flat} = Q^{K}$, so $\langle Q^{K^\circ} \rangle = \langle Q^{K^\flat} \rangle = \langle Q^K \rangle = K^\circ$.
		\item $K^\circ$ is the subnormal closure of $Q$ in $K$, hence $K^\circ = K_\circ Q = \langle Q^{O^p(K)} \rangle = \< Q^{K_\circ} \>$.
		\item $K_\circ = [K_\circ,Q]$. 
		\item Suppose that $K \le H$ for a subgroup $H$ of $\L$. Then
		\[
		K^\circ \norm \norm H \iff K^\circ = H^\circ \iff Q^K =Q^H \iff H=KN_H(Q) \iff H \subseteq KN_\L(Q).
		\]
	\end{enumerate}
\end{lemma}
\begin{proof}
	\begin{enumerate}[(a)]
		\item If $\L$ is a group, the result holds by \cite[Lemma 1.46(a)]{Meier-Stell-Stroth:2012}. Now by \cite[Lemma 2.10(a)]{Chermak:2022:I} there exists $P \in \Delta$ such that $K \le N_\L(P)$, thus we may replace $K$ by $N_\L(P)$. Suppose that $Q \le T \in \Syl_p(N_\L(P))$ and $x \in N_\L(P)$ is such that $Q^x \le T$. By \cite[Lemma 2.9]{Chermak:2022:I} there is an element $g \in \L$ such that $P \le S_g$ and $N_S(P^g) \in \Syl_p(N_\L(P^g))$, so that one has 
		\[
		Q^g, (Q^x)^g \le T^g \le N_\L(P^g).
		\]
		Moreover, $(x,g) \in D(\L)$ via $(P,P,P^g)$, so we may write $(Q^x)^g=Q^{xg}$; let $h \in N_\L(P^g)$ be such that $(T^g)^h = N_S(P^g)$, so also $(x,g,h) \in D(\L)$ via $(P,P,P^g,P^g)$ and one has $Q^{gh}, Q^{xgh} \in N_S(P^g) \le S$. Since $Q$ is weakly closed in $\L$, $Q^{gh} =Q^{xgh}$, proving that $Q=Q^x$.
		\item The inclusions in $K^\circ \norm K^\flat \le K$ are obvious; clearly $K^\circ \norm K$ and every $K$-conjugate of $Q$ is an $\L$-conjugate of $Q$ that stays in $K$. The second statement is obvious.
		\item By (b) we get $Q^{K^\circ} \le Q^{K^\flat} \le Q^K$; since $Q$ is weakly closed in $K$ by (a), \cite[Lemma 1.46(c)]{Meier-Stell-Stroth:2012} applied with $K=G$ shows that $Q^{K^\circ} = Q^K$, proving all the equalities.
		\item If $K_0$ is the subnormal closure of $Q$ in $K$, then by \cite[Lemma 1.13(a)]{Meier-Stell-Stroth:2012} $K_0 = \< Q^{K_0} \> \le \< Q^{K} \> = K^\circ \norm K$.\\
		Conversely, let $Q \le P \norm \norm K$ with $P \norm P_1 \norm \dots \norm P_{n-1} \norm P_n =K$, then by (a) we can apply \cite[Lemma 1.46(c)]{Meier-Stell-Stroth:2012} with $(P_j,K)$ in place of $(K,H)$, descending the normal chain. This shows that $Q^K = Q^{P_{n-1}} = \dots = Q^{P_1} = Q^P$, whence, by (c) $\< Q^{K^\circ} \> = \< Q^K \> = \< Q^P \> \le P$; we conclude that $K^\circ = \< Q^{K^\circ} \> =K_0$.\\
		Now $K^\circ=K_\circ Q = \< Q^{K_\circ} \>$ is \cite[Lemma 1.13(b)]{Meier-Stell-Stroth:2012} with $(Q,K^\circ)$ in place of $(A,K)$. Since $K_\circ \le O^p(K)$ we have
		\[
		K^\circ = \< Q^{K_\circ} \> \le \< Q^{O^p(K)} \> \le \< Q^K \> = K^\circ,
		\]
		showing the last equality.
		\item By (d) one has $K^\circ = \< Q^{K_\circ} \> = [K_\circ,Q]Q$ so that, since $Q$ is a $p$-group and $[K_\circ,Q] \norm K$ by the definition of $K^\circ$, $K_\circ = O^p(K^\circ) \le [K_\circ,Q]$. Thus $K_\circ = [K_\circ,Q]$.
		\item $Q$ is also a weakly closed $p$-subgroup of $H$, hence the equivalences between the first four statements follow from \cite[Lemma 1.46(f)]{Meier-Stell-Stroth:2012}. For the last equivalence, clearly $H=KN_H(Q)$ implies $H \subseteq KN_\L(Q)$; suppose now that $H \subseteq KN_\L(Q)$ holds, then
		\[
		H \subseteq KN_\L(Q) \cap KH = K(N_\L(Q) \cap H) = KN_H(Q) \subseteq H,
		\]
		proving the desired equality.
	\end{enumerate}
\end{proof}

\begin{lemma}\label{lemma:large implies weakly closed}
    Let $(\L,\Delta,S)$ be a locality over $\F$ and $Q \le S$ be a large subgroup of $\L$. Then:
    \begin{enumerate}[(a)]
        \item $Q$ is weakly closed in $\L$;
        \item $Q \norm S$ and $O_p(N_\L(Q))$ is well-defined and large in $\L$;
        \item if $R \le S$ is large in $\L$, then $RQ$ is large in $\L$.
    \end{enumerate}
\end{lemma}
\begin{proof}
    This is \cite[Lemma 4.6]{Henke-Salati:2026}.
\end{proof}

\begin{definition}
    Let $(\L,\Delta,S)$ be a locality and $Q \norm S$. $\L$ is \emph{$Q$-replete} if every non-trivial subgroup of $C_S(Q)$ is in $\Delta$; we call $Q$ a \emph{repleteness subgroup} for $\L$.
\end{definition}
In \cite[Lemma 4.7]{Henke-Salati:2026} it is shown to be equivalent to the requirement that every non-trivial subgroup $P \le S$ with $Q \le N_S(P)$ is in $\Delta$.\\

The following, apparently innocuous, lemma is actually a fundamental tool for the success of the analysis that we will undertake; it also yields the same conclusion as \ref{lemma:properties of large in fus system}(e) under slightly weaker hypothesis.
\begin{lemma}\label{lemma:repleteness implies Y in Delta for parabolics}
    Let $(\L,\Delta,S)$ be a locality of objective characteristic $p$ and $Q \norm S$ a repleteness subgroup for $\L$. Then every subgroup $1 \ne P \norm S$ is such that $P \in \Delta$. In particular, the associated fusion systems $\F_S(\L)$ is of parabolic characteristic $p$ and for every parabolic subgroup $H \le \L$ one has $Y_H \in \Delta$.
\end{lemma}
\begin{proof}
    Since $P$ is non-trivial and normal in $S$, $1 \ne P \cap Z(S) \le C_S(Q)$, thus $P \cap Z(S) \in \Delta$ and, therefore, also $P \in \Delta$. Now \ref{cor:locality transl of char p type} gives that $\F_S(\L)$ is of parabolic characteristic $p$.\\
    Let now $H \le \L$ be a parabolic subgroup; then it is of characteristic $p$ by \ref{lemma:parabolic implies char p} and therefore $O_p(H) \ne 1$. Note that $O_p(H) \le O_p(C_H(\Omega Z(O_p(H))))$, thus $\Omega Z(O_p(H))$ is $p$-reduced in $H$ and we deduce that $1 \ne Y_H \norm S$, hence $Y_H \in \Delta$.
\end{proof}

The following lemma is a translation of \cite[Lemma 1.24]{Meier-Stell-Stroth:2012} and provides a number of useful properties. 

\begin{lemma}[1.24] \label{lemma:properties of parabolics in L}
Let $(\L,\Delta,S)$ be a locality of objective characteristic $p$ and $Q \norm S$ a repleteness subgroup. Let $L, \, M$ be parabolic subgroups of $\L$ and set $S_0 := C_S(Y_L)$ and $L_0:= N_L(S_0)$. Then $L$ and $M$ have characteristic $p$, $Y_L, Y_M \in \Delta$ and the following hold.
\begin{enumerate}[(a)]
\item $S_0 \in \Syl_p(C_\L(Y_L))$.
\item If $LC_\L(Y_M) =MC_\L(Y_M)$, then $Y_M \le Y_L$.
\item If $Y \le L \le M$ with $Y$ a $p$-reduced elementary abelian normal $p$-subgroup of $L$, then $\< Y^M \>$ is a $p$-reduced elementary abelian normal $p$-subgroup of $M$. In particular, $Y_L \le Y_M$.
\item $Z_L := \< \Omega Z(T) \mid T \in \Syl_p(L) \>$ is a $p$-reduced elementary abelian normal $p$-subgroup of $L$.
\item $Z_L = \Omega Z(L) [Z_L,O^p(L)]$ and $[Z_L,L] = [Z_L,O^p(L)]$.
\item If $L \le M$, then $Y_L \le Y_M$.
\item $O_p(L) \le S_0 \le C_L(Y_L)$ and $\Omega Z(S) \le Z_L \le Y_L \le \Omega Z(O_p(L))$.
\item Suppose that $L \le M$ and that there is a subgroup $Y_M \le Y \le \L$ such that $M \subseteq LC_\L(Y)$. Then $Y_L =Y_M$ and $LC_\L(Y_L) = MC_\L(Y_L)$.
\item The following hold.
\[
L=L_0 C_L(Y_L), \quad S_0 = O_p(L_0), \quad C_S(S_0)\le S_0, \quad \text{and} \quad Y_L=\Omega Z(S_0) =Y_{L_0}.
\]
\item Set $L^* := N_M(S_0)$. If $MC_\L(Y_L) = LC_\L(Y_L)$, then $Y_{L^*} = Y_L$ and $LC_\L(Y_L) = L^*C_\L(Y_L)$.
\item If $C_\L(Y_L)$ is $p$-closed, then
\[
Y_L = \Omega Z(O_p(L)) \quad \text{and} \quad O_p(L) = S_0 \in \Syl_p(C_\L(Y_L)).
\]
\end{enumerate}
\end{lemma}
\begin{proof}
The first assertion is simply \ref{lemma:parabolic implies char p} and \ref{lemma:repleteness implies Y in Delta for parabolics}.
\begin{enumerate}[(a)]
\item Clearly $S \le L \le N_\L(Y_L)$, hence $S \in \Syl_p(N_\L(Y_L))$. As $C_\L(Y_L) \norm N_\L(Y_L)$ we have $C_S(Y_L) = C_\L(Y_L) \cap S \in \Syl_p(C_\L(Y_L))$.
\item Since $L \le MC_\L(Y_M)$, we get that $Y_M$ is normalized by $L$. Moreover, we get
\[
\frac{L}{C_L(Y_M)} \cong \frac{LC_\L(Y_M)}{C_\L(Y_M)} = \frac{MC_\L(Y_M)}{C_\L(Y_M)} \cong \frac{M}{C_M(Y_M)};
\]
thus $Y_M$ is $p$-reduced also for $L$, proving $Y_M \le Y_L$.
\item This is \cite[Lemma 2.2(b)]{MS:2009}.
\item Since $L$ has characteristic $p$ and $\Omega Z(T) \le C_L(O_p(L)) \le O_p(L)$ for all $T \in \Syl_p(T)$, we have $Z_L \le S$. As $\Omega Z(S)$ is trivially $p$-reduced for $S$, (c) gives the assertion.
\item Since $\Omega Z(S) [Z_L,L]$ is normal in $L$ and $Z_L$ is the normal closure of $\Omega Z(S)$ in $L$, we get $\Omega Z(S) [Z_L,L] = Z_L$. By \cite[Theorem C.17]{Meier-Stell-Stroth:2012} we get $C_{Z_L}(S)[Z_L,L] = C_{Z_L}(L)[Z_L,L]$, hence, since $\Omega Z(S) \le C_{Z_L}(S)$ clearly holds,
\[
Z_L = C_{Z_L}(S)[Z_L,L] = C_{Z_L}(L)[Z_L,L] = \Omega Z(L) [Z_L,L].
\]
We therefore have $[Z_L,L]=[Z_L,L,L]$, so $L$ acts trivially on no non-trivial quotient of $[Z_L,L]$. Since $L = O^p(L)S$ and $S$ acts trivially on all the quotients of $[Z_L,L]$ of prime index (for example), we get $[Z_L,L] = [Z_L,L,O^p(L)] \le [Z_L, O^p(L)]$, thus $[Z_L,L]=[Z_L,O^p(L)]$ and (e) follows.
\item Since $L$ is a parabolic subgroup of $M$ and $M$ is of characteristic $p$, the result follows from \cite[Lemma 2.2(c)]{MS:2009} with $(L,M)$ in place of $(P,L=G)$.
\item $Z_L$ is $p$-reduced for $L$ by (d), hence $\Omega Z(S) \le Z_L \le Y_L \le O_p(L)$. By definition of $Y_L$, $O_p(L) \le C_L(Y_L)$, which shows (f).
\item By (c) $Y_L \le Y_M$; clearly we also have $C_\L(Y) \subseteq C_\L(Y_L)$, hence $M \subseteq LC_\L(Y_L)$. In particular, we get $MC_\L(Y_L) = LC_\L(Y_L)$, thus $Y_M = Y_L$ by (b).
\item By (a) $S_0 \in \Syl_p(C_\L(Y_L))$ and the Frattini argument yields $L=C_L(Y_L) L_0$.\\
Since $S_0 \norm S$ we get $S \le L_0$, thus $L_0$ is of characteristic $p$. Applying (g) to the pair $L_0 \le L$ we get $Y_{L_0} = Y_L$. Since $L_0$ is a parabolic subgroup of $\L$, always by (f) we have
\[
Y_L =Y_{L_0} \le \Omega Z(O_p(L_0)).
\]
Let $U$ be the unique maximal, normal subgroup of $L_0$ acting nilpotently on $\Omega Z(O_p(L_0))$; then $U$ acts nilpotently also on $Y_{L_0}$ and \cite[Lemma A.10]{Meier-Stell-Stroth:2012} shows that $U \le C_{L_0}(Y_{L_0})$. We get thus the following chain of inequalities
\[
O_p(L_0) \le U \cap S \le C_S(Y_{L_0}) = C_S(Y_L) = S_0 \le O_p(L_0), 
\]
which are therefore all equalities. This gives $O_p(L_0) = S_0$; since the generators of $O^p(U)$ act coprimely and nilpotently, we have $O^p(U) \le C_{L_0}(\Omega Z(O_p(L_0))) \le U$, where the latter inclusion holds by maximality of $U$. Then $U = (U \cap S)O^p(U) = (U \cap S) C_{L_0}(\Omega Z(O_p(L_0))) = C_{L_0}(\Omega Z(O_p(L_0)))$ as $U \cap S \le O_p(L_0)$. By the definition of $U$ and another application of \cite[Lemma A.10]{Meier-Stell-Stroth:2012}, $\Omega Z(O_p(L_0))$ is $p$-reduced for $L_0$, so we get
\[
Y_L \le \Omega Z(S_0) = \Omega Z (O_p(L_0)) \le Y_{L_0} = Y_L.
\]
Finally, $C_{L_0}(S_0) = C_{L_0}(O_p(L_0)) \le O_p(L_0) =S_0$, hence also $C_S(S_0) \le S_0$.
\item We have $Y_L \norm M$, so $C_M(Y_L) \norm M$ and $C_S(Y_L) \in \Syl_p(C_M(Y_L))$. By Frattini's argument then $M = C_M(Y_L) N_M(C_S(Y_L)) = C_M(Y_L) L^*$, so $LC_\L(Y_L) = M C_\L(Y_L) = L^* C_M(Y_L) C_\L(Y_L) = L^* C_\L(Y_L)$. Now (b) shows $Y_L \le Y_{L^*}$; by (h) $Y_L = \Omega Z(S_0)$ and $C_S(S_0) \le S_0$ so, as also $S_0 \le O_p(L^*)$, $Y_{L^*} \le \Omega Z(O_p(L^*)) \le C_S(S_0) \le S_0$. Hence we also get $Y_{L^*} \le \Omega Z(S_0) = Y_L$, proving the equality.
\item By (h) we only need to show $S_0 =  O_p(L)$. Since $C_L(Y_L)$ is $p$-closed,  $S_0 \in \Syl_p(C_L(Y_L))$ by (a) and $Y_L$ is $p$-reduced we get
\[
S_0 = O_p(C_L(Y_L)) \le O_p(L) \le S_0.
\]
\end{enumerate}
\end{proof}

\begin{remark}
    The lemma just proven shows that, in the presence of a repleteness subgroup, the parabolic subgroups of a locality of objective characteristic $p$ behave just as the parabolic subgroups of characteristic $p$ of a finite group, at least with respect to their action over their maximal, normal, elementary abelian, p-reduced subgroups. This analogy is what makes the techniques used in \cite{Meier-Stell-Stroth:2012} so successful also when applied to our locality $\L$.
\end{remark}

\space{0.4cm}
We conclude the section reporting the first fundamental result about the connection between fusion systems and localities with a large subgroup.
\begin{lemma}\label{lemma:large for fus systems and localities}
    Let $(\L,\Delta,S)$ be a locality over $\F$.
    \begin{enumerate}[(a)]
        \item If $Q$ is large in $\L$ and $\L$ is $Q$-replete, then $Q$ is large in $\F$, $\F^c \subseteq \Delta$ and $\F$ is saturated.
        \item If $\L$ is $Q$-replete and of objective characteristic $p$, then $Q$ is large in $\L$ if and only if it is large in $\F$. If so, then $\F$ is saturated and $\L$ a linking locality over $\F$.
        \item If $\F$ is saturated and $Q$ is large in $\F$, then $Q$ is large in any linking locality over $\F$.
    \end{enumerate}
\end{lemma}
\begin{proof}
    This is shown in \cite[Lemma 4.10]{Henke-Salati:2026}.
\end{proof}

In particular, if $\F$ is saturated, $Q \le S$ is large in $\F$ and $(\L^s, \F^s, S)$ is the subcentric linking locality over $\F$, then $\L^s$ is $Q$-replete by \ref{lemma:properties of large in fus system}(e). On the other hand, if a locality $(\L,\Delta,S)$ has a large subgroup $Q \le S$ which is also a repleteness subgroup, then \ref{lemma:large for fus systems and localities}(a) tells us precisely that $\F$ is saturated, thus we may replace the locality $\L$ with a linking locality by point (c) of the same lemma.\\
We conclude that, if our goal is to study saturated fusion systems with a large subgroup, we can equivalently study $Q$-replete localities of objective characteristic $p$, which by (b) are linking localities; moreover, we may additionally assume to have at hand the subcentric linking locality (over a given fusion system $\F$), often a convenient choice due to the extra properties that this particular linking locality guarantees.

\section{The strategy of the local structure theorem for groups} \label{section: strategy of LST for groups}
Understanding the strategy followed by Meierfrankenfeld, Stellmacher and Stroth in the proof of the local structure theorem is instrumental for our purposes; in particular, our choice of working with a $Q$-replete linking locality allows to follow a very similar argument, at least in most of the cases. It remains, instead, unclear how one could attempt a similar approach working directly with the fusion system.\\
In this section we provide an in depth description of the argument used by the authors of \cite{Meier-Stell-Stroth:2012} and we indeed begin by borrowing most of the notation which is in use in their paper.

\begin{notation}\label{notation for groups}
	Let $G$ be a finite group and define the following.
	\begin{itemize}
		\item $\C = \C_G := \{H \le G \mid O_p(H) \ne 1, \; C_G(O_p(H)) \le O_p(H) \}$; in \cite{Meier-Stell-Stroth:2012} this is the set denoted by $\L$, changed in this context to avoid confusion with our chosen standard symbol for localities.
		\item $\M = \M_G := \{ M \in \C \mid \textit{$M$ is maximal with respect to inclusion} \}$.
		\item $\P = \P_G := \{ P \in \C \mid \textit{$P$ is $p$-minimal} \}$.
		\item If $\K \subseteq \Sub(G)$ and $H \le L \le G$, then $\K_L(H) := \{ K \in \K \mid H \le K \le L \}$.
        \item If $Q \le L \le G$, set $L^\circ := \wcl_G(Q,L) = \< Q^g \mid g \in G, Q^g \le L \>$.
		\item Let $M \in \C$; set $M^\dagger :=MC_G(Y_M)$ and consider the following properties:
			\begin{itemize}
				\item[(i)] $\M(M) = \{M^\dagger\}$ and $Y_M = Y_{M^\dagger}$,
				\item[(ii)] $C_M(Y_M)$ is $p$-closed and $\displaystyle \frac{C_M(Y_M)}{O_p(M)} \le \Phi \left( \frac{M}{O_p(M)} \right)$.
			\end{itemize}
			Set $\gM = \gM_G := \{ M \in \M \mid \textit{$M$ satisfies (i) and (ii)} \}$; (i) and (ii) are referred to as the {\bf basic property} of $M$.
        \item If $X \le G$, set $\ov{X} := X / C_X(Y_X)$.
	\end{itemize}
\end{notation}

We can now state the Local Structure Theorem for groups, even though we refer the reader to the original paper for the long list of all possible outcomes for the modules.

\begin{maintheorem}{A}{\normalfont (Meierfrankenfeld, Stellmacher, Stroth)}
	\label{theorem A}
    Let $G$ be a finite $\K_p$-group and $S \in \Syl_p(G)$. Suppose that $S$ is contained in at least two distinct maximal, $p$-local subgroups of $G$ and that $Q \le S$ is a large subgroup of $G$. Let $L \le G$ such that $Q \le L$, $O_p(L) \ne 1$ and $Q \not \norm L$.\\
    Then there exist subgroups $M \in \gM_G(S)$ and $L^* \le M$ such that
    \[
        S \le L^*, \, Y_L = Y_{L^*}, \, L C_G(Y_L) = L^* C_G(Y_L) \, \text{ and } \, L^\circ = (L^*)^\circ .
    \]
    Moreover, for any such $M$ and $L$ one of the ten cases in \cite[Theorem A]{Meier-Stell-Stroth:2012} holds, providing an almost complete classification of the possible module structures of $Y_L$ for $L$ and $Y_M$ for $M$.
\end{maintheorem}

\vspace*{0.2cm}
Analyzing a little further the statement and combining with some additional results obtained by the three authors one can extract the following information.
\begin{enumerate}[(i)]
	\item One may assume $Q$ to be the largest large subgroup of $G$ contained in $S$; in such a case, $Q = O_p(N_G(Q))$ (this assumption is avoided in \cite{Meier-Stell-Stroth:2012}).
	\item $Q$ is a weakly closed subgroup of $G$.
	\item By $(Q!)$ one gets $N_G(Q) = N_G(Z(Q))$; in addition, $Z(Q)$ is a T.I. set (trivial intersecting set) with respect to $G$, that is for any $g \in G$,
	\[
	Z(Q) \cap Z(Q)^g \ne 1 \implies Z(Q) =Z(Q)^g;
	\]
	in particular, it implies $g \in N_G(Q)$.
	\item $G$ is of parabolic characteristic $p$.
	\item Either $\M(S) = \{N_G(Q)\}$ or there exists $M \in \gM(S)$ such that $Q \not \norm M$.
	\item For every $L \in \C(S)$ such that $Q \not \norm L$ there exists $M \in \gM_G(S)$ and $L^* \le M$ such that $Y_L = Y_{L^*} \le Y_M$ and $L C_L(Y_L) = L^* / C_{L^*}(Y_{L})$; in other words, $Y_L$ is the same module when regarded as $L$- or as $L^*$-module.
\end{enumerate}
Properties (i) to (iv) come from \ref{prop:properties of large in a group}; (v) and (vi) are instead shown in \cite[Lemma 1.56]{Meier-Stell-Stroth:2012}. In particular, (v) justifies the assumption that $S$ is contained in at least two distinct maximal, $p$-local subgroups, thus excluding the case $\M(S) = \{N_G(Q)\}$. Moreover, we may regard (vi) as a reduction theorem: since the module $Y_L = Y_{L^*}$ does not change when regarded as an $L$- or an $L^*$-module we only need to look at subgroups $S \le L^* \le M$ for those groups $M$ with the basic property and therefore, ultimately, at the groups in $\gM_G(S)$. Since all such groups are of characteristic $p$, the successful strategy deployed by Meierfrankenfeld, Stellmacher and Stroth consists in studying the $M$-module $Y_M$, and this is in turn done by looking for a suitable offender on $Y_M$ (or, more generally, on some module related to $Y_M$). Many results classifying modules with offenders are available and provide lists of possibilities for the pairs $(M,Y_M)$. A large part of the work is then devoted to sifting through such lists, detecting those cases that are in conflict with the hypothesis of the theorem and, therefore, cannot occur within a $G$ with a large $p$-subgroup $Q$.\\
In contrast with the choice made in \cite{Meier-Stell-Stroth:2012}, from now on we will assume to have chosen $Q$ as the largest possible large subgroup of $G$ contained in $S$. By \ref{prop:properties of large in a group}(a) and (b) this is always possible and one has $Q = O_p(N_G(Q))$, and it will greatly simplify the upcoming discussion.\\

Depending mainly on the properties of the offender, the authors of \cite{Meier-Stell-Stroth:2012} split their analysis into five distinct cases, each yielding a list of possibilities for $(M,Y_M)$: the last step consists then in the study of the possible pairs $(L,Y_L)$ embedded in $(M,Y_M)$. We now describe more in detail each of the five cases, starting with the following fundamental definition given by Meierfrankenfeld et al.

\begin{definition}\label{def:symm-tall-groups}
	Let $A$ be a an abelian, normal $p$-subgroup of $G$ and $\N$ a set of subgroups of $G$.
	\begin{itemize}
		\item $A$ is {\bf symmetric} in $G$ if there exists $B \in A^G$ such that
			\[
			1 \ne [A,B] \le A \cap B.
			\]
			Otherwise $A$ is called {\bf asymmetric}.
		\item $A$ is {\bf $\N$-tall} in $G$ if there exist $T \in \Syl_p(C_G(A))$ and a subgroup $N 
			\in	\N$ such that
			\[
			T \le N \quad \text{and} \quad A \not \le O_p(N).
			\]
			$A$ is {\bf $\N$-short} if it is not $\N$-tall, that is if for every $N \in \N$ the following implication holds:
			\[
			T \in \Syl_p(C_G(A)) \text{ is such that } T \le N \; \implies \; A \le O_p(N).
			\]
		\item We say that $A$ is {\bf tall} in $G$ (resp., {\bf short} in $G$) when we take $\N := \{ N \le G \mid O_p(N) \ne 1 \}$; it is {\bf char $p$-tall} (resp., {\bf char $p$-short}) when $\N := \{N \le G \mid \textit{N is of characteristic $p$} \}$; it is {\bf $Q$-tall} (resp., {\bf $Q$-short}) when $\N := N_G(Q)^G$.
	\end{itemize}
\end{definition}

The definition above corresponds to the case subdivision found in \cite[Notation 2.1]{Meier-Stell-Stroth:2012}, where $Y_M$ (for $M \in \gM_G(S)$) is taken in place of $A$. Then \cite[Lemma 2.2(f)]{Meier-Stell-Stroth:2012} shows that $C_S(Y_M) = O_p(M) \in \Syl_p(C_G(Y_M))$, which is therefore the unique Sylow $p$-subgroup. Then for sets of subgroups $\N$ which are $G$-conjugacy invariant one has the following:
\begin{itemize}
	\item $Y_M$ is $\N$-tall if there exists $N \in \N$ such that 
		\[
		O_p(M) \le N \quad \text{and} \quad Y_M \not \le O_p(N).
		\]
	\item $Y_M$ is $\N$-short if for every $N \in \N$
		\[
		O_p(M) \le N \implies Y_M \le O_p(N).
		\]
	\item Assume that $Y_M$ is asymmetric (which will always be the case when tallness and shortness are needed) and that $O_p(M) \le N_G(Q)^x$ for some $x \in G$. Then $O_p(M) \le T \in \Syl_p(N_G(Q)^x \cap N_G(Y_M))$ for some $T$. As $O_p(M) \le S \le N_G(Y_M)$ is weakly closed in $G$ by \cite[Lemma 2.6(b)]{Meier-Stell-Stroth:2012}, we get $T \le N_G(O_p(M)) \le N_G(Y_M)$, where the last inclusion is \cite[Lemma 2.2(c)]{Meier-Stell-Stroth:2012}. In particular, there exists $g \in N_G(O_p(M))$ with $T^g \le S$. Thus, up to replacing $x$ with $xg$ we may assume $T \le S$ and, therefore, $Q \le T$ (the unique $G$-conjugate of $Q$ in $T$); it follows that $Y_M$ is $Q$-short if and only if $Y_M \le O_p(N_G(Q)) = Q$.  This easy observation will be extremely useful when dealing with the locality $\L$ in place of the group $G$.
\end{itemize}

We now dive into the case history appearing in \cite{Meier-Stell-Stroth:2012}, discussing the main strategies for the identification of the modules case by case.
\begin{enumerate}
	\item[(Sym)] {\bf The Symmetric Case.}\\
		 $Y_M$ is symmetric in $G$, i.e. there is a conjugate $Y_M^u \in Y_M^G$ such that $1 \ne [Y_M,Y_M^u] \le Y_M \cap Y_M^u$. One shows that, up to conjugation by an element in $N_G(Y_M)$, one can always choose $u \in G$ such that $Y_M^u \le S$; then one of $Y_M^u$ or $Y_M^{u\inv}$ is a non-trivial quadratic offender on $Y_M$. As a consequence, the triple $(\ov{M},Y_M,\ov{Q})$ satisfies the hypothesis of the \nameref{thm:Q!FF}, which becomes therefore the main tool of investigation.\\
		 The argument is actually carried in a slightly more general situation, for $Y$ any non-trivial, elementary abelian, normal, $p$-reduced $p$-subgroup of $M$ (in place of $Y_M$); this more general setting is then needed in the asymmetric $Q$-tall case and needs no extra work.
	\item[(Asym)] {\bf The Asymmetric Case.}\\
		$Y_M$ is asymmetric in $G$, the main consequence being that for any subgroup $R \le N_G(Y_M)$ the group $\< Y_M^{N_G(R)} \>$ is an elementary abelian $p$-subgroup of $G$. This case is further split into four parts, depending on the relative embedding of $Y_M$, $Q$ and of a suitably chosen class of subgroups of $G$.
			\begin{enumerate}[label=(Asym: \alph*), itemindent=0.5cm]
				\item {\bf The Short Case.}\\
					 $Y_M$ is short in $G$; this corresponds to the following configuration
					 
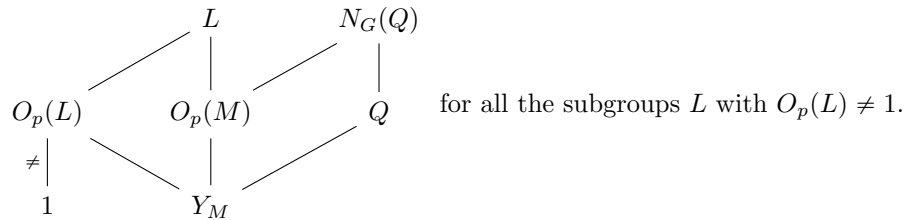
\begin{center}
					 	\begin{tikzcd}
					 				& L 		& N_G(Q)\\
					 		O_p(L)	\ar[ru,no head]	&  O_p(M) \ar[u,no head] \ar[ru,no head] & Q \ar[u,no head]\\
					 		1 \ar[u,no head, "\ne"]	& Y_M \ar[lu,no head] \ar[u,no head] \ar[ru,no head]
					 	\end{tikzcd}
                        for all the subgroups $L$ with $O_p(L) \ne 1$.
					 \end{center}
				Note that one gets $Y_M \le Q$ by applying shortness to the subgroup $N_G(Q)$.\\
                The reasoning here resembles that of the symmetric case. One always looks at pairs of conjugates of $Y_M$, say $(Y_1,Y_2)$, but needs then to replace them with larger elementary abelian $p$-subgroups $(V_1,V_2)$. In detail, one shows that it is possible to find subgroups $L_1, L_2$, called \emph{indicators} (see \cite[Definition 2.19]{Meier-Stell-Stroth:2012}), such that, for $V_i := \< Y_i^{L_i} \>$, one has
		\[
		V_1V_2 \le L_1 \cap L_2 \quad \text{and} \quad 1 \ne [V_1,V_2] \le V_1 \cap V_2;
		\]
		hence the pair $(V_1,V_2)$ plays the role that $(Y_1,Y_2)$ had in the symmetric case.	One then studies the actions of the $L_i$s on the $V_i$s to deduce information about the action of $N_G(Y_i)$ on $Y_i$; in particular, the groups $V_i$ allow to find offenders on the $Y_i$s, but this requires a very technical argument, not as simple as the one used in the symmetric case. Here a strong dual offender is found, hence the main tool becomes \cite[Theorem C.27]{Meier-Stell-Stroth:2012}.
				\item {\bf The Tall, char $p$-Short Case.}\\
					 $Y_M$ is tall, but char $p$-short in $G$; it corresponds to the configuration 
					 
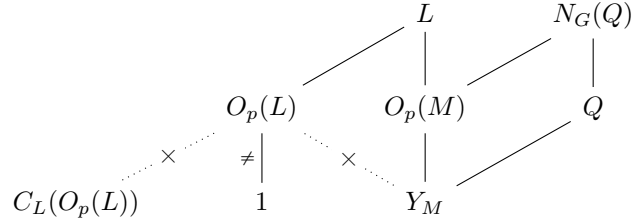
\begin{center}
					 	\begin{tikzcd}
					 			& & L 	& N_G(Q) \\
					 			&	O_p(L)	\ar[ru,no head]	&  O_p(M) \ar[u,no head] \ar[ru,no head] & Q \ar[u,no head] \\
					 		C_L(O_p(L)) \ar[ru,dotted,no head, "\bigtimes" description]	& 1 \ar[u,no head, "\ne"]	& Y_M \ar[lu,no head,dotted, "\bigtimes" description] \ar[u,no head]  \ar[ru,no head]
					 	\end{tikzcd}
                    \end{center}
                for some subgroup $L$ with $O_p(L) \ne 1$ and not of characteristic $p$ (where we have emphasized the lack of the dotted inclusions). All the subgroups $N$ of characteristic $p$ containing $O_p(M)$ realize the inclusion $Y_M \le O_p(N)$; in particular, we obtain again $Y_M \le Q$.\\
                This is the case with the weakest conclusion. All the authors where able to prove is that this case can only occur if there exists $1 \ne y \in Y_M$ such that $C_G(y)$ is not of characteristic $p$; it generates the case (10)(1) in the list found in \cite[Theorem A]{Meier-Stell-Stroth:2012}, the only one where no module was identified for $Y_M$.
				\item {\bf The char $p$-Tall, $Q$-Short Case.}\\
					 $Y_M$ is char $p$-tall, but $Q$-short; it corresponds to the configuration
					 
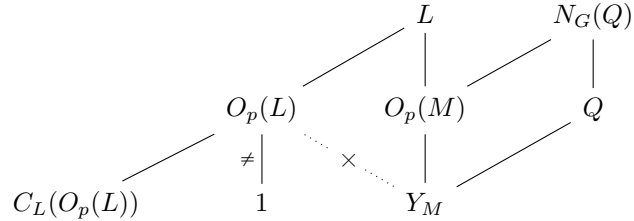
\begin{center}
					 	\begin{tikzcd}
								&			 		& L						&	N_G(Q)	\\
					 			& O_p(L)	\ar[ru,no head]	&  O_p(M) \ar[u,no head] \ar[ur,no head] &  Q \ar[u,no head] \\
					 		C_L(O_p(L)) \ar[ru,no head]	& 1 \ar[u,no head, "\ne"]	& Y_M \ar[lu,no head, dotted, "\bigtimes" description] \ar[u,no head] \ar[ru, no head]
					 	\end{tikzcd}
					 \end{center}
					 for some subgroup $L$ of characteristic $p$, but for no $G$-conjugate of $N_G(Q)$.\\
                The analysis of this case relies on the study of subgroups of $G$ lying in the following sets.
		
		\begin{definition}\label{def:subsets of subgr fot the tall cases}
			\begin{itemize} 
			\item[1.] Let $K \le G$ be such that $O_p(M) \le K$, then 
			\[
			\h_K(O_p(M)) := \{ H \le K \mid \textit{(i), (ii), (iii) hold} \},
			\]
			where one has
			\begin{itemize}
				\item[(i)] $H$ has characteristic $p$,
				\item[(ii)] $O_p(M) \le H$ but $Y_M \not \le O_p(H)$,
				\item[(iii)] if $O_p(M) \le P \le H$, then $Y_M \le O_p(M)$.
			\end{itemize}
			So $\h_K(O_p(M))$ is the set of minimal subgroups of $K$ realizing the tallness condition for $Y_M$.
			\item[2.] Let now $K \le G$ be such that $Y_M \le K$, then 
			\[
			\cL_K(Y_M) := \{Y_M \le L \le K \mid \textit{(I), (II), (III),(IV) hold} \},
			\]
			where one has
			\begin{itemize}
				\item[(I)] $O_p(L) = \< (Y_M \cap O_p(L))^L \> \le N_L(Y_M)$,
				\item[(II)] $L$ is $Y_M$-minimal and of characteristic $p$,
				\item[(III)] $N_L(Y_M)$ is the unique maximal subgroup of $L$ over $Y_M$,
				\item[(IV)] $L / O_p(L) \in \{ \SL_2(q), \mathrm{Sz}(q), D_{2r} \}$ where $q := |Y_M : (Y_M \cap O_p(L)) |$, $p=2$ in the last two cases and $r$ is odd. 
			\end{itemize}
			Note that elements of $\cL_K(O_p(M))$ appear naturally as subgroups of elements in $\h_K(O_p(M))$, see \cite[Lemmas 2.15 and 2.16]{Meier-Stell-Stroth:2012}.
			\end{itemize} 
		\end{definition}
		By definition, there is an element $H \in \h_G(O_p(M))$ and this implies the existence of an element $L \in \cL_H(Y_M)$. One then proves that $Y_H$ contains a quasisimple $H$-submodule $U$ and that $W:=[U,L]$ is a strong offender on $Y_M$. The main tools become then the $FF$-module theorems \cite[Theorems 1, 2 and 3]{MS:2012a} to describe the action of $\<W^M \>$ on $Y_M$.
				\item {\bf The $Q$-Tall Case.}\\	
					$Y_M$ is $Q$-tall; it corresponds to
					
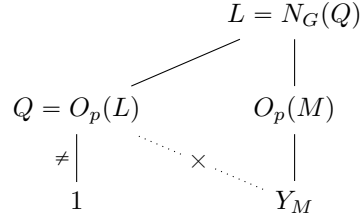
\begin{center}
					\begin{tikzcd}
						& L=N_G(Q) \\
						Q = O_p(L)	\ar[ru,no head]	&  O_p(M) \ar[u,no head] \\
						1 \ar[u,no head, "\ne"]	& Y_M \ar[lu,no head, dotted, "\bigtimes" description] \ar[u,no head]
					\end{tikzcd}
					\end{center}
					for $L \in N_G(Q)^G$; one can then reduce to the case $L=N_G(Q)$, as explained above.\\
                    It is the longest and most complex case of \cite{Meier-Stell-Stroth:2012}. Here $Y_M \le Q$ and the groups in the sets $\h_G(O_p(M))$ and $\cL_G(Y_M)$ play, as in the previous case, a central role. Moreover, $Q$-tallness allows to choose $H \in \h_G(O_p(M))$ and $L \in \cL_G(Y_M)$  such that $L \le H \le N_G(Q)$. The analysis of the $M$-module $Y_M$ is then conducted inside of the amalgam generated by $N_G(Y_M)$ and by $N_G(Q)$. Here the notions of the Fitting submodule and of nearly quadratic action as well as the theorems of Guralnick and Malle \cite{Guralnick-Malle:2002} and \cite{Guralnick-Malle:2004} on modules with a non-trivial $2F$-offender come into play.\\
		In detail, one replaces $Y_M$ with its Fitting submodule $I$, which can be more easily described. Then, for $A := O_p(L)$, the following three cases are studied:
		\begin{enumerate}[(1)]
			\item $I \le A$,
			\item $I \not \le A$ and $[\Omega Z(A),L] \ne 1$,
			\item $I \not \le A$ and $[\Omega Z(A),L] = 1$.
		\end{enumerate} 
		In case (1) $I$ turns out to be symmetric in $G$; since it is normalized by $S$, the results relative to the symmetric case apply to $I$.\\
		Case (2) corresponds to non-trivial action of $H$ on $\Omega Z(O_p(H))$, a situation similar to that occurring in the char $p$-tall, $Q$-short case: indeed, here as well the analysis mainly relies on the $FF$-module theorems.\\
		Case (3) corresponds to the situation where $A$ acts nearly quadratically on $I$ and the Nearly Quadratic $Q!$-Theorem \cite[Theorem D.11]{Meier-Stell-Stroth:2012} is invoked.\\
		This case-analysis is carried out in Chapter 8, together with the identification of the module $Y_M$ whenever a suitable offender is detected. Since this is not always possible, \cite[Chapter 9]{Meier-Stell-Stroth:2012} deals with the outcomes of the previous chapter that don't admit an offender; precisely here is where one can only guarantee the existence of a $2F$-offender. A great deal of work is then devoted to exclude those $2F$-modules that appear in the conclusion of the Guralnick -- Malle theorems but are in conflict with the hypothesis on $G$.
			\end{enumerate}
\end{enumerate}

At the risk of oversimplifying, one may say that each of the cases above leads to a description of the possible pairs $(M,Y_M)$. By the main reduction result one only needs to look at modules $(L,Y_L)$ with $S \le L \le M$ and $Y_L \le Y_M$, which is then taken care in Chapter 10 of \cite{Meier-Stell-Stroth:2012}.

\section{The reduction to the basic property for localities}
\label{basic property in localit}

This section justifies our choice of working with a $Q$-replete linking locality, rather than directly with the fusion system.\\

For the rest of the section, unless differently specified, we assume that $(\L,\Delta,S)$ is a locality of objective characteristic $p$ and $Q \le S$ a large subgroup of $\L$ such that $\L$ is $Q$-replete. Note that \ref{lemma:large for fus systems and localities}(a) and (b) show that $\F_S(\L)$ is saturated and that $\L$ is a linking locality for the fusion system $\F_S(\L)$.\\
Moreover, by \ref{lemma:large for fus systems and localities}(c) and \ref{lemma:large implies weakly closed} there is no limitation in assuming that $\L$ is the subcentric linking locality over $\F_S(\L)$ and $Q$ is the largest large subgroup of $\L$ (contained in $S$).

We adopt a notation analogous to that in \ref{notation for groups}.
\begin{notation}\label{notation for localities}
Let $(\L, \Delta,S)$ be a locality and define the following.
	\begin{itemize}
		\item $\C = \C_\L := \{H \le \L \mid C_\L(O_p(H)) \subseteq O_p(H) \}$; note that the inclusion condition forces $O_p(H) \ne 1$ and $C_\L(O_p(H))$ to be a group, a subgroup of $O_p(H)$, but it does not necessary imply that $O_p(H) \in \Delta$.
		\item $\M = \M_\L := \{ M \in \C \mid \textit{$M$ is maximal with respect to inclusion} \}$.
		\item $\P = \P_\L := \{ P \in \C \mid \textit{$P$ is $p$-minimal} \}$.
		\item If $\K \subseteq \Sub(\L)$ and $H \le L \le \L$, then $\K_L(H) := \{ K \in \K \mid H \le K \le L \}$.
		\item Let $M \in \C$; set $M^\dagger :=MC_\L(Y_M)$ and consider the following properties:
			\begin{itemize}
				\item[(i)] $\M(M) = \{M^\dagger\}$ and $Y_M = Y_{M^\dagger}$,
				\item[(ii)] $C_M(Y_M)$ is $p$-closed and $\displaystyle \frac{C_M(Y_M)}{O_p(M)} \le \Phi \left( \frac{M}{O_p(M)} \right)$.
			\end{itemize}
			Set $\gM = \gM_\L := \{ M \in \M \mid \textit{$M$ satisfies (i) and (ii)} \}$; (i) and (ii) are referred to as the {\bf basic property} of $M$.
        \item If $X \le \L$, set $\ov{X} := X / C_X(Y_X)$.
	\end{itemize}
\end{notation}
Note that condition (i) of the basic property implicitly requires $M^\dagger$ to be a subgroup of $\L$; in Theorem~\ref{theorem: main reduction} we will see that, in the locality $\L$, $M \in \gM$ may be chosen to be a parabolic subgroup. Then one has $Y_M \in \Delta$ by \ref{lemma:repleteness implies Y in Delta for parabolics} and $M^\dagger$ becomes a subgroup of the group $N_\L(Y_M)$.

\begin{remark}
    More generally, \ref{lemma:parabolic implies char p} and \ref{lemma:repleteness implies Y in Delta for parabolics} show that, if $H \le \L$ is parabolic, then $1 \ne Y_H \in \Delta$ and therefore $H^\dagger =HC_\L(Y_H) \le N_\L(Y_H)$ is subgroup of $\L$. If moreover $C_\L(Y_H) \le H$, that is $H=H^\dagger$, then $C_\L(O_p(H)) \le C_\L(Y_H) \le H$ and one gets
    \[
    C_\L(O_p(H)) = C_H(O_p(H)) \le O_p(H),
    \]
    hence $H \in \C_\L$. This argument doesn't require the presence of a large subgroup, only that of a repleteness subgroup in a locality of objective characteristic $p$.\\    
    As a consequence, lemmas \ref{lemma:partial order for the main reduction}, \ref{lemma:unique-factor-family} and \ref{lemma:M in factorization family is local} that follow remain true under the weaker hypothesis that $Y_L \in \Delta$ for every parabolic subgroup of the locality $\L$ of objective characteristic $p$, a condition that holds by \ref{lemma:repleteness implies Y in Delta for parabolics} under the existence of a repleteness subgroup of $\L$.
\end{remark}

The argument now follows quite closely that in \cite[Section 3]{MS:2009}, adapted to our locality, as well as ideas similar to those in \cite[Lemmas 1.25 and 1.56]{Meier-Stell-Stroth:2012}.

\begin{definition}
	Set $\S := \{ \textit{parabolic subgroups of $\L$} \}$ and let $A,B \in \S$. Define
	\[
	A \ll B \iff \left( A \subseteq C_\L(Y_A)B \quad \text{and} \quad Y_A \le Y_B \right).
	\]
	Set also
	\[
	A^\dagger := AC_\L(Y_A) \quad \text{and} \quad \S^\dagger := \{ L \in \S \mid L = L^\dagger\}.
	\]
\end{definition}

\begin{definition}\label{def:factoriz-family}
A \emph{factorization family} for $\S$ is a subset $F_\S \subseteq \S$ satisfying the following two properties:
\begin{itemize}
	\item[(i)] for every $H \in \S$ there exists $M \in F_\S$ such that
	\[
	H \ll M;
	\]
	\item[(ii)] if $H \in \S$ and $M \in F_\S$ satisfy $M \ll H$, then
	\[
	Y_H = Y_M \quad \text{ and } \quad H \le M.
	\]
	\end{itemize}
\end{definition}
Observe that under condition (ii) $H \le M$ trivially implies $H \subseteq C_\L(Y_H)M$, thus $H \ll M$. This will translate in the fact that $\ll$ restricted to $\S^\dagger$ yields a partial order, of which $F_\S$ constitutes the set of maximal elements.

\begin{lemma}\label{lemma:partial order for the main reduction}
    Let $L,M \in \S$; then the following hold.
    \begin{enumerate}[(a)]
        \item $Y_L \le Y_{L^\dagger}$, $L \ll L^\dagger$ and $(L^\dagger)^\dagger = L^\dagger$; hence $L^\dagger \in \S^\dagger$.
        \item $\S^\dagger = \{L \in \S \mid C_\L(Y_L) \le L \}$.
        \item $L \subseteq C_\L(Y_L) M \iff L \le C_\L(Y_L)N_M(Y_L) \le N_\L(Y_L)$.
        \item If $L \subseteq C_\L(Y_L)M$, then $Y_L$ is $p$-reduced for $N_M(Y_L)$ and $L \ll N_M(Y_L)$.
        \item If $L \in \S^\dagger$, then $L \ll M$ if and only if $Y_L \le Y_M$ and $L = C_\L(Y_L)(L \cap M)$.
        \item The relation $\ll$ is reflexive and transitive on $\S$ and a partial order on $\S^\dagger$.
    \end{enumerate}
\end{lemma}
\begin{proof}
\begin{enumerate}[(a)]
	\item By \ref{lemma:repleteness implies Y in Delta for parabolics} $L^\dagger$ is a subgroup of the group $N_\L(Y_L)$. Clearly $Y_L \norm L^\dagger$ and $\displaystyle O_p \left( \frac{L^\dagger}{C_{L^\dagger}(Y_L)} \right) = O_p \left( \frac{L C_\L(Y_L)}{C_\L(Y_L)} \right) =1$, so that $Y_L \le Y_{L^\dagger}$. Since $L \le C_\L(Y_L)L^\dagger = L^\dagger$ we get $L \ll L^\dagger$. Also $(L^\dagger)^\dagger = L^\dagger C_\L(Y_{L^\dagger}) \le L^\dagger C_\L(Y_L) = L^\dagger \le (L^\dagger)^\dagger$.
	\item If $C_\L(Y_L) \le L$, then $L^\dagger= L C_\L(Y_L) = L \le L^\dagger$, so that $L \in \S^\dagger$. On the other hand, if $L \in \S^\dagger$, then $L^\dagger= L C_\L(Y_L)=L$, that is $C_\L(Y_L) \le L$.
	\item Since $N_M(Y_L) \le N_\L(Y_L)$, we have that $C_\L(Y_L)N_M(Y_L)$ is a subgroup of $\L$. Clearly $L \le C_\L(Y_L)N_M(Y_L)$ implies $L \subseteq C_\L(Y_L)M$. If $L \subseteq C_\L(Y_L)M$, as also $L \le N_\L(Y_L)$ we have $L \subseteq C_\L(Y_L)(M \cap N_\L(Y_L)) = C_\L(Y_L)N_M(Y_L)$.
	\item By (c) $D:=C_\L(Y_L)N_M(Y_L)$ isa subgroup of $\L$ and $L \le L^\dagger \le D$. Since $L$ is parabolic we have $O_p(D) \le O_p(L^\dagger) \le O_p(L)$. Thus
    \[
        [Y_L,O_p(D)] \le [Y_L,O_p(L^\dagger)] \le [Y_L,O_p(L)] =1,
    \]
    showing that $Y_L$ is $p$-reduced for $D$ and, therefore, for $N_M(Y_L)$. In particular, $Y_L \le Y_{N_M(Y_L)}$ and so $L \ll N_M(Y_L)$.
	\item Since $C_\L(Y_L) \le L$, we have that $L \subseteq C_\L(Y_L)M$ implies $L=C_\L(Y_L)(L \cap M)$. Now (e) follows immediately.
	\item That $\ll$ is reflexive is trivial. Suppose $A \ll B \ll C$. Then $Y_A \le Y_B \le Y_C$. Also $A \subseteq C_\L(Y_A)B$ and $B \subseteq C_\L(Y_B)C$; as $C_\L(Y_B) \le C_\L(Y_A)$, we get $A \subseteq C_\L(Y_A)C$, thus $A \ll C$.\\
	Suppose now that $A,B \in \S^\dagger$ and $A \ll B$ and $B \ll A$. Then clearly $Y_A = Y_B$ and both $A \subseteq C_\L(Y_A)B$ and $B \subseteq C_\L(Y_B)A$ hold. Thus (e) implies $A = C_\L(Y_A)(A \cap B) = C_\L(Y_B)(B \cap A) = B$, proving that $\ll$ is a partial order on $\S^\dagger$.
\end{enumerate}
\end{proof} 

\begin{lemma}\label{lemma:unique-factor-family}
Let $F_\S$ be a factorization family for $\S$, then
\[
    F_\S = \{ \textit{maximal elements of $\S^\dagger$ with respect to $\ll$} \}.
\]
In particular, $F_\S$ is uniquely determined.
\end{lemma}
\begin{proof}
Let $M \in F_\S$; by \ref{lemma:partial order for the main reduction}(a) $M \le M^\dagger$, $Y_M \le Y_{M^\dagger}$ and $M^\dagger \in \S^\dagger$. By (ii) of \ref{def:factoriz-family} we obtain then $Y_M = Y_{M^\dagger}$ and $M^\dagger \le M$,  so $M=M^\dagger \in \S^\dagger$; whence $F_\S \subseteq \S^\dagger$.\\
Recall that $\ll$ is a partial order on $\S^\dagger$; if $H \in \S^\dagger$ is such that $M \ll H$, then again by (ii) of \ref{def:factoriz-family} $Y_M = Y_H$ and $H \le M$ and \ref{lemma:partial order for the main reduction}(e) yields $M = C_\L(Y_M)(M \cap H) = C_\L(Y_H)H=H^\dagger=H$. Thus we get
\[
F_\S \subseteq \{ \textit{maximal elements of $\S^\dagger$ with respect to $\ll$} \}.
\]
Suppose now that $M$ is a maximal element of $\S^\dagger$ with respect to $\ll$. Since $F_\S$ is a factorization family, by (i) there exists $L \in F_\S$ such that $M \ll L$, hence $L \in \S^\dagger$ by the first part of the proof and the maximality of $M$ yields $M = L$. Thus $M \in F_\S$, proving the wanted equality.
\end{proof}

\begin{lemma}\label{lemma:M in factorization family is local}
    Let $F_\S$ be the factorization family for $\S$ and $M \in F_\S$ and $H \in \S$ be such that $M=C_M(Y_M)(M \cap H)$. Then $H \le M$, $M$ is a local subgroup of $\L$ and the unique maximal subgroup of $\L$ containing $H$ and, moreover, $M=M^\dagger=N_\L(Y_M)$. 
\end{lemma}
\begin{proof}
We prove that $M \ll H$; as clearly $M \subseteq C_\L(Y_M)H$, we only need to show that $Y_M \le Y_H$.\\
From $M=C_M(Y_M)(H\cap M)$ we have that $M/C_M(Y_M) \cong (H \cap M)/C_{H\cap M}(Y_M)$, so $Y_M$ is a $p$-reduced subgroup of $H \cap M$. Now $S \le H \cap M$ and $H$ has characteristic $p$, so \ref{lemma:unique largest p-reduced} yields $Y_M \le Y_{H \cap M} \le Y_H$. By \ref{def:factoriz-family}(ii) we have then $H \le M$.\\
Suppose that $H \le L$ for some other subgroup $L$, which then falls in $\S$. From $M = C_M(Y_M)(H \cap M) \le C_M(Y_M)(L \cap M) \le M$ we get that the inclusions are all equalities. Applying now the first part of the lemma to $M$ and $L$ we have $L \le M$, hence $M$ is the unique maximal subgroup of $\L$ containing $H$. Since every subgroup of a locality is contained in a local subgroup, $M$ has to be local; taking now $H=N_\L(Y_M)$ we see $M \le H$ and $H \le M$; thus $M=N_\L(Y_M)$ and $M^\dagger = M C_\L(Y_M) = M$.
\end{proof}

Recall that, for a sugroup $L \le \L$ subgroup such that $Q \le L$, we have set $L^\circ:= \wcl_L(Q,L)$ and $L_\circ := O^p(L^\circ)$.

\begin{lemma}[1.52]\label{lemma: L containing Q}
	Let $L \le \L$ be a subgroup such that $Q \le L$ and $1 \ne Y \in \Delta$ be a subgroup normalized by $L$. Let also $g \in \L$ be an element such that the conjugation map $c_g \colon Q \fto Q^g$ is a well-defined group homomorphism. Then the following hold.
	\begin{enumerate}[(a)]
		\item $L^\circ = (LC_\L(Y))^\circ$ and $[L^\circ, C_\L(Y)] \le O_p(L^\circ)$; hence $C_\L(Y)$ normalizes $L^\circ$.
		\item Set $\wde{L}:=L / O_p(L)$ and suppose that $O^p(L) \le L^\circ$ and $L=O^{p'}(L)$. Then $\wde{C_L(Y)} \le Z(\wde{L^\circ}) \le \Phi(\wde{L^\circ}) = \Phi(\wde{L_\circ})$.
		\item If $g \in \L \setminus N_\L(Q)$, then $Z(Q) \cap Z(Q^g) = 1$.
	\end{enumerate}
\end{lemma}
\begin{proof}
	
	Let $H:= LC_\L(Y) \le N_\L(Y)$; since $Y \ne 1$ also $C_Y(Q) \ne 1$, thuse $C_Y(Q) \in \Delta$ by $Q$-repleteness and $Q!$ yields $C_\L(Y) \le N_\L(C_Y(Q)) \le N_\L(Q)$. Therefore $H \subseteq LN_\L(Q)$, which is equivalent to $L^\circ =H^\circ$ by \ref{lemma:weakly closed subgroups in localities}(f).  Since $Q $ normalizes $Y$,  conjugation by elements of $Q$ extends to $C_\L(Y)$, so that
	\[
	[Q,C_\L(Y)] \le Q \cap C_\L(Y) \norm C_{L^\circ}(Y) \norm L^\circ.
	\]
	Then $Q \cap C_\L(Y) \le O_p(C_{L^\circ}(Y)) \le O_p(L^\circ)$ and we conclude $[Q,C_\L(Y)] \le O_p(L^\circ)$. By \ref{lemma:weakly closed subgroups in localities}(c) $L^\circ = \< Q^{L^\circ} \>$, so
	\[
	[L^\circ,C_\L(Y)] = [Q,C_\L(Y)]^{L^\circ} \le O_p(L^\circ),
	\]
	proving (a).\\
	
	Since $O^p(L) \le L^\circ$ by hypothesis, $L_\circ = O^p(L^\circ) = O^p(L)$ and therefore $O^p(\wde{L}) = \wde{L_\circ}$. Setting $D:= C_L(Y)$, (a) shows that $[L^\circ, D] \le O_p(L^\circ) \le O_p(L)$, hence $[\wde{L^\circ},\wde{D}]=1$. As $D$ is normal in $L$, $O^p(D) \le O^p(L) \cap D \le L^\circ$, yielding $O^p(\wde{D}) \le Z(\wde{D})$ and, as a consequence, that $\wde{D}$ is nilpotent. But $O_p(\wde{D}) \le O_p(\wde{L}) = 1$, so $\wde{D}$ is a $p'$-group and $\wde{D} = O^p(\wde{D}) \le O^p(\wde{L}) = \wde{L_\circ} \le \wde{L^\circ}$, whence $\wde{D} \le Z(\wde{L^\circ})$. By definition, $L^\circ$ is generated by $p$-elements, thus $\wde{L^\circ} = O^{p'}(\wde{L^\circ})$ and \cite[Lemma 1.7(a),(b)]{Meier-Stell-Stroth:2012} shows $Z(\wde{L^\circ}) \le \Phi(\wde{L^\circ})  =\Phi(\wde{L_\circ})$, which is (b).\\
	
	Suppose that $Z:= Z(Q) \cap Z(Q)^g \ne 1$; as $\L$ is $Q$-replete, $Z \in \Delta$, so $N_\L(Z)$ is a well-defined subgroup of $\L$ containing both $Q$ and $Q^g$. By \ref{lemma:weakly closed subgroups in localities}(a), $Q$ is weakly closed in $N_\L(Z)$ and by $(Q!)$, $N_\L(Z) \le N_\L(Q)$, so $Q$ is normal in $N_\L(Z)$ and we conclude that $QQ^g$ is a $p$-subgroup of $N_\L(Z)$. Moreover, there exists $U \le Z(Q)$ such that $U^g=Z$ and $U \in \Delta$ by $Q$-repleteness.\\
	Let $m \in \L$ be such that $Z^m \le S$ is fully normalized; since $Z \in \Delta$ and $Q$ is weakly closed we may choose $m \in N_\L(Q) = N_\L(Z(Q))$, thus $(g,m) \in D(\L)$ via $(U,Z,Z^m)$ and $(QQ^g)^m =Q Q^{gm} \le N_\L(Z^m)$. In particular, there exists $n \in N_\L(Z^m)$ such that $(QQ^{gm})^n \le S$ and $(g,m,n) \in D(\L)$ via $(U,Z,Z^m,Z^m)$. Thus $Q^n Q^{gmn} \le S$ and, $Q$ being weakly closed in $\L$, we obtain $Q^n = Q = Q^{gmn}$. In particular, $Q=Q^{gm}$, so recalling that $m \in N_\L(Q)$ we have $Q=Q^{m\inv} =Q^g$.
\end{proof}

\begin{theorem}[Main reduction]\label{theorem: main reduction}
	Let $L \le \L$ be a parabolic subgroup. Then the following hold.
	\begin{enumerate}[(a)]
		\item There exist $M \in \gM_\L(S)$ and a parabolic subgroup $L^* \le \L$ such that
		\[
		L^* \le M, \; LC_\L(Y_L) = L^*C_\L(Y_L), \; \text{and} \; Y_L=Y_{L^*} \le Y_M.
		\]
		Moreover, $L^\circ = (L^*)^\circ \le M^\circ$.
		\item Suppose $Q \not \norm L$ and pick $M$ and $L^*$ as in (a). Then $Q \not \norm L^*$ and $Q \not \norm M$.
		\item One of the following holds:
		\begin{enumerate}[(i)]
			\item $\M_\L(S) = \{N_\L(Q) \}$, or
			\item there exists $M \in \gM_\L(S)$ such that $Q \not \norm M$.
		\end{enumerate}
	\end{enumerate}
\end{theorem}
\begin{proof}
	\begin{enumerate}[(a)]
		\item Since $L$ is parabolic, $Y_L \in \Delta$. Let $L_1:=LC_\L(Y_L) \le N_\L(Y_L)$; we will first show that (a) holds for $L$ provided it holds for $L_1$.\\
		Clearly $Y_L$ is a $p$-reduced elementary abelian normal $p$-subgroup of $L_1$, so $Y_L \le Y_{L_1} \in \Delta$ and therefore $C_\L(Y_{L_1}) \le C_\L(Y_L)$. If (a) holds with $L_1$ in place of $L$ there exist $M \in \gM_\L(S)$ and a parabolic subgroup $L_1^* \le \L$ such that
		\[
		L_1^* \le M, \; Y_{L_1} = Y_{L_1^*} \text{ and } L_1C_\L(Y_{L_1}) = L_1^* C_\L(Y_{L_1}).
		\]
		Since $L_1 \le L_1C_\L(Y_{L_1}) \le L_1 C_\L(Y_L) = L_1$, we get $L_1^* \le L_1$ and therefore $L_1^* C_\L(Y_L) \le L_1  \le L_1^* C_\L(Y_{L_1}) \le L_1^* C_\L(Y_L)$. Thus
		\[
		LC_\L(Y_L) = L_1 = L_1^*C_\L(Y_L).
		\]
		Set $L^*:= N_{L_1^*}(C_S(Y_L))$; \ref{lemma:properties of parabolics in L}(j) with $(L_1^*,L_1)$ in place of $(M,L)$ yields $Y_L = Y_{L^*}$ and $LC_\L(Y_L)=L^*C_\L(Y_L)$. Clearly $S \le L^*$, thus $L^*$ has characteristic $p$; since $Y_{L^*} =Y_L \le Y_{L_1} \le Y_M$, then $L^*$ and $M$ satisfy
		\[
		L^* \le M, \quad Y_L=Y_{L^*} \text{ and } LC_\L(Y_L) = L^*C_\L(Y_L),
		\]
		which is (a) for $L$, apart from proving the last assertion.\\
		
		We may therefore prove (a) for a pair $(L,M)$ satisfying $C_\L(Y_L) \le L$.\\
		Let $F_\S$ be the unique factorization family for $\L$, as in \ref{lemma:unique-factor-family}; recall that it means that:
		\begin{itemize}
			\item[(i)] for every parabolic $K \le \L$ there exists $F \in F_\S$ with $K \subseteq C_\L(Y_K)F$ and $Y_K \le Y_F$;
			\item[(ii)] if a parabolic $K$ and $F \in F_\S$ are such that $F \subseteq C_\L(Y_F)K$ and $Y_F \le Y_K$, then $Y_K=Y_F$ and $K \le F$.
		\end{itemize}
		We therefore choose $F$ according to (i), so such that $L \subseteq C_\L(Y_L)F$ and $Y_L \le Y_F$. Now (i) together with $C_\L(Y_L) \le L$ imply $L \subseteq C_L(Y_L)F$. Thus
		\[
		L = C_L(Y_L)(L \cap F)
		\]
		and we may apply \ref{lemma:properties of parabolics in L}(g) to $(L, L \cap F, Y_L)$ in place of $(M,L,Y)$ and obtain $Y_L = Y_{L \cap F}$.\\
		Choose $M \le F$ minimal such that $S \le M$ and $F=MC_F(Y_F)$ and let $L^*:=L \cap M$. Since $F$ is a parabolic subgroup, it has characteristic $p$, so \ref{lemma:unique largest p-reduced} applied to $(L \cap F, F)$ gives $Y_L=Y_{L \cap F} \le Y_F$. Thus $C_F(Y_F) \le C_\L(Y_L) \le L$ and we can write
		\[
		F=M C_{(F \cap L)}(Y_L) \quad \text{and} \quad F \cap L = (M \cap L)C_{F \cap L}(Y_L),
		\]
		yielding
		\[
		LC_\L(Y_L) =L =  (L \cap F)C_\L(Y_L) = (L \cap M) C_\L(Y_L) = L^* C_\L(Y_L).
		\]
		Then $Y_L$ is $p$-reduced for $L \cap M$, so $Y_L \le Y_{L \cap M}$ and \ref{lemma:properties of parabolics in L}(g) applied to $(L \cap M, L, Y_L)$ in place of $(L,M,Y)$ shows that $Y_{L \cap M} = Y_L$. In addition, \ref{lemma:unique largest p-reduced} applied to $(L \cap M, M)$ yields $Y_{L \cap M} \le Y_M$. Thus
		\[
		Y_L = Y_{L^*} \le Y_M
		\]
		and all is left to show is that $M$ has the basic property.\\
		Since $M \le F$ and $F = MC_F(Y_F)$, always \ref{lemma:properties of parabolics in L}(g) applied this time to $(M,F,Y_F)$ gives 
		\[
		Y_M=Y_F, \quad \text{thus } F=MC_F(Y_M).
		\]
		By \ref{lemma:M in factorization family is local} $F = N_\L(Y_F) = F^\dagger$ and is the unique maximal subgroup of $\L$ containing $M$, thus also $F = N_\L(Y_F) = N_\L(Y_M)$ proving that $ F\le M^\dagger = M C_\L(Y_M) \le N_\L(Y_M) = F$ and therefore
		\[
		F = M^\dagger \qquad \text{and} \qquad\M_\L(M) = \{ M^\dagger \}.
		\]
		This gives condition (i) of the basic property.\\
		Let $M_0 := N_M(C_S(Y_M))$; it is clearly a parabolic subgroup of $\L$. Since $C_\L(Y_F)=C_F(Y_F)$, the Frattini argument applied to $C_S(Y_F)\le C_F(Y_F) \norm F =M^\dagger$ yields
		\[
		F=C_F(Y_F) N_F(C_S(Y_F)) = C_F(Y_M)N_M(C_S(Y_M))= M_0 C_F(Y_F).
		\]
		The minimality of $M$ now shows $M_0=M$. Thus $C_S(Y_M)$, which is a Sylow $p$-subgroup of $C_M(Y_M)$, is normal in $C_M(Y_M)$; hence $C_M(Y_M)$ is $p$-closed with Sylow $p$-subgroup $C_S(Y_M)=O_p(M)$, which follows from $Y_M$ being $p$-reduced.\\
		Let now $X$ be a maximal subgroup of $M$ with $O_p(M) \le X$ and suppose $XC_M(Y_M) = M$. Then $F=MC_F(Y_F) = X C_M(Y_F) C_F(Y_F) = X C_F(Y_F)$. Note that $X$ contains a Sylow $p$-subgroup of $M$, as $XC_M(Y_M)=M$, and we may suppose $S \le X$. This now contradicts the minimality of $M$, thus $XC_M(Y_M) \ne M$, proving $C_M(Y_M) \le X$ which is condition (ii) of the basic property.\\
		
		We have proven then (a) apart from the last statement; since $Y_L=Y_{L^*}$ the claim follows easily from \ref{lemma: L containing Q}(a), recalling that $L^* \le M$ and noting
		\[
		L^\circ = (LC_\L(Y_L))^\circ = (L^*C_\L(Y_L))^\circ = (L^*)^\circ \le M^\circ.
		\]
		\item Since $\Omega Z(S) \le Y_L \cap C_\L(Q)$, $C_{Y_L}(Q) \ne 1$ and $Q!$ yields
		\[
		C_\L(Y_L) \le C_\L(C_{Y_L}(Q)) \le N_\L(C_{Y_L}(Q)) \le N_\L(Q).
		\]
		As $LC_\L(Y_L) = L^*C_\L(Y_L)$, then $Q \not \norm L^*$ and thus also $Q \not \norm M$.
		\item Suppose $\M_\L(S) \ne \{N_\L(Q)\}$. Then there exists a maximal subgroup $L \in \M_\L(S)$ distinct from $N_\L(Q)$, thus with $Q \not \norm L$. Then (a) and (b) yield the existence of $M \in \gM_\L(S)$ with $Q \not \norm M$.
	\end{enumerate}
\end{proof}

The result above justifies the assumptions on $\L$ given in \ref{hyp on the locality}, which we recall here and will assume, possibly without further notice, for the rest of the paper.

\hypothesisLocality*

Again, \ref{lemma:large for fus systems and localities}(b) guarantees that $\L$ is a linking locality for $\F:=\F_S(\L)$. \\

We now can prove our translation of the hypothesis regarding the subgroup $M \le \L$ in terms of solely the fusion system, that is Theorem~\ref{theorem:translation of hyp on local to fus systems}.
\begin{proof}
	Suppose that $\L$ satisfies hypothesis~\ref{hyp on the locality}; then saturation of $\F$ follows from \ref{lemma:large for fus systems and localities}(a).\\ Since $M$ is parabolic, $S \le M \le N_\L(Y)$ for some $Y \in \Delta$, hence $Y \norm S$; as $Q \not \norm M$ implies $Q \not \norm N_\L(Y)$ we may replace $M$ with $N_\L(Y)$. Since $N_\L(Y) \le N_\L(O_p(N_\L(Y)))$, we may also assume $Y = O_p(N_\L(Y))$; since $\L$ is objective characteristic $p$ we obtain $C_\L(Y) = Z(Y)$. Suppose that $Q Z(Y) \norm N_\L(Y)$, thus $N_\L(Y) \le N_\L(QZ(Y))$. Moreover, the fact that $Z(QZ(Y)) \le N_\L(Y) \cap Z(Q)$ together with $Q!$ yields
		\[
			N_\L(Y) \le N_\L(QZ(Y)) \le N_\L(Z(QZ(Y))) \le N_\L(Q)
		\]
	proving that then $Q \norm N_\L(Y)$, a contradiction. We conclude that $QZ(Y) = QC_\L(Y) \not \norm N_\L(Y)$.\\
	Consider now the group $N_\L(Y) / C_\L(Y) \cong \Aut_\F(Y)$. Since $Y \norm S$ there is an induced map $\xi \colon S \to \Aut_\F(Y)$ (defined by taking conjugations) and the condition $Q \not \norm N_\L(Y)$ is then equivalent to $\xi(Q)\not \norm \Aut_F(Y)$. Observe in addition that $Y$ is $\F$-centric.\\
	
	Viceversa, assume $\F$ to be a saturated fusion system on the $p$-group $S$ with a large subgroup $Q$ and let $\L=(\L,\Delta,S)$ be a $Q$-replete linking locality for $\F$. Such a linking locality $\L$ exists and $Q$ is large in it by \ref{lemma:large for fus systems and localities}. Suppose additionally that there exists an $\F$-centric subgroup $Y$ such that $Y \norm S$ and $\xi(Q) \not \norm \Aut_\F(Y)$, where $\xi \colon S \fto \Aut_\F(Y)$ is given by taking conjugations. Since $\L$ is $Q$-replete we get that $Y \in \Delta$ and so, as $Y$ is $\F$-centric, $\Aut_\F(Y) \cong N_\L(Y)/C_\L(Y) = N_\L(Y) / Z(Y)$. Hence $QZ(Y) \not \norm N_\L(Y)$, thus $Q \not \norm N_\L(Y)$.
\end{proof}

\vspace{1cm}

We conclude the section with some lemmas that will be needed.

\begin{lemma}[1.55(d)]\label{lemma:dichotomy for subgroups containing Q}
	Let $L$ be a subgroup of $\L$ such that $Q \le L$. Then either $C_\L(L^\circ)=1$ or $Q \norm L$, that is $L^\circ = Q$.
\end{lemma}
\begin{proof}
	Since $C_\L(L^\circ) \subseteq C_\L(Q)=Z(Q)$, it is a subgroup of $\L$. Suppose that $C_\L(L^\circ) \ne 1$; then $Q!$ and the fact that $C_\L(L^\circ) \le Q \le L^\circ$ imply $L \subseteq N_\L(L^\circ) \subseteq N_\L(C_\L(L^\circ)) \le N_\L(Q)$, so $Q \norm L$ and by definition we get $L^\circ = \< Q^L \> = Q$.
\end{proof}

\begin{lemma}[1.57]\label{lemma1.57}
Let $M \le \L$ be a subgroup with $Q \le M$ and let also $V \le M$ be an elementary abelian normal $p$-subgroup of $M$. Let $g \in \L$ be such that $c_g \colon M \to M^g =: M_0$ is a well-defined group homomorphism and set $Q_0:= Q^g$, $V_0:=V^g$ and $\ov{M_0} := M_0/C_{M_0}(V_0)$. Then
\begin{enumerate}[(a)]
	\item for every $M_0$-submodule $1 \ne A_0 \le C_{V_0}(Q_0)$, $N_{M_0}(A_0) \le N_{M_0}(Q_0)$;
	\item if $Q_0 \not \norm M_0$, then $V_0$ is a $Q_0!$ $M_0$-module with respect to $Q_0$ and a faithful $Q_0!$ $\ov{M_0}$-module with respect to $\ov{Q_0}$;
	\item For every $U_0 \le M_0$ transitive on $V_0$ we have $Q_0^{M_0} = Q_0^{U_0}$, hence also $M_0^\circ = U_0^\circ$.
\end{enumerate}
\end{lemma}
\begin{proof}
Note that $C_{V_0}(Q_0)$ and $N_{M_0}(A_0)$ are all well-defined subgroups of $M_0$. In particular, since we may conjugate first in $M$ via $g\inv$ and then back in $M^g$ via $g$, we only need to prove the assertion when $g=1$. 
\begin{enumerate}[(a)]
	\item Since $Q$ is large in $\L$, $A \le Z(Q)$, so that (a) is an immediate consequence of the $Q!$ property.
	\item (a) and $Q \not \norm M$ give the definition of a $Q!$-module. Clearly $V$ is a faithful $\ov{M}$-module; Lemma \cite[A.51]{Meier-Stell-Stroth:2012} now shows that $V$ is a $Q!$-module for $\ov{M}$ with respect to $\ov{Q}$.
	\item Since $U$ acts transitively on $V$, for any $1 \ne v \in V$, $M=U C_M(v)$. Pick $v \in C_V(Q)$; as $Q$ is large, $v \in Z(Q)$, so always by $Q!$ we have $C_M(v) \le N_M(Q)$ and therefore $M = U N_M(Q)$. This proves that $Q^M = Q^U$, thus $M^\circ = \< Q^M \> = \< Q^U\> = U^\circ$. 
\end{enumerate}
\end{proof}

\begin{lemma}[1.6]\label{lemma1.6}
	Let $(\L,\Delta,S)$ be any locality, $H \in \C_\L$ and $K\le H$ be such that $O_p(H) \le K$. Then
	\[
	\C_H(K) = \{L \le \L \mid K \le L \le H \} = \{ L \in \C_\L(K) \mid L \le H \}.
	\]
\end{lemma}
\begin{proof}
	Let $L \le H$; if $C_\L(O_p(L)) \le O_p(L)$, then clearly also $C_H(O_p(L)) \le O_p(L)$, thus the following inclusions are trivial
	\[
	\{ L \in \C_\L(K) \mid L \le H \} \subseteq \C_H(K) \subseteq \{L \le \L \mid K \le L \le H \}.
	\]
	Let then $K \le L \le H$; since $O_p(H) \le K \le L \le H \in \C_\L$ we get
	\[
	C_\L(O_p(L)) \subseteq C_\L(O_p(H)) \subseteq O_p(H) \le O_p(L),
	\]
	showing that $\{L \le \L \mid K \le L \le H \} \subseteq \{L \le \L \mid K \le L \le H \}$.
\end{proof}

\begin{lemma}[2.2]\label{lemma2.2}
	Suppose as usual that $\L=(\L,\Delta,S)$ is a locality satisfying hypothesis \ref{hyp on the locality} and let $M \in \gM_\L(S)$. Then the following hold.
	\begin{enumerate}[(a)]
		\item $M \in \C_\L$, that is $C_\L(O_p(M)) \le O_p(M)$.
		\item $Q$ is weakly closed in $\L$ and in every subgroup of $\L$ in which it is contained.
		\item If $1 \ne K \norm M$, then $O_p(K) \in \Delta$ and $N_\L(O_p(K))\le M^\dagger$; in particular, $N_\L(O_p(M)) \le M^\dagger$.
		\item $M^\dagger = N_\L(Y_M)$.
		\item $Y_M = \Omega Z(O_p(M))$.
		\item $C_S(Y_M) = O_p(M) \in \Syl_p(C_\L(Y_M))$. In particular, $C_\L(Y_M)$ is $p$-closed if and only if $N_\L(O_p(M))=M^\dagger$.
		\item $Q \norm M \; \iff \; Q \norm M^\dagger$,
		\item $(M^\dagger)^\circ =\wcl_{M^\dagger}(Q,M) = M^\circ$.
	\end{enumerate}
\end{lemma}
\begin{proof}
	\begin{enumerate}[(a)]
		\item Since $1 \ne O_p(M) \norm S$, $O_p(M) \in \Delta$ by \ref{lemma:repleteness implies Y in Delta for parabolics} and so $N:=N_\L(O_p(M))$ is a group of characteristic $p$; as $O_p(N) \le O_p(M)$ one has 
		\[
		C_N(O_p(M)) \le C_N(O_p(N)) \le O_p(N) \le O_p(M),
		\]
		so by \ref{lemma1.6} $M \in \C_{N}(M) = \{L \in \C_\L(M) \mid L \le N \}$, in particular $C_\L(O_p(M)) \le O_p(M)$.
		\item This is \ref{lemma:large implies weakly closed} and \ref{lemma:weakly closed subgroups in localities}(a).
		\item As $K \norm M$, by \ref{lemma:properties groups of char p}(c) $K$ has characteristic $p$ and, in particular, $O_p(K) \ne 1$. Since $O_p(K) \norm M$, it is a normal subgroup of $S$, hence $O_p(K) \in \Delta$ and $N:=N_\L(O_p(K))$ is a subgroup of $\L$. We also get $C_\L(O_p(N)) \le C_\L(O_p(K)) \le O_p(K) \le O_p(N)$, so $N_\L(O_p(K)) \in \C_\L$. Since $\M_\L(M) = \{M^\dagger\}$ and $M \le N$, we get $N_\L(O_p(K)) \le M^\dagger$.\\
		By taking $K:=M$ we obtain $N_\L(O_p(M)) \le M^\dagger$.
		\item Clearly $M^\dagger = M C_\L(Y_M) \le N_\L(Y_M)$. Now \ref{lemma:unique-factor-family} and the proof of \ref{theorem: main reduction} (or otherwise (c)) give the other inclusion.
		\item By the basic property $\{C_S(Y_M)\} = \Syl_p(C_M(Y_M)$, thus $O_p(M) = C_S(Y_M)$ and $Y_M = \Omega Z(O_p(M))$ by \ref{lemma:properties of parabolics in L}(j).
		\item That $O_p(M) = C_S(Y_M)$ was already noted in (e). As $S \in \Syl_p(N_\L(Y_M))$, $O_p(M) =C_S(Y_M) \in \Syl_p(C_\L(Y_M))$.\\
		By (c) and the definition of $M^\dagger$ we have $N_\L(O_p(M)) \le M^\dagger = MC_\L(Y_M)$; then the equality holds if and only if $C_\L(Y_M)$ normalizes $O_p(M)$, which in turn holds if and only if $O_p(M)$ is the unique Sylow $p$-subgroup of $C_\L(Y_M)$.
		\item We only need to show that $C_\L(Y_M)$ normalizes $Q$. Since $Y_M \norm S$, $Y_M \cap Z(S) \ne 1$; as $Q$ is large, $Y_M \cap Z(S) \le Z(Q)$ and so, by $Q!$, $C_\L(Y_M) \le N_\L(Y_M \cap Z(S)) \le N_\L(Q)$. 
		\item As for (g), $C_\L(Y_M) \le N_\L(Q)$, hence $Q^{M^\dagger} = Q^M$. Since $Q$ is large in $\L$ and contained in $M^\dagger$, it is trivially also large, thus weakly closed, in $M^\dagger$ as well. Thus \cite[Lemma 1.46(b) and (c)]{Meier-Stell-Stroth:2012} yield $\wcl_{M^\dagger}(Q,M) = \< Q^M \> = \<Q^{M^\dagger}\>$.
	\end{enumerate}
\end{proof}

For future reference we report the statement of the $Q!FF$-module theorem.

\begin{theorem}[$Q!FF$-Module Theorem]\label{thm:Q!FF}
	Let $H$ be a finite group such that $O_p(H) = 1$, $Q \le H$ a $p$-subgroup and $V$ a faithful $Q!$-module for $H$. Set $H^\circ := \< Q^H \>$ and $J := J_H(V)$ and suppose that there exists an offender $Y \le H$ on $V$ such that $[H^\circ,Y] \ne 1$ and that one of the following holds:
	\begin{enumerate}[(i)]
		\item $Y$ acts quadratically on $V$,
		\item $Y$ is a best offender on $V$,
		\item $C_Y([V,V]) \ne 1$,
		\item $C_Y(H^\circ) =1$.
	\end{enumerate}
	Then one of the following cases holds.
	\begin{enumerate}[(1)]
		\item $V$ is an $\SL_2(q)$-wreath product module; more in  detail, there exists an $H$-invariant set $\K$ of subgroups of $H$ satisfying the following:
		\begin{enumerate}[(a)]
			\item for every $K \in \K$, $K \cong \SL_2(q)$ and $[V,K]$ is a natural module for $K$;
			\item $\displaystyle J = \bigtimes_{K \in \K} K$ and $\displaystyle V=\bigoplus_{K \in \K} [V,K]$;
			\item $Q$ is transitive on $\K$;
			\item $H^\circ = O^p(J)Q$.
		\end{enumerate}
		\item Let $R:=F^*(J)$. Then the following hold.
		\begin{enumerate}[(a)]
			\item $R$ is a quasisimple subgroup of $H^\circ$ and either $J=R$ or $p=2$ and $J \in \{\O^{\pm}_{2n}(q), \, \Sp_4(2), \, G_2(2) \}$.
			\item $C_V(R) = 0$, $[V,R]$ is semisimple as $J$-module and faithful as $H$-module.
			\item Put $J^0 := J \cap H^\circ$. Then one of the following holds:
			\begin{enumerate}[(i)]
				\item \begin{itemize}
					\item $R=J^0  \cong 
					\begin{cases*}
						\SL_n(q) & \textit{ for $n \ge 3$}\\
						\Sp_{2n}(q) & \textit{ for $n \ge 3$}\\
						\SU_n(q) & \textit{ for $n \ge 8$}\\
						\Omega^\pm _n(q) & \textit{ for $n \ge 10$}
					\end{cases*} $
					\item $[V,R]$ is the direct sum of at least two isomorphic natural modules for $R$.
					\item $H^\circ = R C_{H^\circ}(R)$.
					\item if $V \ne [V,R]$, then $R \cong \Sp_{2n}(2^k)$ for $n \ge 4$.					
				\end{itemize}
				\item \begin{itemize}
					\item $[V,R]$ is simple as $R$-module,
					\item either $H^\circ = R = J^0$ or one has $H^\circ = 
					\begin{cases*}
						\Sp_4(2) \\
						3 \cdot \Sym(6) \\
						\SU_4(q).2 \, (\cong \O^-_6(q)) \\
						G_2(2)
					\end{cases*} $ \\
					and if $H^\circ \cong \SU_4(q).2$, then $[V,R]$ is the natural $\SU_4(q)$-module,
					\item one of the cases (1) to (9) or (12) of the $FF$-Module Theorem, see \cite[Theorem C.3]{Meier-Stell-Stroth:2012} or \cite{MS:2012a}, where $n$ is taken satisfying $n \ge 3$ in case (1), $n \ge 2$ in case (2) and $n=6$ in case (12).
				\end{itemize}
				\item $p = 2$, $J=R\cong \SL_4(q)$, $|H^\circ /R|=2$, it induces a graph automorphism on $R$ and $V$ is the direct sum of two non-isomorphic natural modules for $R$.
			\end{enumerate}
		\end{enumerate}
	\end{enumerate}
\end{theorem}
\begin{proof}
	It is proven mainly in \cite[Theorem 4.6]{MS:2012b}, with small differences treated in \cite[Theorem C.24]{Meier-Stell-Stroth:2012}.
\end{proof}

Further preliminary results needed for the analysis of each case will be provide within the study of each case. There is, however, the case in which $Y_M$ is the natural orthogonal module for $O^\epsilon_{2n}(2)$ that requires a separate, specific reasoning already in \cite{Meier-Stell-Stroth:2012}. This is true for the our locality $\L$ as well and we therefore begin with dealing with this situation.

\section{The main case history and the orthogonal case}
\label{case history and orthog case}

We suggest a case subdivision modeled on the definitions seen in \ref{def:symm-tall-groups} and on the case history in \cite{Meier-Stell-Stroth:2012}. As seen in section \ref{section: strategy of LST for groups}, the notions of tallness and shortness are particularly meaningful with respect to a set $\N$ of subgroups which is invariant under conjugation. This, however, can be quite problematic in our locality, since conjugation is, in general, not a transitive relation.\\
The definitions for a locality are thought in a way to circumvent this problem, but it demands a number of results that reduce ``within $S$" certain configurations; the possibility of restricting attention to a chosen Sylow $p$-subgroup was, indeed, often noticed already in the work of Meierfrankenfeld, Stellmacher and Stroth, even though rarely necessary.\\

\subsection{The case subdivision in a locality}

We continue to assume Hypothesis \ref{hyp on the locality} for $(\L,\Delta,S) $ and for $Q \le S$.

\begin{definition} \label{def:symm-tall in a locality}
	Let $A \in \Delta$ be an abelian $p$-group and $\N$ be a set of subgroups of $\L$ such that every $N \in \N$ satisfies $O_p(N) \ne 1$.
	\begin{itemize}
		\item $A$ is {\bf symmetric} in $\L$ if there exists $x \in \L$ such that:
			\begin{itemize}
				\item[(i)] $A \subseteq S_x$,
				\item[(ii)] $1 \ne [A,A^x] \le A \cap A^x$.	
			\end{itemize}
			Otherwise $A$ is said {\bf asymmetric} in $\L$.\\
		\item $A$ is {\bf $\N$-tall} if $C_S(A) \in \Syl_p(C_\L(A))$ and there exist $L \in \N$ such that
		\[
		C_S(A) \le L \quad \text{and} \quad A \not \le O_p(L).
		\]
		$A$ is {\bf $\N$-short} if $C_S(A) \in \Syl_p(C_\L(A))$ and $A$ is not $\N$-tall; that means that for every $L \in \N$ we have
		\[
		C_S(A) \le L \implies A \le O_p(L).
		\]
		In particular:
		\begin{itemize}
			\item $A$ is {\bf tall} (resp. {\bf short}) if it is $\N$-tall (resp. $\N$-short) for 
			\[
			\N= \{ \textit{subgroups $N$ of $\L$ such that $O_p(N) \ne 1$} \}.
			\]
			\item $A$ is {\bf char $p$-tall} (resp. {\bf char $p$-short}) if it is $\N$-tall (resp. $\N$-short) for
			\[
			\N= \{ \textit{subgroups $N \le \L$ of characteristic $p$} \}.
			\]
			\item $A$ is {\bf $Q$-tall} (resp. {\bf $Q$-short}) if it is $\N$-tall (resp. $\N$-short) for $\N= \{ N_\L(Q) \} $.
		\end{itemize}
	\end{itemize}
\end{definition}

Note that the condition $C_S(A) \in \Syl_p(C_\L(A))$, for $A \in \Delta$, is satisfied whenever $A$ is fully $\F_S(\L)$-normalized.\\

Our case subdivision will then be the following: let $M \le \L$ be a parabolic subgroup with the basic property and such that $Q \not \norm M$, then
\begin{center}
\renewcommand{\arraystretch}{1.5}
\begin{tabular}{|c|c|}
    \hline
    (\textbf{Sym}) & $Y_M$ is symmetric \\
    \hline
    (\textbf{Asym - short}) & $Y_M$ is asymmetric, short in $\L$ \\
    \hline 
    (\textbf{Asym - tall, char p-short})  & $Y_M$ is asymmetric and tall, but char $p$-short \\
    \hline
    (\textbf{Asym - char p-tall, Q-short}) & $Y_M$ is asymmetric, char $p$-tall, but $Q$-short \\
    \hline
    (\textbf{Asym - Q-tall}) &  $Y_M$ is $Q$-tall \\
    \hline
\end{tabular}
\end{center}

\subsection{The orthogonal case}

We wish to produce a result for the locality $\L$ similar to \cite[Theorem C]{Meier-Stell-Stroth:2012}, which deals with occurrences of natural $\O^\epsilon_{2n}(2)$-modules and requires a separate, quite lengthy reasoning. This is mostly due to the fact that the natural simple $\O^\epsilon_{2n}(2)$-module, say $V$, is the only $Q!$-module admitting a non-trivial offender $A$ such that $[V,A]$ contains no non-trivial $2$-central element of $M$. In other words, every non-trivial vector in $[V,A]$ is non-singular.\\
Surprisingly, the result for our locality $(\L,\Delta,S)$, whose proof mimics that given in \cite{Meier-Stell-Stroth:2012}, allows to deduce global properties of the entire locality $\L$, and therefore of the associated fusion system $\F_S(\L)$, only from the existence of such an orthogonal module.\\
We start with reporting the following lemma, which describes the natural $\O^\epsilon_{2n}(2)$-module.

\begin{lemma}\label{lemma3.1 - orthogonal module}
    Let $V$ be a natural $\O^\epsilon_{2n}(2)$-module for $X=\O^\epsilon_{2n}(2)$ and $n \ge 2$. Let $0 \ne z \in V$ be a singular vector and $y \in V$ be a non-singular vector. Then the following hold.
    \begin{enumerate}[(a)]
        \item $X$ is transitive on the set of non-trivial singular vectors and on the set of non-singular vectors.
        \item $C_X(z)=AK$ where
        \begin{enumerate}[(1)]
            \item $K \cong \O^\epsilon_{2n-2}(2)$ and $A$ is a natural $\O^\epsilon_{2n-2}(2)$-module for $K$;
            \item $[z^\perp,A]=\<z\>$, $C_X(z^\perp) = 1$ and $A$ induces $\Hom_{\bF_2}(z^\perp/\<z\>, \<z\>)$ on $z^\perp$;
            \item $C_X(z)$ is a parabolic subgroup of $X$;
            \item if $(2n,\epsilon) \ne (4,+)$, then $\O_2(C_X(z)) = A \le \Omega^\epsilon_{2n}(2)$.
        \end{enumerate}
        \item $C_X(y) =T \times E$ where
        \begin{enumerate}[(1)]
            \item $T \cong C_2$, $E \cong \Sp_{2n-2}(2)$, $y^\perp$ is a natural $\O_{2n-1}(2)$-module for $E$ and $y^\perp/\<y\>$ is a natural $\Sp_{2n-2}(2)$-module for $E$;
            \item $T = C_X(y^\perp)$, $[V,T]=\<y\>$, $y^\perp=C_X(T)$ and $T \not \le \Omega^\epsilon_{2n}(2)$;
            \item if $\Z$ is the set on non-trivial singular vectors of $y^\perp$, then $y^\perp= \< \Z \>$ and $E$ acts on $\Z$ transitively.
        \end{enumerate}
        \item Let $0 \ne v \in V$ and assume $v$ singular if $X = \O^+_4(2)$. Then $C_X(v)$ is a maximal subgroup of $X$.
    \end{enumerate}
\end{lemma}
\begin{proof}
    This is \cite[Lemma 3.1]{Meier-Stell-Stroth:2012}.
\end{proof}

We also report the statement of \cite[Theorem C]{Meier-Stell-Stroth:2012} for the reader's convenience.

\begin{maintheorem}{C}(Meierfrankenfeld, Stellmacher, Stroth)
	\label{theorem:thm C MSS}
	Let $G$ be a finite $\K_2$-group and $S \in \Syl_2(G)$, and let $Q \le S$ be a large $2$-subgroup of $G$. Let $M \in \gM_G(S)$ and suppose that the following hold:
	\begin{itemize}
		\item[(i)] $M/ C_M(Y_M) \cong \O^\epsilon_{2n}(2)$ with $n \ge 2$,
		\item[(ii)] $[Y_M , O^2(M)]$ is a natural $\O^\epsilon_{2n}(2)$-module for $M/C_M(Y_M)$,
		\item[(iii)] $C_G(y) \not \le M^\dagger$ for all non-singular elements $y \in [Y_M , O^2(M)]$,
		\item[(iv)] $Q \not \norm M$.
	\end{itemize}
	Then $C_G(y)$ is not of characteristic $2$ for all non-singular elements $y \in [Y_M , O^2(M)]$.
\end{maintheorem}

Now note that $M$ acts transitively on the set of non-singular vectors, as well as on the set of non-trivial singular vectors, for example by \cite[Lemma 3.1(a)]{Meier-Stell-Stroth:2012}; suppose that $C_G(u) \le M^\dagger$ for some non-singular element $u \in [Y_M,O^2(M)]$, then for every $m \in M$ one has $C_G(u^m) \le M^\dagger$, hence all non-singular elements would fail hypothesis (iii). Similar considerations hold for the conclusion of the theorem, hence we may replace hypothesis (iii) with 
\[
\textit{(iii') $\qquad C_G(y) \not \le M^\dagger$ for some non-singular element $y \in [Y_M , O^2(M)].$}
\]

When turning our attention to the locality $\L$ we prove the following.

\begin{maintheorem}{C$\L$}\label{theorem CL}
	Let $S$ be a finite $2$-group and $\L=(\L,\Delta,S)$ be a finite $\K_2$-linking locality satisfying hypothesis \ref{hyp on the locality}. Let $M \in \gM_\L(S)$ and suppose that the following hold:
	\begin{itemize}
		\item[(i)] $M/ C_M(Y_M) \cong \O^\epsilon_{2n}(2)$ with $n \ge 2$,
		\item[(ii)] $[Y_M , O^2(M)]$ is a natural $\O^\epsilon_{2n}(2)$-module for $M/C_M(Y_M)$,
		\item[(iii')] $C_\L(y) \not \subseteq M^\dagger$ for some non-singular element $y \in [Y_M , O^2(M)]$,
		\item[(iv)] $Q \not \norm M$,
	\end{itemize}
	Then $\<y \> \not \in \Delta$ for all non-singular elements $y \in [Y_M , O^2(M)]$.
\end{maintheorem}

Here it is not clear whether (iii') is equivalent to 
\[
    \textit{(iii) $\qquad C_\L(y) \not \subseteq M^\dagger$ for all non-singular elements in $[Y_M,O^2(M)].$}
\]
Indeed, there is no guarantee that the conjugation maps induced by elements of $M$ on the non-singular vectors of $[Y_M,O^2(M)]$ extend to their full centralizers; thus, a priori some elements might have smaller centralizer, possibly falling in $M^\dagger$.\\

We then get the following easy corollary, which gives global information about the locality (hence about the fusion systems) depending only on the (local) presence of an orthogonal module.

\begin{cor}\label{corollary to Theorem CL}
    Let $(\L,\Delta,S)$ and $M$ be as in Theorem~\ref{theorem CL} and set $\F:=\F_S(\L)$. Then the elements of $U:=[Y_M,O^2(M)]^\natural$ are split in exactly two $\F$-conjugacy classes:
    \begin{align*}
        z^\F &= \{ u \in U \mid \textit{$u$ is singular for $M$} \} = \{u \in U \mid \<u\> \in \Delta \}, \\
        y^\F &= \{ u \in U \mid  \textit{$u$ is non-singular for $M$} \} = \{u \in U \mid \<u\> \not \in \Delta \}.
    \end{align*}
	In particular, $\F$ is not of characteristic $2$-type.
\end{cor}
Recall that $\F$ remains, however, of parabolic characteristic $p$.
\begin{proof}
Observe that, since any singular vector is $2$-central in $M$, there exists a non-trivial singular element $z \in \Omega Z(S) \le C_S(Q) = Z(Q)$; by $Q$-repleteness we get $\<z\> \in \Delta$ and, since all the non-trivial singular vectors are $M$-conjugate and $[Y_M,O^2(M)] \le S$, we get that $\< u \> \in \Delta$ for all the non-trivial singular vectors. By Theorem~\ref{theorem CL} all non-singular vectors $u \in U$ are such that $\<u\> \not \in \Delta$, thus
\[
\{u \in U \mid \textit{u is $M$-singular} \} = z^M \le z^\F \le \{z \in U \mid \< z \> \in \Delta \} \le \{u \in U \mid \textit{u is $M$-singular} \},
\]
proving the equalities of the first line.\\
The second-line equalities now follow since also all the non-singular vectors are $M$-conjugate and $U$ is partitioned as
\[
    U = \{ u \in U \mid \textit{$u$ is $M$-singular} \} \sqcup \{ u \in U \mid  \textit{$u$ is non-singular for $M$} \} = \{u \in U \mid \<u\> \in \Delta \} \sqcup \{u \in U \mid \<u\> \not \in \Delta \}.
\]
Since we may assume $\L$ to be the subcentric linking locality and we can find non-trivial subgroups of $S$ not contained in $\Delta$, \ref{cor:locality transl of char p type} implies that $\F$ is not of characteristic $2$-type.
\end{proof}

\begin{notation}\label{notation: orthogonal case}
	As usual $\ov{M^\dagger} := M^\dagger / C_{M^\dagger}(Y_M)$ and $Z_M := \< \Omega Z(X) \mid X \in \Syl_2(M) \>$. By hypothesis, $[Y_M, O^2(M)]$ is a natural $\O^\epsilon_{2n}(2)$-module for $M$; in particular, pick $1 \ne y,z \in [Y_M,O^2(M)]$ such that
	\[
	\text{$z$ is singular, $y$ is non-singular and $y \perp z$.}
	\]
	Since $z$ is $2$-central we may actually choose the elements so that $z \in \Omega Z(S)$ and $C_S(y) \in \Syl_2(C_M(y))$.

    \begin{lemma}\label{lemma:uniquen of some conjug of Q-orthog case}
		For any $g \in \L$ such that $\<z^g\> \le Y_M$ there exists a unique $M$-conjugate $Q_0$ of $Q$ such that $[z^g,Q_0]=1$.
	\end{lemma}
	\begin{proof}
	Since $1 \ne z \in \Omega Z(S)$ we see that $z \in Z(Q)$, so one has $C_\L(z) \le N_\L(Q)$ by $Q!$ and $\< z \> \in \Delta$ by $Q$-repleteness; then also every $\L$-conjugate of $z$ that stays in $Y_M$ generates a subgroup of $\Delta$.\\
	Let $g \in \L$ be such that $\<z^g\> \le Y_M$; then conjugation by $g$ extends to an isomorphism $C_\L(z) \fto C_\L(z^g)$, thus also to $Q \fto Q^g$. Suppose that there exists $m\inv \in M$ such that 
	\[
	[z^g,Q^g]=1=[z^g,Q^{m\inv}];
	\]
	then $(z^g)^m \in C_\L(Q)=Z(Q) \le S$, so $(g,m) \in D(\L)$ via $(\<z\>,\<z^g\>,\<(z^g)^m\> )$ and conjugation by $m$ extends to $C_\L(z^g)$, so also to $Q^g$. We get
	\[
	[z^{gm},Q^{gm}]=1=[z^{gm},Q]
	\]
	so $1 \ne z^{gm} \le Z(Q) \cap Z(Q^{gm})$. Now \ref{lemma: L containing Q}(c) implies $gm \in N_\L(Q)$, hence $Q^g = Q^{m\inv}$.		
	\end{proof}
	According to the previous lemma, for any such $g \in \L$ we may define
	\[
	Q_{z^g} := Q^g
	\]
    and, at need, replace $g$ by an element $m \in M$ with $Q^g=Q^m$. Set also
	\[
	F_0 := C_M(y), \, T^*:=C_S(y), \, Y := y^\perp \le [Y_M, O^2(M)], \, T:= C_S(Y) , \, F:= \< (Q \cap F_0)^{F_0} \>,
	\]
	and note that $y \in Y$ (since $y$ is isotropic) and so $T \le T^* \in \Syl_2(F_0)$. In addition, by our choice $z \perp y$, hence $\<z\> \le Y$ and we have $Y \in \Delta$. Similarly we also get $T,T^* \in \Delta$.
\end{notation}

\begin{remark}[Adjustments to Lemma 3.3]\label{lemma3.3} \leavevmode \\[-15pt]
	\begin{itemize}
		\item[(a)] It holds in the following shape: $C_\L(M^\circ)=1$ and $Z(M) = 1$.
		\item[(c)] We remark that, in virtue of \ref{lemma:uniquen of some conjug of Q-orthog case}, for every singular vector $u \in Z_M^\natural$ there exists a unique conjugate $Q_{u}$ of $Q$ that centralizes $u$. In particular, the map given by
		\begin{center}
			\begin{tikzcd}
				\{\textit{singular vectors of $Z_M^\natural$} \} \ar[r] & \{Q_u \mid \textit{$u \in Z_M^\natural$ is singular} \} \\[-20pt]
				v  \ar[r,maps to]		 & Q_v
			\end{tikzcd}
		\end{center}
	is bijective.
	\end{itemize}
\begin{proof}
	For (a), since, for example, $Q \le M^\circ$, $C_\L(M^\circ) \le C_\L(Q)$, so $C_\L(M^\circ)$ is a subgroup of $\L$. Since $Q \not \norm M$ by hypothesis, $M^\circ \ne Q$ and \ref{lemma:dichotomy for subgroups containing Q} yields $C_{\L}(M^\circ) = 1$. Since $Z(M) \le C_\L(M) \le C_\L(M^\circ)$ we also have $Z(M)=1$.\\
	
	For (c) note that $Z_M \le Y_M$ by \cite[Lemma 3.3(b)]{Meier-Stell-Stroth:2012} and now \ref{lemma:uniquen of some conjug of Q-orthog case} shows that the map is well-defined and, clearly, surjective. By \cite[Lemma 3.3(c)]{Meier-Stellm-Stroth:2003} $C_{Z_M}(Q_z)$ is a one dimensional, singular subspace; thus $C_{Z_M}(Q_z) =\<z\> = \{1,z\}$, which proves injectivity.\\

	The rest of the proof applies to our setting just as it appears in \cite{Meier-Stell-Stroth:2012}. References to results that needed a translation in terms of localities have been worked out. In detail:
	\begin{itemize}
		\item[(i)] reference to \cite[Lemma 1.24]{Meier-Stell-Stroth:2012} in part (b) is covered by \ref{lemma:weakly closed subgroups in localities};
		\item[(ii)] reference to \cite[Lemma 1.57]{Meier-Stell-Stroth:2012} in parts (c) and (d) is covered by \ref{lemma1.57};
		\item[(iii)] references to \cite[Lemma 2.2]{Meier-Stell-Stroth:2012} and to \cite[Lemma 1.52]{Meier-Stell-Stroth:2012} in part (h) are covered, respectively, by \ref{lemma2.2} and \ref{lemma: L containing Q};
        \item[(iv)] for the reference to \cite[Lemma 1.58]{Meier-Stell-Stroth:2012} one only needs to note that the proof works as it is with the locality $(\L,\Delta,S)$ in place of $H$ and with $L \in \C_\L(S)$. \\
        All the references in the proof either apply without further adjustment or have been modified to work in our setting.
	\end{itemize}
\end{proof}
\end{remark}

\begin{remark}[Adjustments to Lemma 3.4]\label{lemma3.4}
	The entire proof of the lemma has a local nature, dealing with the substructure of $M^\dagger=N_\L(Y_M)$, and this reflects in the proof, which is mostly self-contained and makes almost exclusively use of subgroups of $M^\dagger$. We only need to make a couple of notes.
	\begin{itemize}
		\item[(i)] As $Q_z=Q$ and $y \in S$, all the subgroups and products in (a) are well-defined.
		\item[(ii)] References to \cite[Lemma 1.52]{Meier-Stell-Stroth:2012}(c) in (b)-(d) and (k) are covered by \ref{lemma: L containing Q}(a). Note that in point (k) $C_G(Y_M)$ (which actually appear as $C_G(Y)$ during the proof due to a typo) can clearly be replaced by $C_\L(Y_M)$.
		\item[(iii)] Reference to \cite[Lemma 2.2]{Meier-Stell-Stroth:2012} in (f) is covered by \ref{lemma2.2}.
	\end{itemize}
\end{remark}

\subsection{The proof of Theorem~\ref{theorem CL}}

Even though the proof closely follows the lines of that in \cite{Meier-Stell-Stroth:2012}, we mostly provide the full reasoning in order to prevent an impair in readability, as there is a need for multiple adjustments and ad-hoc arguments to justify why the result holds in the locality.\\
As in the Local Structure Theorem, we begin by assuming that the pair $(\L,M)$ is a counterexample to Theorem~\ref{theorem CL}. Then there exists $y \in [Y_M,O^2(M)]$ non-singular and such that $\< y \> \in \Delta$. Recall that $\< z \> \in \Delta$, that $[Y_M,O^2(M)]=Z_M=\< \Omega Z(T) \mid T \in \Syl_2(M) \>$ by \ref{lemma3.3}(b) and that $M$ acts transitively both on the set of non-trivial singular vectors and on the set of non-singular vectors of $Z_M$ by \ref{lemma3.1 - orthogonal module}. Our assumption now implies that $\< x \> \in \Delta$ for any non-trivial vector $x \in Z_M$; in particular, any non-trivial subspace is an object of the locality.

\begin{lemma}[3.5]\label{lemma3.5}
	Suppose that $[O_2(M),O^2(M)] \le Y_M$. Then $Y_M=O_2(M)=C_\L(Y_M)$ and $M^\dagger =M$.
\end{lemma}
\begin{proof}
	It is verbatim the same as in \cite[Lemma 3.5]{Meier-Stell-Stroth:2012} with $C_\L(Y_M)$ in place of $C_G(Y_M)$.
\end{proof}

For this chapter we will always assume that the Baumann subgroup is taken with respect to the prime $p=2$. Note that for $E \in \{T,T^*\}$, $B(E) = \Omega Z (J(E)) \in \Delta$ since, for example, $z \in B(E)$. The next proof follows, with minor adjustments, the same argument as in \cite{Meier-Stell-Stroth:2012}, but it is quite complex and we therefore provide the entire reasoning. Here techniques under the name of ``pushing up" appear; for a more comprehensive treaty of these see \cite[Section 2.1]{Meier-Stellm-Stroth:2003} and \cite[Section 2]{PPS:2003}.

\begin{lemma}[3.6]\label{lemma3.6}
	The following hold.
	\begin{itemize}
		\item[(a)] If $2n =4$, then $N_\L(T^*) \le M^\dagger$ and $T^* \in \Syl_2(C_\L(y))$; in particular, a Sylow $2$-subgroup of $C_\L(y)$ has order strictly smaller than $|S|$.
		\item[(b)] If $(2n,\epsilon) \ne (4,+)$, then $N_\L(T) \le N_\L(B(T)) \le M^\dagger$ and $T=O_2(N_{\L}(T))$.
	\end{itemize}
\end{lemma}
\begin{proof}
	\begin{itemize}
		\item[(a)] By \cite[Lemma 3.4(a)]{Meier-Stell-Stroth:2012} there exists $q \in Q$ such that $y^q=yz$; then $Y \ne Y^q$ and we may pick $y_1 \in Y^q \setminus Y$. If $y_1$ is singular, since $yz$ is non-singular, also $y_1(yz)$ is non-singular and it is in $Y^q \setminus Y$; whence we may choose $y_1$ non-singular. By the transitivity of the action of $M$ over non-singular vectors, $y_1=y^u$ for some $u \in M$. Then the inclusion $\<y,y^u\> \le Y^q = C_{Z_M}(T^q)$ shows that $\ov{T^q} \le C_{\ov{M}}(y^u) = \ov{T^u} \times \ov{E^u}$; as $|\ov{T}|=2$, $[\ov{T^q},\ov{T^u}]=1$, so $T^q$ normalizes $T^uC_M(Y_M)$ and, therefore, also a Sylow $2$-subgroup of $T^uC_M(Y_M)$. Since $T^u \in \Syl_2(T^uC_M(Y_M))$, we may pick $u \in M$ such that $T^q$ normalizes $T^u$. For this choice we set $M_1:= \<T,T^u\>$; \cite[Lemma 3.4(h)]{Meier-Stell-Stroth:2012} now shows that $T \in \Syl_2(M_1)$, $M_1/O_2(M) \cong \Sym(3)$ and, as $2n=4$, $T^*$ normalizes $M_1$. Set $M^* := T^* M_1$; then we also get $T^* \in \Syl_2(M^*)$, so $\Sym(3) \cong M_1 / O_2(M) \lesssim M_1 T^* / O_2(M^*)$ as a normal subgroup and therefore $T^*/O_2(M^*)$ is either trivial or $C_2$. Since by \ref{lemma3.1 - orthogonal module}(c) $T^*/C_M(Z_M)$ centralizes $T/C_M(Z_M))$ we get $T^*/O_2(M^*)=1$, thus $M^* / O_2(M^*) \cong  \Sym(3)$ and so $Z(M^*)=1$ (using that $Z(M)=1$ by \cite[Lemma 3.3(a)]{Meier-Stell-Stroth:2012} and $M$ is of characteristic $2$). Note that $F^*(M^*)=O_2(M^*)$.\\
        We now want to show that $N_\L(T^*) \le M^\dagger$.\\
		
		Suppose first that $T^*$ contains no non-trivial characteristic subgroup which is also normal in $M^*$; by the Baumann argument, see for example \cite[Lemmas 2.8]{PPS:2003}, one gets $O_2(M^*) \lneq B(T^*)$ and, since $T^* / O_2(M^*) \cong C_2$, $B(T^*) \in \Syl_2(B(M^*))$. Moreover $O_2(M^*)= O_2(B(M^*))$, so also $B(M^*)=M^*$. Recalling that $M^*/O_2(M^*) \cong \Sym(3)$, by the pushing up argument for $\Sym(3) \cong \SL_2(2)$ one sees that $O_2(M^*)$ has a unique, non-central $M^*$-chief factor. Consider the sequence $1 < \dots < Y_M \le O_2(M^*)$; since $[Y_M,O^2(M^*)] \ne 1$ the uniqueness of the central factor gives $[O_2(M^*)/Y_M , O^2(M^*)] =1$, hence
		\[
		[O_2(M),O^2(M^*)] \le [O_2(M^*),O^2(M^*)] \le Y_M.
		\]
		By \cite[Lemma 3.4(h)]{Meier-Stell-Stroth:2012} $\<M_1,S \> =M$, so $[O_2(M),O^2(M)] \le Y_M$ and \ref{lemma3.5} yields $Y_M=O_2(M)=C_M(Y_M)$.	Since $Y_M = Z_M$, for example by \cite[Lemma 3.3(f)]{Meier-Stell-Stroth:2012}, the structure of the natural orthogonal module and $T^*/Y_M \cong C_2 \times C_2$ yield $\mathcal{A}_{T^*} = \{Y_M,A,A_1,A_2\}$, denoted so that $|Y_M / Y_M \cap A_i|=2$ and $|Y_M / Y_M \cap A|=4$. Thus the pairs $\{A_1,A_2\}$ and $\{Y_M,A\}$ are the only ones intersecting in a group of order $4$; then, for $N:= N_\L(T^*)$, $N/C_N(\mathcal{A}_{T^*})$ is a $2$-group. As $2n=4$, $S \le N_\L(T^*)$; then $C_N(\mathcal{A}_{T^*}) \le N_\L(Y_M)=M^\dagger$ gives $N = S \, O^2(N) \le M^\dagger$.\\
		
		Suppose now that there exists a non-trivial characteristic subgroup $K \le T^*$ normal in $M^*$. Then \cite[Lemma 3.4(h)]{Meier-Stell-Stroth:2012} implies $\<M^*,N_M(T^*)\>=M$ and the fact that $K$ is characteristic in $T^*$ yields $K \norm M$. \ref{lemma2.2}(c) then gives $N_\L(T^*) \le N_\L(K) \le M^\dagger$.\\
		
		As $C_S(Y_M) \le C_S(y)=T^* \in \Syl_2(C_M(y))$, $T^* \in \Syl_2(C_{M^\dagger}(y))$; if $T^* \le T_1 \in \Syl_2(C_\L(y))$, then $N_{T_1}(T^*) \le N_\L(T^*) \cap T_1 \le M^\dagger \le N_{C_{M^\dagger}(y)}(T^*)$. Then $N_{T_1}(T^*)=T_1$, that is $T_1=T^*$.
		
		\item[(b)] Searching for a contradiction, suppose $L:=N_\L(B(T)) \not \le M^\dagger$. If $K$ is a non-trivial, characteristic subgroup of $B(T)$ normal in $M_1$, where $M_1$ is defined as in (a), then \cite[Lemma 3.4(h)]{Meier-Stell-Stroth:2012} yields $K \norm M$ and, by \ref{lemma2.2}(c), $N_\L(B(T)) \le N_\L(K) \le M^\dagger$, contradicting our assumption. Thus we may assume that no non-trivial characteristic subgroup of $B(T)$ is normal in $M_1$.\\
        Applying once more the Baumann argument we get $O_2(M) \lneq B(T)$ and, therefore, $T=B(T) \in \Syl_2(M_1)$ and $B(M_1)=M_1$; hence, there exists a unique non-central chief factor of $M_1$ on $O_2(M_1)=O_2(M)$. Since $[Y_M,O^2(M_1)] \ne 1$ we conclude again $[O_2(M),O^2(M)] \le Y_M$ and \ref{lemma3.5} yields $Y_M=O_2(M)=C_M(Y_M)$ and $M = M^\dagger$. Pick $t \in L \setminus M^\dagger$, so that $Y_M \ne Y_M^t$; recalling that $T/C_T(Z_M) \cong C_2$ we see that $Y_MY_M^t =T=B(T)$ and $|T : Y_M \cap Y_M^t|=4$. Set $\mathrm{Inv}(T) := \{x \in T \mid x^2=1 \}$, then
        \[
        | \mathrm{Inv}(T) | \ge |Y_M| + |Y_M^t| - |Y_M \cap Y_M^t| = \frac{3}{4}|T|;
        \]
        since $T$ is non-abelian (for example, $[Y_M,O^2(M_1)] \ne 1$ and $2 \mid |O^2(M_1)|$), according to \cite[Theorem 4.1]{Edmonds-Norwood:2009} this is the maximum number of elements of order $1$ or $2$ that $T$ can contain, hence $\mathrm{Inv}(T)^\natural = Y_M^\natural \cup (Y_M^t)^\natural$ and $\A_T=\{Y_M,Y_M^t\}$. Since $N_L(Y_M) = L \cap M = N_M(T) = FT$ we have
        \[
        |L:FT| = |L:N_L(Y_M)| = |N_\L(T):C_{N_\L(T)}(\mathcal{A}_T)| =2
        \]
        so $FT \norm L$. Let $T^* \le T_1 \in \Syl_2(L)$, so $L=(FT)T_1=FT_1$ and $T_1 \not \le M$; note also that $T_1 = T\<t\>$. Since $Y_M^t \le T \le S$, conjugation by $t$ extends to a group isomorphism $M \le N_\L(Y_M) \fto N_\L(Y_M^t)$. Assume that $T_1$ centralizes $z$; by $Q!$ we obtain $T_1 \le C_\L(z) \le N_\L(Q_z)$; by \cite[Lemmas 3.4(i) and 3.3(e)(i)]{Meier-Stell-Stroth:2012} we also have $\<FT,Q_z\> = M^\circ T = M$, hence $t \in L = FT_1$ normalizes $M$, so $Y_M^t = O_2(M)^t = O_2(M) = Y_M$, a contradiction. Whence $[T_1,z] \ne 1$.\\
        Since $T_1 \not \le M$ and $|T_1:T|=2$ we have $T^* = T$, thus $L=N_\L(T)=N_\L(T^*)$; moreover, $T^* = C_S(y) = C_S(\<y,z\>)$. By \cite[Lemma3.4(a) and (f)]{Meier-Stell-Stroth:2012} $y^{Q_z} = \{y,yz\}$, thus $Q_z \le L$, and $\<z,y\> = \Omega Z(T^*) \norm L$, so both $Q_z$ and $T_1$ act on it non-trivially as distinct involutions. Thus $L/C_L(\Omega Z(T^*)) \cong \Sym(3)$ and so $L$ is transitive on $\Omega Z(T^*)$; in particular, there exists $h \in L$ such that $y=z^h$. Since $\<y,z\> \le S$, conjugation by $h$ extends to a group homomorphism $C_\L(z) \fto C_\L(y)$, hence $C_\L(y)$ contains a Sylow $2$-subgroup of the same order of $S$ and (a) implies $2n \ge 6$. As $Q_z \le C_\L(z)$, $Q_z^h =Q_y$ and then $|Q_y/Q_y \cap T^*|=|Q_z/Q_z \cap T^*| = 2$, since $Q_z \cap T^*$ is contained in the kernel of the action of $H$ on $\Omega Z(T^*)$. We use this information to find a non-trivial offender on $Y_M$ contained in $Q_zY_M$, which clearly cannot be contained in $Y_M=O_2(M)$, providing us with a contradiction to \cite[Lemma 3.4(j)]{Meier-Stell-Stroth:2012}.\\
		We have
		\[
		Q_y \cap T^* \le Q_y \cap M \le O_2(C_M(y)) = O_2(F_0) = O_2(FT) = T,
		\]
		thus showing $|Q_y/Q_y \cap T| \le 2$; as $|T/Y_M|=2$, $|Q_y / Q_y \cap Y_M| \le 4$ and we easily conclude 
		\begin{equation*}
			\label{eq1-lemma3.6}
			\tag{*}
			|Q_y| \le 4|Q_y \cap Y_M| \le 4|Y_M|.
		\end{equation*}
		On the other hand $2n \ge 6$ and \cite[Lemma 3.3(c) and (f)]{Meier-Stell-Stroth:2012} yield $|Q_zY_M/Y_M| = |O_2(C_{\ov{M}}(z))| = 2^{2n-2}$, $|Z_M/Z_M \cap Q_z| \le |Z_M/[Z_M,Q_z]| = 2$ and $|Y_M/Z_M| \le 2$. Thus
		\begin{multline*}
			\label{eq2-lemma3.6}
			\tag{**}
			|Q_z| = 2^{2n-2}|Q_z \cap Y_M| \ge 2^{2n-2}|Q_z \cap Z_M| \ge 2^{2n-2}|[Q_z, Z_M]| \ge \\
			\ge 2^{2n-3}|Z_M| \ge 2^{2n-4}|Y_M| \ge 4|Y_M|.
		\end{multline*}
		Comparing \eqref{eq1-lemma3.6} and \eqref{eq2-lemma3.6} we get $|Q_y| = 4|Y_M| = |Q_z|$ and all the inequalities of the two lines above are equalities. In particular, $Y_M \cap Q_y = Q_y$ and $Y_M \cap Q_z = [Q_z,Z_M] < Q_z$, that is $Y_M \le Q_y$ and $Y_M \not \le Q_z$. Set $Y_z := Y_M^{h\inv}$, then $Y_z$ is elementary abelian of order $|Y_M|$ contained in $Q_z$, so it is distinct from $Y_M$. Since $Y_M=C_M(Y_M)$, \cite[Lemma A.40]{Meier-Stell-Stroth:2012} shows the existence of the wanted non-trivial offender on $Y_M$ contained in $Q_zY_M$, yielding our contradiction.\\
		
		Thus $N_\L(T) \le N_\L(B(T)) \le M^\dagger$ and \cite[Lemma 3.4(e)]{Meier-Stell-Stroth:2012} shows (b).
	\end{itemize}
\end{proof}

\begin{lemma}[3.7]\label{lemma3.7}
	$z^\L \cap Y_M = z^M = z^\L \cap Z_M$
\end{lemma}
\begin{proof}
	Recall that $Z_M=[Y_M,O^2(M)] \le S$, that $M$ acts transitively on the set of non-singular elements of $Z_M$ and that each non-singular vector (as well as each non-trivial singular vector) of $Z_M$ generates a subgroup which is in $\Delta$.\\
	Note that $z^\L \cap Y_M \le S$ and also $z^\L \cap Z_M \le S$; moreover the inclusions $z^M \le z^\L \cap Z_M \le z^\L \cap Y_M$ are trivial, so we only need to prove $z^\L \cap Y_M \le z^M$. For this, suppose there exists $u \in (z^\L \cap Y_M) \setminus z^M$.
	\begin{itemize}
		\item If $u \in Z_M$, since $u$ is not an $M$-conjugate of $z$, it is an $M$-conjugate of $y$, namely $u=y^m$, $m \in M$. Then $y = u^{m\inv} \in S$; but we also know that $u=z^x \in Y_M$ for some $x \in \L$, then $(x,m\inv) \in D(\L)$ via $(\< z\>, \< u \>, \<y \>)$, so we get
		\[
		y = z^{xm\inv} \in z^{\L} \cap S.
		\]
		Since conjugation by $xm\inv$ extends to a group isomorphism $C_\L(z) \fto C_\L(y)$ and $C_\L(z)$ is parabolic, we get that $y$ is centralized by a $p$-subgroup of order $|S|$. If $2n=4$ this yields a contradiction to \ref{lemma3.6}(a), so $2n \ne 4$. Note that $Q_y = Q^{xm\inv}$ and it is uniquely determined by \ref{lemma:uniquen of some conjug of Q-orthog case} and normal in $C_\L(y)$. Then \ref{lemma3.6}(b) and \cite[Lemma 3.4(e)]{Meier-Stell-Stroth:2012} yield
		\[
		N_\L(T)= N_{M^\dagger}(T) \le C_{M^\dagger}(y) \le C_{\L}(y)
		\]
		so $N_\L(T)$ normalizes $Q_y$ and always \ref{lemma3.6}(b) gives $N_{Q_y}(T) \le O_2(N_\L(T)) = T$. We conclude $N_{Q_y}(T) = Q_y \cap T$. Note that $TQ_y$ is a $p$-group and assume $T < TQ_y$, then $T < N_{TQ_y}(T) = N_{Q_y}(T)T =T$, a contradiction; therefore we have
		\[
		Q_y \le T \le S.
		\]
		Since $Q=Q_z$ is weakly closed in $\L$, we obtain $Q_y=Q_z$; then $xm\inv \in N_\L(Q)$, hence it normalizes also $Z(Q)$. Since $z \in Z(Q)$, also $y \in Z(Q)$ and $[Q_z,y]=1$, contradicting \cite[Lemma 3.4(a)]{Meier-Stell-Stroth:2012}.\\
		\item If $u \in Y_M \setminus Z_M$, we may apply \cite[Lemma 3.3(f)]{Meier-Stell-Stroth:2012}, so $(2n,\epsilon) = (6,+)$ and $Y_M$ is the 7-dimensional module quotient of the natural permutation module for $\ov{M} \cong O^+_6(2) \cong \Sym(8)$. Here, up to isomorphism, $C_{\ov{M}}(v) \in \{ \Sym(7), \Sym(3) \times \Sym(5) \}$ for a non-trivial vector $v$; in both cases $O_2(C_{M^\dagger}(u)) \le C_{M^\dagger}(Y_M)$ and so
		\[
		O_2(C_{M^\dagger}(u)) \le O_2(C_{M^\dagger}(Y_M)) = O_2(M^\dagger) \le O_2(M).
		\]
		As $Q_u=Q^x \le O_2(C_\L(u))$, always for $u=z^x$, and $N_\L(O_2(M)) \le M^\dagger$ by \ref{lemma2.2}(c) we obtain $N_{Q_u}(O_2(M)) \le O_2(C_{M^\dagger}(u)) \le O_2(M)$; arguing as above for $T$ and $Q_y$, we have $O_2(M) \le N_{Q_uO_2(M)}(O_2(M)) = N_{Q_u}(O_2(M))O_2(M) = O_2(M)$. Therefore $Q_u \le O_2(M) \le S$ and, as $Q$ is weakly closed, $Q = Q_u = Q^g$ for all $g \in M$; then $Q \norm M$ contradicting Hypothesis (iv) of Theorem~\ref{theorem CL}.
	\end{itemize}
\end{proof}

\begin{lemma}[3.8]\label{lemma3.8}
	The following hold.
	\begin{enumerate}[(a)]
		\item $\Omega Z(T) = C_{Y_M}(T)$ and $Y = C_{Z_M}(T) = \Omega Z(T) \cap Z_M$.
		\item Set $\mathcal{Z} := \{u \in Y^\natural \mid u \text{ is singular in $Z_M$} \}$. Then $N_\L(\mathcal{Z})$ is a subgroup of $\L$ and
		\[
		N_\L(T) \le N_\L(\Omega Z (T)) = N_\L(\mathcal{Z}) = N_\L(Y) 
		\]
		\item $N_\L(Y)\le M^\dagger$; in particular, $N_\L(Y) = C_{M^\dagger}(y) = FTC_{M^\dagger}(Y_M)$, $O_2(N_\L(T)) =T$ and $C_\L(Y) = TC_{M^\dagger}(Y_M)$.
	\end{enumerate}	
\end{lemma}
\begin{proof}
	\begin{enumerate}[(a)]
		\item Lemma \ref{lemma2.2}(e) and (f) give $\Omega Z(O_2(M)) =Y_M$ and $\{O_2(M)\} = \Syl_2(C_M(Y_M))$. Now \cite[Lemma 3.4(c)]{Meier-Stell-Stroth:2012} yields $O_2(M) \le T$ and therefore $\Omega Z(T) =C_{Y_M}(T)$ and therefore also $C_{Z_M}(T) = \Omega Z(T) \cap Z_M$. Then \cite[Lemma 3.3(b)]{Meier-Stell-Stroth:2012} and \ref{lemma3.1 - orthogonal module}(c:2) imply $C_{Z_M}(T) = y^\perp =Y$ and (a) is proved.
		\item Note that $z \in \mathcal{Z} \subseteq Y$, hence \ref{lemma:uniquen of some conjug of Q-orthog case} shows that $Q_u$ is well-defined for all $u \in \mathcal{Z}$ and contained in $M$. Now \ref{lemma3.1 - orthogonal module}(c:3) shows that $\< \mathcal{Z} \> = Y$. Let $u \in \mathcal{Z}$ and $x \in N_\L(\mathcal{Z})$; since $\<u\> \in \Delta$ the conjugation map $c_x: u \fto u^x \in \mathcal{Z}$ extends to $c_x: C_\L(u) \fto C_\L(u^x)$. Since $Y$ is abelian and contains $u$, $Y \in C_\L(u)$, so $c_x$ is well-defined on $Y$. Thus $N_\L(\mathcal{Z}) \subseteq N_\L(Y)$ and it is, in particular, a subgroup of $\L$.\\
		The transitivity of the action of $M$ over the non-trivial singular vectors of $Z_M$ shows that $\mathcal{Z}=Y \cap z^M$; (a) and $z^M \subseteq Z_M$ now characterize $\mathcal{Z}$ as $\mathcal{Z} = \Omega Z(T) \cap z^M$. Applying \ref{lemma3.7} and the fact that $\Omega Z(T) \le Y_M$ we get 
		\[
		\mathcal{Z} = \Omega Z(T) \cap z^M = \Omega Z(T) \cap z^\L \cap Y_M = \Omega Z(T) \cap z^\L \ge \mathcal{Z},
		\]
		so $\mathcal{Z} = \Omega Z(T) \cap z^\L$ and, as $\mathcal{Z} = Y \cap z^m \le Y \cap z^\L = \Omega Z(T) \cap Z_M \cap z^\L \le \mathcal{Z}$, we also find $\mathcal{Z}= Y \cap z^\L$ and therefore have
		\[
		N_\L(T) \le N_\L(\Omega Z(T)) \le N_\L(\mathcal{Z}) = N_\L(Y).
		\]
		\item As seen in (b), for every $x \in N_\L(\mathcal{Z})=N_\L(Y)$ and every $u \in \Z$ the conjugation map $c_x: C_\L(u) \fto C_\L(u^x)$ is well-defined and it sends $Q_u$ to $(Q_u)^x = Q_{u^x}$. In particular, since $\<Z_M,Q_u\> \le C_\L(u)$, conjugation by $x$ is defined also on $Z_M$ and, as $Q_u$ normalizes $Z_M$, we get $[Q_{u^x},Z_M^x] \le Z_M^x$. Since this holds for every $u \in \mathcal{Z}$ and $x$ permutes the elements of $\mathcal{Z}$, it holds in particular for $z^{x\inv}$, so $[Q_z,Z_M^x] \le Z_M^x$.\\
		Suppose that $Q_z \le N_\L(Y)$; since $Y=y^\perp$ we have that $Q_z \le N_\L(Y) \cap M =C_M(y)$, contradicting \cite[Lemma 3.4(a)]{Meier-Stell-Stroth:2012}. Thus $Q_z \not \le N_\L(Y)$ and, since $Y$ is a hyperplane of $Z_M$ contained also in $Z_M^x$ and $Q_z$ normalizes $Z_M^x$, we obtain
		\[
		Z_M = \< Y^{Q_z} \> \le \< (Z_M^x)^{Q_z} \> = Z_M^x
		\]
		hence $Z_M^x = Z_M$. Since $Z_M \norm M$, by \ref{lemma2.2}(c) we conclude $N_\L(Y) \le N_\L(Z_M) \le M^\dagger$.\\
		As $N_\L(Y) \le M^\dagger$, \cite[Lemma 3.4(k)]{Meier-Stell-Stroth:2012} gives $N_\L(Y)= N_{M^\dagger}(Y) =C_{M^\dagger}(y) = FTC_{M^\dagger}(Y_M)$. Using (b) and \cite[Lemma 3.4(e)]{Meier-Stell-Stroth:2012} one gets $O_2(N_\L(T))=T$; (b) and \ref{lemma3.1 - orthogonal module}(c:2) yield instead $C_\L(Y) = TC_{M^\dagger}(Y_M)$.
	\end{enumerate}
\end{proof}

\begin{notation}[3.9]\label{notation3.9}
	Since $FT \le C_\L(y) \not \le M^\dagger$ there exists a subgroup $L \le C_\L(y)$ such that $FT \le L$ and $L \not \le M^\dagger$; choose one with $|L|$ minimal, so that $L \cap M^\dagger$ is the unique maximal subgroup of $L$ containing $FT$. By \cite[Lemma 3.4(c)]{Meier-Stell-Stroth:2012} $T^* \in \Syl_2(FT)$, hence let $T_0 \in \Syl_2(L)$ with $T^* \le T_0$ and set
	\[
	D:= L \cap M^\dagger, \; Z_L := \< \Omega Z(T_0) ^L \>, \; P:= C_L(z), \; P^*:= O^{2'}(P).
	\]
\end{notation}

\begin{lemma}[3.10]\label{lemma3.10}
	The following hold.
	\begin{enumerate}[(a)]
		\item $\<Q_z,L \>$ is not a subgroup of $\L$ and $\coll_2(\<Q_z, L\>) = 1$.
		\item $[Q_z,P] \le Q_z \cap P = C_{Q_z}(y) = Q_z \cap L = Q_z \cap F \le O_2(P)$.
		\item If $t \in Q_z \setminus L$, then $L^t$ is a well-defined subgroup, $t^2 \in L$ and $\coll_2(\<L, L^t \>)=1$.
		\item $O_2(C_\L(y)) \le O_2(L) \le T$.
		\item $L$ is of characteristic $2$. In particular, for $t \in Q_z \setminus L$, $L$ does not normalize $Z_L^t$ and $L^t$ does not normalize $Z_L$.
		\item $\Omega Z(T_0) = \Omega Z(T^*) = \<y, z \>$.
		\item $P =C_L(\Omega Z(T_0))$, so $P^*$ is a point-stabilizer for $L$ on $Y_L$ and $Z_L$.
		\item $Z_L = \<Y^L\>$ and $Y \le C_{Z_L}(T) \le C_{Y_M}(T)$.
		\item $D=FTC_D(Y_M)=FC_D(Y)$, thus $Y$ is a natural $\O_{2n-1}(2)$-module and $Y/\<y\>$ is a natural $\Sp_{2n-2}(2)$-module for $D$.
		\item $F$ is normal in $D$.
		\item If $2n=4$, then $T^* =T_0 \in \Syl_2(L) \cap \Syl_2(C_\L(y))$.
	\end{enumerate}
\end{lemma}
\begin{proof}
	\begin{enumerate}[(a)]
		\item By \cite[Lemma 3.4(b) and (i)]{Meier-Stell-Stroth:2012} one has $M = M^\circ T = \<Q_z,TF \> \le \<Q_z, L \>$, hence $S \le \<Q_z, L \>$. Suppose that it is a subgroup of $\L$, then there exists an object $R \in \Delta$ such that $\<Q_z,L\> \le N_\L(R)$, implying also $R \le \coll_2(\<Q_z,L\>)$ and $M \le N_\L(R)$. Since $\M_\L(M) = \{M^\dagger \}$ we have $N_\L(R) \le M^\dagger$, so $L \le N_\L(R) \le M^\dagger$, a contradiction. Thus $\<Q_z, L\>$ is not a subgroup.\\
		Since $\coll_2(\<Q_z,L\>) \le S \le \<Q_z,L\>$ we get $\coll_2(\<Q_z,L\>) \norm S$; by $Q$-repleteness it is either $\coll_2(\<Q_z,L\>) \in \Delta$ or $\coll_2(\<Q_z,L\>)=1$. The first case implies that $\<Q_z,L\>$ is a subgroup of $\L$, a contradiction, thus the second case holds.
		\item Since $Q_z,P \le C_\L(z)$, $[Q_z,P]$ is well-defined. Now the proof is precisely as in \cite[Lemma 3.10(b)]{Meier-Stell-Stroth:2012}; we report the argument, as it is very simple and short.\\
		The equalities $Q_z \cap F = Q_z \cap L = Q_z \cap P$ are trivial; since $Q_z \cap L \le C_{Q_z}(y) \le Q_z \cap F$, the previous groups are also equal to $C_{Q_z}(y)$.\\
		Since $Q_z=Q \norm C_\L(z)$, it is normalized by $P$; on the other hand, $Q!$ gives $P \le C_\L(z) \le N_\L(Q_z)$, hence $[Q_z,P] \le Q_z \cap P$. In particular, $Q_z \norm Q_zP \le C_\L(z)$, so $Q_z \cap P \norm P$. 
		\item As $t \in Q_z \le S$, $y^t \in S$; thus conjugation by $t$ extends to $C_\L(y) \fto C_\L(y^t)$; since $L \le C_\L(y)$, $L^t$ is a well-defined subgroup of $\L$.\\
		By (b) and \cite[Lemma 3.4(a)]{Meier-Stell-Stroth:2012} $|Q_z/(Q_z \cap L)|=2$, whence $t^2 \in L$ and, as partial subgroups of $\L$,
		\[
		\<L,Q_z\> = \<L,Q_z \cap L,t \> = \< L,t \>.
		\]
		Assume there exists a subgroup $R \le S$ normalized by $\<L,L^t\>$; then $[R,L] \le R$, $[R,L^t]\le R$ and conjugation by $t$ yields
		\[
		[R^t,L] \le R^t \qquad \text{and} \qquad [R^t,L^t] \le R^t.
		\]
		Since $t^2 \in L$, we conclude that $\<R,R^t\> \le S$ is normalized by $\<L,t\>$, hence $R=1$ by (a). In particular, $\coll_2(\<L,L^t\>)=1$.
		\item Set $U := O_2(L)O_2(C_\L(y)) \le C_\L(y)$. Then $FT$ normalizes $U$, by \ref{lemma3.8}(b) $N_\L(T) \le FTC_{M^\dagger}(Y_M)$ and by \cite[Lemma 3.4(c)]{Meier-Stell-Stroth:2012} $O_2(\ov{FT}) = \ov{T}$. We conclude that $\ov{N_U(T)} \le O_2(\ov{FT}) = \ov{T}$, so $N_U(T) \le TC_{M^\dagger}(Y_M)$. Since $T$ normalizes $N_U(T)$ and is a Sylow $2$-subgroup of $TC_{M^\dagger}(Y_M)$, we get $T \le N_{TC_{M^\dagger}(Y_M)}(N_U(T))$ and, therefore, $N_U(T) \le T$. Thus $N_{UT}(T) = N_U(T)T = T$, so $U \le T \le S \cap L$ and we conclude $O_2(C_\L(y)) \le O_2(L) \le U \le T$.
		\item By (d) and since $C_\L(y)$ has characteristic $2$, $Z_L \le C_\L(O_2(L)) \le C_\L(O_2(C_\L(y))) \le O_2(C_\L(y)) \le O_2(L)$, hence $L$ has characteristic $2$ and it follows that $Z_L$ is an elementary abelian $2$-group and the same clearly holds then for $Z_L^t$.\\ Note that since $t^2 \in L$, $L$ normalizes $Z_L^t$ if and only if $L^t$ normalizes $Z_L$; suppose that $L$ normalizes $Z_L^t$, then $\<Z_L, Z_L^t \>$ is a $2$-group normalized by $\<L,L^t\>$. As $Z_L \le T \le S$ and also $t \in S$, $1 \ne \< Z_L, Z_L^t\> \le \coll_2(\<L,L^t\>)$, a contradiction to (c).
		\item For a suitable choice of $T_1$ we get the following chain of inclusions:
		\[
		O_2(C_\L(y)) \le T \le T^* \le T_0 \le T_1 \in \Syl_2(C_\L(y)),
		\]
		where the chain stays inside $S$ until $T^*$. By \cite[Lemma 2.9]{Chermak:2022:I} we may conjugate the entire chain inside $S$; in detail, there exists $x \in \L$ such that $\<y^{x}\> \in \Delta$ and $C_S(y^{x}) \in \Syl_2(C_\L(y^{x}))$, hence conjugating the chain above,
		\[
		O_2(C_\L(y^{x})) \le T^{x} \le (T^*)^{x} \le T_0^{x} \le T_1^{x} \in \Syl_2(C_\L(y^{x})),
		\]
		with $T_1^{x} = C_S(y^{x})$. We therefore get $\Omega Z (S) \, \Omega Z(T_0^x) \le C_\L(O_2(C_\L(y^{x}))) \le O_2(C_\L(y^{x})) \le (T^*)^x \le T_0^x \le S$. Now \cite[Lemma 3.4(f) and (g)]{Meier-Stell-Stroth:2012} yield $\Omega Z(T^*) = \< y, z \>$ and $\Omega Z (S) = \< z \> \in \Delta$, thus
		\[
		\<z\> = \Omega Z(S) \le \Omega Z(T_0^x) \le \Omega Z ((T^*)^x) = \< y^x, z^x \>
		\]          
		and the entire line is conjugable by $x\inv$. As $y^x \in \Omega Z(T_0^x)$, we have
		\[
		\< z^{x\inv}, y \> \le \Omega Z (T_0) \le \<y,z \>;
		\]
		since $z^\L \cap Y_M = z^M$, $z^{x\inv} \ne y$, so $\Omega Z(T_0) = \<z,y\>$.
		\item $C_L(\Omega Z(T_0)) = C_L(z) = P$ follows trivially from (f) and since $L \le C_\L(y)$.
		\item Let $\Z$ be the set on non-trivial singular vectors of $Y$, so $Y = \<\Z\>$ by \ref{lemma3.1 - orthogonal module}(c:3) and $C_{\ov{M}}(y)$ is transitive on $\Z$. Since \cite[Lemma 3.4(b)]{Meier-Stell-Stroth:2012} shows $C_{\ov{M}}(y) = \ov{FT}$ we obtain $Y=\< z^F \> = \< \Omega Z(T_0)^F \>$, hence $Z_L = \<\Omega Z(T_0)^L \> = \< (z^F)^L\> = \<Y^L\>$.\\
		Now $Z_L \le O_2(L) \le T$ by (d) and (e) and \ref{lemma3.8}(a) gives $Y \le C_{Z_L}(T) \le \Omega Z(T) = C_{Y_M}(T)$.
		\item Note that $FT \le D = M^\dagger \cap L \le C_{M^\dagger}(y)$; since by \cite[Lemma 3.4(k)]{Meier-Stell-Stroth:2012} $C_{M^\dagger}(y)=FTC_{M^\dagger}(Y_M)$ we have
		\[
		D=FT(D \cap C_{M^\dagger}(Y_M)) = FTC_D(Y_M).
		\]
		Recalling that $T=C_S(Y)$ we get $D = FTC_D(Y_M) \le FC_D(Y) \le D$, so all equalities are proven.\\
		Since \cite[Lemma 3.4(c)]{Meier-Stell-Stroth:2012} shows that $Y$ is a natural $\O_{2n-1}(2)$-module for $FT$ and $Y/\<y\>$ is a natural $\Sp_{2n}(2)$-module for $FT$, we conclude that the same holds for $D$.
		\item Again \cite[Lemma 3.4(k)]{Meier-Stell-Stroth:2012} yields $F \norm C_{M^\dagger}(y)$; since $F \le D \le C_{M^\dagger}(y)$, we also get $F \norm D$.
		\item As $2n=4$, \ref{lemma3.6}(a) gives $T^* \in \Syl_2(C_\L(y))$; however $T^* \le T_0 \in \Syl_2(L)$ and $L \le C_\L(y)$, thus $T^* = T_0$.
	\end{enumerate}
\end{proof}

\begin{lemma}[3.11]\label{lemma3.11}
	$C_L(Z_L) = O_2(L) = C_L(Y_L)$ and $Y_L = \Omega Z(O_2(L))$
\end{lemma}
\begin{proof}
	The argument is essentially the same as in \cite[Lemma 3.11]{Meier-Stell-Stroth:2012}, which we here reproduce.\\
	By \ref{lemma3.10}(h) $Y \le Z_L$ and by \ref{lemma3.8}(c) $C_L(Z_L) \le C_L/Y) \le T C_{M^\dagger}(Y_M)$. Since $T$ normalizes $C_{M^\dagger}(Y_M)$ we see that $C_L(Z_L)C_{M^\dagger}(Y_M) / C_{M^\dagger}(Y_M)$ is a $2$-group, thus $O^2(C_L(Z_L)) \le O^2(C_L(Z_L)C_{M^\dagger}(Y_M)) \le C_{M^\dagger}(Y_M)$.\\
	By \ref{lemma: L containing Q}(a) $[M^\circ, C_{M^\dagger}(Y_M)] \le O_2(M^\circ) \le O_2(M^\dagger)$, hence $[Q_z, O^2(C_L(Z_L))] \le O_2(M^\dagger)$; that is, $Q_z$ normalizes $O^2(C_L(Z_L))O_2(M^\dagger)$. Since $O^2(M^\dagger) \le T \le L$, $O^2(C_L(Z_L)) O_2(M^\dagger) = O^2(C_L(Z_L))$, a subgroup of $L$ normalized by $Q_z$ and by $L$. Then $O_2(O^2(C_L(Z_L))) \le O_2(L) \le S$ by \ref{lemma3.10}(d) and, since $O_2(\<Q_z,L\>)=1$ by \ref{lemma3.10}(a), we get $O_2(O^2(C_L(Z_L)))=1$. Since $L$ is of characteristic $2$, so is $O^2(C_L(Z_L))$, hence $O^2(C_L(Z_L))=1$ and thus $C_L(Z_L) = O_2(L)$.\\
	Set now $U:=\Omega Z(O_2(L))$ and note that $Z_L \le Y_L \le U$, so
	\[
	O_2(L) \le C_L(U) \le C_L(Y_L) \le C_L(Z_L) = O_2(L);
	\]
	we deduce that $O_2(L) = C_L(U)$, so $U$ is $2$-reduced for $L$, showing that $U = Y_L$.
\end{proof}

\begin{remark}\label{remark-to-3.10-3.14}
	$\<Z_L, Z_L^t \> =Z_L Z_L^t \not \norm L$ and $\<O_2(L),O_2(L^t) \> = O_2(L)O_2(L^t) \le S$.
\end{remark}
\begin{proof}
	By \ref{lemma3.10}(f) and (g) $z \in Z_L \cap Z_L^t$, so $\<Z_L, Z_L^t\> \le C_L(z) = P$. Hence, always by (g), $Z_L^t$ centralizes $\Omega Z(T_0)$; if $v \in \Omega Z(T_0)$ and $l \in L$, for any $x \in Z_L^t$ one has
	\[
	(v^l)^{x}=(l\inv v l)^{x} = x\inv l\inv v l x= (x\inv l\inv x) x\inv v x (x\inv l x) = (l\inv)^x v l^x \in Z_L.
	\]
	Thus $Z_L^t$ normalizes $Z_L$ and, conjugating by $t$, also the converse is true. In particular, $\<Z_L, Z_L^t \> =Z_L Z_L^t$. Assume that $Z_L Z_L^t \norm L$; since $t^2 \in L$ also $t$ normalizes $Z_LZ_L^t$, so this is a non-trivial subgroup of $S$ normalized by $\<L,L^t\>$, a contradiction to \ref{lemma3.10}(c).\\
	By \ref{lemma3.11} and again \ref{lemma3.10}(g) we obtain $O_2(L) \norm P$; since $t \in Q_z$ and $Q_z$ normalizes $P$ by \ref{lemma3.10}(b), $O_2(L^t) \norm P^t = P$, thus $O_2(L)O_2(L^t)$ is a subgroup of $P$. Moreover, $O_2(L) \le T \le S$ by Lemma~\ref{lemma3.10}(e); since $t \in Q_z \le S$ and conjugation by $t$ is a group homomorphism taking shape
	\begin{center}
		\begin{tikzcd}
			c_t \colon &[-20pt] C_\L(y) \ar[r] & C_\L(y^t) \\[-10pt]
					& O_2(L) \ar[u,hook] \ar[r] & O_2(L)^t \ar[u,hook]
		\end{tikzcd}
	\end{center}
	we get also $O_2(L^t) = O_2(L)^t \le S$.
\end{proof}

Until now we have followed consistently the notation and lemmas subdivision of \cite[Chapter 3]{Meier-Stell-Stroth:2012}; we were often able to use the same arguments, but needed a number of adjustments and additional checkings to be sure that all the operations appearing in \cite{Meier-Stell-Stroth:2012} are possible and meaningful also in our locality $\L$. We have then identified certain subgroups of $\L$ which are analogous to those in the group $G$ in \cite{Meier-Stell-Stroth:2012}; the results \cite[Lemma 3.12 to 3.27]{Meier-Stell-Stroth:2012}, on the other hand, almost exclusively involve such subgroups, with arguments most of the time taking place within $C_G(y)$ or $M^\dagger$. We know the latter to be a subgroup of $\L$ as well; moreover, $C_\L(y)$ also is a group, since our hypothesis imply that $\< y \> \in \Delta$.\\
The following remarks apply.

\begin{itemize}
	\item We have already observed at the beginning of the current section that $B(T) \in \Delta$; also $Y, \Omega Z(T_0) \in \Delta$ (for example, $y$ is contained in both), hence $N_\L(B(T))$, $C_\L(Y)$ and $C_\L(\Omega Z(T_0))$ are subgroups of $\L$. This covers the occurrences of $N_G(B(T))$ in \cite[Lemma 3.14]{Meier-Stell-Stroth:2012}, of $C_G(Y)$ in \cite[Lemma 3.19]{Meier-Stell-Stroth:2012} and of $C_G(\Omega_1 Z(T_0))$ in \cite[Lemma 3.12]{Meier-Stell-Stroth:2012}.
	\item $O_2(L)O_2(L^t)$ defines a subgroup of $\L$ since both $O_2(L)$ and $O_2(L^t)$ are subgroups of $S$ by Remark~\ref{remark-to-3.10-3.14}.
	\item $G_0 := \< L,L^t\>$, defined so in \cite[Notation 3.24]{Meier-Stell-Stroth:2012}, is not a subgroup in the locality $\L$; we have however proven that $\coll_2(G_0) = 1$ in Lemma~\ref{lemma3.10}(c). This is enough for its application in \cite[Lemma 3.25]{Meier-Stell-Stroth:2012}.
\end{itemize}

We record the following notation in use as well as few results important for the final contradiction.

\begin{enumerate}[label=(\alph*)]
	\item Set $\wde{L} := L/O_2(L)$; by \cite[Lemma 3.14(d)]{Meier-Stell-Stroth:2012} there is a subgroup $A \le O_2(P^*)$ such that $A$ is a minimal, non-trivial quadratic offender on $Y_L$. One then defines $H := \<A^L\>O_2(L)$ and, if $X \le Y_L$, sets
	\[
	X^+ := \frac{X C_{Y_L}(H)}{C_{Y_L}(H)}.
	\]
	\item By \cite[Lemma 2.23(a)]{Meier-Stell-Stroth:2012} $L=H$, $\wde{L} \in \{\SL_3(2), \Sp_4(2), G_2(2)\}$ and $Z_L^+$ is a corresponding natural module.
	\item We have $O_2(L)O_2(L^t) = O_2(P)$, $Z_L = Y_L$ and $Y_L^t \le O_2(L)$, shown in \cite[Lemma 3.26(a),(e) and 3.27]{Meier-Stell-Stroth:2012}.
\end{enumerate}

It is then possible to produce a final contradiction in the same way as in \cite[3.28]{Meier-Stell-Stroth:2012}; the argument once again relies on the subgroups $L$ and $L^t$, regardless whether embedded in a group $G$ or in our locality $\L$.
Set $R := O_2(L)O_2(L^t)$. The contradiction is produced in an argument involving the free amalgamated product $G^*$ of $L$ and $L^t$ over $R$; in other words, $G^*$ is the pushout of the span $L \leftarrow R \rightarrow L^t$ (where the arrows are inclusions) in the category $\mathrm{Grp}$ of groups. Setting $L_1$ and $L_2$ for the copies of $L$, respectively $L^t$, in $G^*$, we can then identify $R$ with $L_1 \cap L_2$. Just as in \cite{Meier-Stell-Stroth:2012}, one shows that $G^*$ satisfies \cite[Hypothesis 1]{Parmeggiani:2009II} for $p=2$; the conclusion of \cite[Main Theorem]{Parmeggiani:2009II} then produces a contradiction to the description of $\wde{L}$ in (b).\\

Recall that conjugation by $t$ induces a group-isomorphism as below
\begin{center}
	\begin{tikzcd}
		c_t \colon 	&[-20pt] C_\L(y) \ar[r] 	& C_\L(y^t) \\[-10pt]
					&  L \ar[u,hook] \ar[r]		& L^t \ar[u,hook]
	\end{tikzcd}
\end{center}
For $i \in \{1,2\}$ set
\[
	Z_i := \Omega Z(O_2(L_i)), \qquad Z_i^+ := Z_i/C_{Z_i}(L_i), \qquad \wde{L}_i := L_i / O_2(L_i).
\]
Moreover, in $L_i$ the point stabilizer (for $S_i \in \Syl_2(L_i)$) is defined for the $L_i$-module $\Omega (O_2(L_i))$. Since $L_i$ is of characteristic $2$, $C_{\Omega (O_2(L_i))}(S_i) = \Omega Z(S_i)$, so
\[
	P_{L_i}(S_i) = O^{2'}(C_L(\Omega Z(S_i)))
\] 

We now check one by one \cite[Hypothesis 1]{Parmeggiani:2009II}.
\begin{enumerate}[label=(\arabic*)]
	\item $R$ is a $2$-group with $C_{L_i}(Z_i) \le R$.
	\begin{proof}
		We have $R = O_2(L)O_2(L^t) \le S$ by Remark~\ref{remark-to-3.10-3.14}, thus it is a $2$-group.\\
		By Lemma~\ref{lemma3.11} $C_L(Z_L) = O_2(L) \le R$. Now always Lemma~\ref{lemma3.11} together with \cite[Lemma 3.36]{Meier-Stell-Stroth:2012} give $Z_L = Y_L = \Omega Z(O_2(L)) = Z_1$, hence (1) holds.
	\end{proof}
	\item $\wde{L}_i \in \{\SL_{n_i}(2^k), \Sp_{2n_i}(2^k), G_2(2^k)\}$, $k \ge 1$, and $Z_i^+$ is a corresponding natural module for $\wde{L}_i$.
	\begin{proof}
		\cite[Lemma 3.23]{Meier-Stell-Stroth:2012}(a) actually gives a smaller set of possibilities: $\wde{L} \in \{\SL_3(2), \Sp_4(2), G_2(2)\}$. It also shows that $Z_L^+$ is a corresponding natural module, thus (2) holds.
	\end{proof}
	\item There is $S_i \in \Syl_2(L_i)$ such that $R \norm P_{L_i}(S_i)$ and either $R = O_2(P_{L_i}(S_i))$ or $\wde{L}_i \cong G_2(2^k)$ and $\wde{R}$ is elementary abelian of order $2^{3k}$.
	\begin{proof}
		By \cite[Lemma 3.14(c)]{Meier-Stell-Stroth:2012} $R \norm P$, hence $R \norm O^{2'}(P)=P^*$. In Lemma~\ref{lemma3.10}(g) it is shown that $P^*$ is the point stabilizer of $L$ for $T_0 \in \Syl_p(L)$.\\
		Since $R = O_2(L)O_2(L^t) = O_2(P)$ we see that $R=O_2(P^*)$, proving (3).
	\end{proof}
	\item $Z_1Z_2 \le O_2(L_i)$, but it is not normal in $L_i$.
	\begin{proof}
		By \cite[Lemma 3.27]{Meier-Stell-Stroth:2012} $Z_L^t \le O_2(L)$, hence $Z_L Z_L^t \le O_2(L)$. Since $t$ has order $2$ it conjugates $L$ to $L^t$ and viceversa, hence $Z_1Z_2 \le O_2(L_i)$ for both $i \in \{1,2\}$. Remark~\ref{remark-to-3.10-3.14} gives $Z_LZ_L^t \not \norm L$, which by the same argument on $t$ is sufficient to prove (4).
	\end{proof}
	\item No subgroup $1 \ne U \le R$ is normal in $G^*$.
	\begin{proof}
		If $U \le R$ is normal in $G^*$, then $U \norm L_1$ and $U \norm L_2$. By Remark~\ref{remark-to-3.10-3.14} $R \le S$, thus (under the identification of $L_1$ with $L$ and $L_2$ with $L^t$) $U \le \coll_2(\<L,L^t\>) = 1$, where the last equality follows from Lemma~\ref{lemma3.10}(c).
	\end{proof}
\end{enumerate}

We may then apply \cite[Main Theorem]{Parmeggiani:2009II}, which yields the following conclusion: $\wde{L}_i \cong \SL_n(q)$, where $q$ is a power of the prime $p$, and
\begin{itemize}
	\item[(i)] either $p=3$ and $n=2$,
	\item[(ii)] or $q=p=2$ and $n=4$.
\end{itemize}

Case (i) is to exclude since $p=2$ in our case, but (ii) now contradicts the conclusion of \cite[Lemma 3.23(a)]{Meier-Stell-Stroth:2012}, reported in point (b) above, thus concluding a proof of Theorem~\ref{theorem CL}.

\begin{remark}
	The property we rely on for applying the same argument as in \cite{Meier-Stell-Stroth:2012} also to our locality $\L$ is essentially  the following. The span of groups
	\begin{equation*}
		\label{eq:pushout span for orthogonal case}
		\tag{$\dagger$}
		\begin{tikzcd}[row sep=tiny]
				& L \\
			R \ar[ru,hook] \ar[rd,hook]		\\[-5pt]
				& L^t
		\end{tikzcd}
	\end{equation*}
	may be mapped both in subgroups of the group $G$, as in \cite{Meier-Stell-Stroth:2012}, and in subgroups of $\L$, namely $L$, $L^t$ and $R \le L \cap L^t$. One can then compute the pushout $\mathcal{E}$ of such span in the category of partial groups ($\mathcal{E}$ exists by \cite[Theorem A]{Salati:2023}). Since $L$ and $L^t$ are groups and all the maps in the span are injective, $\mathcal{E}$ has, as underlying set, the pushout of the span taken in the category $\mathrm{Set}$ of sets, that is $(L \sqcup L^t) / \sim$, for $\sim$ the equivalence relation identifying elements coming from $R$. We refer the reader to \cite{Salati:2023} or to \cite{Hackney-Lynd:2025} for the computation of the pushout.\\
	It is now clear that $\mathcal{E}$ embeds then in the pushout $G^*$ taken in the category of groups; actually, $G^*$ is the fundamental group of $\mathcal{E}$, when seen as a simplicial set according to \cite{Hackney-Lynd:2025}. In other words, our locality $\L$ contains sufficient information to build the group $G^*$ with the same properties as in \cite{Meier-Stell-Stroth:2012}, thus yielding the same contradiction. 
\end{remark}

\section{The Symmetric Case}
\label{symm case}

We now proceed with the analysis of the first of the five cases in which the problem is split. This is the easiest case in \cite{Meier-Stell-Stroth:2012} and one of the easiest for our locality $\L$;  indeed, the arguments here almost perfectly overlap. The main issue arising within $\L$ is related to conjugates of the large subgroup $Q$: if in \cite{Meier-Stell-Stroth:2012} the authors are able to work with all such conjugates, we are forced here to cleverly choose a subset of $\L$-conjugates of $Q$, large enough to retain sufficient flexibility but small enough to also have sufficiently many conjugation maps among them. This issue will present itself also in other cases, as already explained in the introduction.\\

Recall that in \cite{Meier-Stell-Stroth:2012} an abelian $p$-subgroup of a group $G$ is called symmetric if there exists an element $g \in G$ such that $1 \ne [A, A^g] \le A \cap A^g$. The subgroups $A_1$ and $A_2$ are not required to be lying in a same Sylow $p$-subgroup, but this is a necessary and wanted hypothesis when dealing with a fusion system, e.g. a locality.\\
Nonetheless, the properties that are of interest to us are related to choices of symmetric subgroups such that at least one element of the pair is an elementary abelian, $p$-reduced, normal subgroup $Y \norm M$; then \cite[Lemma 4.2]{Meier-Stell-Stroth:2012} shows that there exists an element $u \in G$ such that $YY^u \le S \cap S^u$ and $1 \ne [Y,Y^u] \le Y \cap Y^u$. This justifies our definition of symmetric subgroups of a locality $\L$ given in \ref{def:symm-tall in a locality}.\\
The argument then follows the same lines as in the work of Meierfrankenfeld, Stellmacher and Stroth: in \cite[Lemma 4.9 (c)]{Meier-Stell-Stroth:2012} it is shown that either $Y^u$ or $Y^{u\inv}$ is a non-trivial quadratic offender on $Y$. After possibly replacing $(Y,Y^u)$ with $(Y^{u\inv},Y)$ one can then apply \cite[$Q!FF$-Module Theorem]{MS:2012a} to obtain a list of possible modules for $Y$.\\
We need to show that a similar replacement of the pair $(Y,Y^u)$ is possible in the locality $\L$, thus allowing us to obtain the same list of possible modules for $Y$. One proves in this way a result analogous to \cite[Theorem D]{Meier-Stell-Stroth:2012}.

\begin{maintheorem}{D$\L$}\label{theorem DL}
	Let $(\L,\Delta,S)$ be a $\K_p$-locality with a large subgroup $Q \le S$ satisfying hypothesis \ref{hyp on the locality}. Let $M \in \gM_\L(S)$ and $Y \le S$ be an elementary abelian, normal $p$-subgroup of $M$. Suppose that $Y$ is $p$-reduced for $M$ and symmetric in $\L$.\\
	Then, for $\ov{M}:=M/C_M(Y_M)$ and $q$ a power of $p$, one of the following holds.
	\begin{enumerate}
		\item $\ov{M^\circ} \cong \SL_n(q)$ for $n \ge 3$ and $Y$ is a corresponding natural $\SL_n(q)$-module.		
		\item \begin{itemize}
			\item[(a)] $\ov{M^\circ} \in \{ \Sp_{2n}(q), \Sp_4(2)'\}$ for $n \ge 2$ and $[Y,M^\circ]$ is a corresponding natural module.
			\item[(b)] If $[Y,M^\circ] \ne Y$, then $p=2$ and $|Y : [Y,M^\circ]| \le q$.
			\item[(c)] If $Y \not \le Q$, then $p=2$ and also $[Y,M^\circ] \not \le Q$.
		\end{itemize}
		\item $Y_M$ is a natural $\SL_2(q)$-wreath product module for $\ov{M}$ with respect to $\K$, where $\K$ is a uniquely determined set of $\ov{M}$-invariant subgroups of $\ov{M}$. In this case:
		\begin{itemize}
			\item[(a)] $\ov{M^\circ} = O^p(\< \K \>) \ov{Q}$,
			\item[(b)] $Q$ is transitive on $\K$,
			\item[(c)] for $Y=Y_M$, $Y_M = Y_{M^\circ S}$.
		\end{itemize}
		\item $Y \not \le Q$ and $\ov{M^\circ}$ is one of the following.
		\begin{itemize}
			\item[(1)] $\ov{M^\circ} \cong \begin{cases*}
				\Omega^+_{2n}(q) \; \text{for $2n \ge 6$,} \\
				\Omega^-_{2n}(q) \; \text{for $2n \ge 6$ and $p=2$,} \\
				\Omega^-_{2n}(q) \; \text{for $2n \ge 8$ and $p$ odd,} \\
				\Omega_{2n+1}(q) \; \text{for $2n+1 \ge 7$ and $p$ odd} 
			\end{cases*}$ \\
			and in all cases $Y$ is a corresponding natural module.
			\item[(2)] $\ov{M^\circ} \cong \SL_n(q) / \< (-id)^{n-1} \>$ for $n \ge 5$ and $Y$ is the exterior square of a natural $\SL_n(q)$-module.
			\item[(3)] $\ov{M^\circ} \cong \Spin_{10}^+(q)$ and $Y$ is a corresponding half-spin module.
			\item[(4)] $\ov{M^\circ} \cong \SL_n(q) \circ \SL_m(q)$ for $p \ne 2$, $n,m \ge 2$ and $n+m \ge 5$ and $Y$ is the tensor product of the corresponding natural modules.
		\end{itemize}
		\item $\ov{M} \cong \O_{2n}^\epsilon (2)$ and $\ov{M^\circ} \cong \Omega_{2n}^\epsilon(2)$ for $2n \ge 4$ and $(2n,\epsilon) \ne (4,+)$ with $[Y,M]$ a corresponding natural module. Moreover,
		\begin{itemize}
			\item[(a)] if $Y \ne [Y,M]$, then $\ov{M} \cong \O_6^+(2)$ and $|Y:[Y,M]| =2$,
			\item[(b)] there exists $y \in [Y,M]$ non-singular such that $C_\L(y) \not \subseteq M^\dagger$,
			\item[(c)] if $Y=Y_M$, then for every non-singular vector $y \in [Y,M]$, $\< y \> \not \in \Delta$.
		\end{itemize}
	\end{enumerate}
\end{maintheorem}

\subsection{The proof of Theorem~\ref{theorem DL}}

\begin{lemma}[4.1 and 4.2]\label{lemma4.2}
	$Y \le Y_M$ and $N_\L(Y) = N_\L(Y_M)=M^\dagger$; moreover, there exists $u \in \L$ such that $Y \le S_u \cap S_{u\inv}$, $1 \ne [Y,Y^u] \le Y \cap Y^u$ and $YY^u \le S \cap S^u$. In particular, also $Y^{u\inv} \in \Delta$.
\end{lemma}
\begin{proof}
	Clearly $Y \le Y_M$, since by hypothesis it is $p$-reduced for $M$; moreover, $Y \norm S$, so our hypothesis on the locality $\L$ ensure that $Y \in \Delta$. Since $M^\dagger = MC_\L(Y_M)$ and $M$ normalizes $Y$, $M^\dagger = N_\L(Y_M) \le N_\L(Y)$ and equality follows by the maximality of $M^\dagger$.\\
	
	Let $x \in \L$ be such that $Y \le S_x$ and $1 \ne [Y,Y^x] \le Y \cap Y^x$; by definition, $Y^x \le S$, hence conjugation by $x$ extends to an isomorphism $N_\L(Y) \fto N_\L(Y^x)$. In particular, $Y \le N_\L(Y^x)$, so there exists $n \in N_\L(Y^x)$ such that $Y \le (S^x)^n$. Note that $(x,n) \in D(\L)$ via $(Y,Y^x,Y^x)$ and set $u:=xn$. Then $Y^x = Y^u$, thus
	\[
	1 \ne [Y,Y^x] = [Y,Y^u] \le Y \cap Y^u,
	\]
	and also
	\[
	Y^u = Y^x \le S \cap S^u \qquad \text{and} \qquad Y \le S^u \cap S.
	\]
	Thus $YY^u \le S \cap S^u$ and $Y$ is conjugable by $u\inv$ with $Y^{u\inv} \le (S^u)^{u\inv} = S$, that is $Y \le S_{u\inv}$.
\end{proof}

We now fix $u \in \L$ as in \ref{lemma4.2}, so that $M$ is conjugable both by $u$ and $u\inv$ via group isomorphisms; in detail, the conjugation maps $c_u: N_\L(Y) \fto N_\L(Y^u)$ and$c_{u\inv} \colon N_\L(Y) \fto N_L(Y^{u\inv})$ are well-defined isomorphisms of groups. This allows us to fix a notation similar to that used by the authors in \cite{Meier-Stell-Stroth:2012}.

\begin{notation}\label{notation4.3}
	Fix $u$ as above and set
	\begin{gather*}
		M_1 := M, \; S_1 := S, \; Q_1 := Q, \; Y_1 := Y \\
		M_2 := M^u, \; S_2 := S^u, \; Q_2 := Q^u, \; Y_2 := Y^u
	\end{gather*}
	Let $i \in \{1,2\}$ and set
	\[
	\ov{M_i} := M_i / C_{M_i}(Y_i),  \quad A_i := C_{Y_i}(Q_i) \quad \text{and} \quad \ov{L_i}:= [F^*(\ov{M_i}),Q_i];
	\]
	Define also $\ov{F_i}$ to be the (unique) largest normal subgroup of $F^*(\ov{M_i})$ such that $[\ov{F_i},Q_i] = 1$ and $F_i$, $L_i$ to be the full preimages of $\ov{F_i}$, respectively $\ov{L_i}$, in $M_i$. More generally, the bar symbol will be used to denote quotients of subgroups of $M_i$ modulo $C_{M_i}(Y_i)$.\\
	Note that $\< Y_1,Y_2 \> \le M_1 \cap M_2$ and recall that for a subgroup $Q \le K \le M$ we have set $K^\circ := \<Q^K\>$; in particular, $M^\circ$ is the normal closure of $Q$ in $M$.
\end{notation}

The proof now follows very closely that in \cite[Chapter 4]{Meier-Stell-Stroth:2012}, but due to the high number of small changes and additional reasoning arguments we provide a complete list of all the results required. 

\begin{lemma}[4.4]\label{lemma4.4}
	Let $\{i,j\} = \{1,2\}$. Then $Y_i$ acts quadratically on $Y_j$.
\end{lemma}
\begin{proof}
	This is immediate; by definition $[Y_1,Y_2] \le Y_1 \cap Y_2$, so, as $Y_i$ is abelian, $[Y_j,Y_i,Y_i] = 1$.
\end{proof}

\begin{lemma}[4.5]\label{lemma4.5}
	The following hold.
	\begin{enumerate}[(a)]
		\item $F_1 \le N_\L(Q_1)$; in particular $F_i \le N_{M_i}(Q_i)$.
		\item $\ov{L_i}$ and $\ov{F_i}$ are normal subgroups of $F^*(\ov{M_i})\ov{S_i}$, hence both $L_i$ and $F_i$ are normal in $L_iF_iS_i$.
		\item $\ov{F_i} = C_{F^*(\ov{M_i})}(L_iQ_i)$; thus $[\ov{L_i},\ov{F_i}]=1$.
		\item $\ov{L_i} = [\ov{L_i},Q_i]$.
		\item $C_{\ov{M_i}}(\ov{L_iF_i})$ is a $p'$-group.
		\item If a $p$-subgroup $B$ of $N_{M_i}(Q_i)$ is such that $[\ov{L_i}, B] \le \ov{F_i}$, then $[\ov{L_i},B]=1$.
		\item $\ov{L_i} \cap \ov{F_i} \le \Phi(\ov{L_i})$.
	\end{enumerate}
\end{lemma}
\begin{proof}
	Points (c) - (f) follow from \cite[Lemma 1.17]{Meier-Stell-Stroth:2012} and (b) is just as in \cite[Lemma 4.5]{Meier-Stell-Stroth:2012}, therefore we only need to deal with (a), which needs only slightly extra effort.
	\begin{itemize}
		\item[(a)] Recall that $Q_1=Q$. Since $\Omega Z(S) \le Y_1 \cap C_\L(Q)$, $Q!$ implies $C_{M_1}(Y_1) \le N_\L(\Omega Z(S)) \le N_\L(Q_1)$, so $Q_1 \le O_p(Q_1C_{M_1}(Y_1)) \le S$ and, therefore, also $O_p(Q_1C_{M_1}(Y_1)) \in \Delta$. Thus $N_\L(O_p(Q_1C_{M_1}(Y_1)))$ is a parabolic subgroup of $\L$ and, since $Q$ is weakly closed, every element $g \in N_\L(O_p(Q_1C_{M_1}(Y_1)))$ must normalize also $Q_1$, that is $N_\L(O_p(Q_1C_{M_1}(Y_1))) \le N_\L(Q_1)$. Since $F_1$ normalizes $O_p(Q_1C_{M_1}(Y_1))$ we obtain $F_1 \le N_\L(Q_1)$; hence also $F_i \le N_{M_i}(Q_i)$.
	\end{itemize}
\end{proof}

\begin{lemma}[4.6]\label{lemma4.6} 
	We have $C_{Y_2}(\ov{L_1}) \in \{Y_2, C_{Y_2}(Y_1)\}$.
\end{lemma}
\begin{proof}
	In \cite[Lemma 4.6]{Meier-Stell-Stroth:2012} the authors consider the weak closure of $Q_1$ in $H_1$ and invoke \cite[Lemma 1.46(c)]{Meier-Stell-Stroth:2012}; this is not so well-behaved in a locality, hoewever on can replace such weak closure with the normal closure $\<Q_1^{H_1} \>$ (which constitutes no problem within our locality) and the proof works equally well. Indeed, \cite[Lemma 1.46(b)]{Meier-Stell-Stroth:2012} shows that, in the situation where $G$ is a group with large $p$-subgroup $Q_1$,  $\<Q_1^{H_1} \>$ is equal to the weak closure of $Q_1$ in $H_1$.\\
	We provide here the argument.\\
	
	Since both $\ov{L_1}$ and $\ov{F_1}$ are contained in $F^*(\ov{M_1})$, \ref{lemma4.5}(b) shows that $\ov{L_1}\ov{F_1}$ is subnormal in $\ov{M_1}$, hence $O_p(\ov{L_1}\ov{F_1}) \le O_p(\ov{M_1}) = 1$. Let $X:=C_{Y_2}(\ov{L_1})$ and $H_1:= L_1F_1Q_1Y_2$, well-defined since all are subgroups of $M_1$; note that always \ref{lemma4.5}(b) guarantees that $Q_1Y_2$ normalizes $L_1F_1$, so $L_1F_1 \norm H_1$. Let $K_1:=\langle Q_1^{H_1} \rangle$; \ref{lemma4.5}(a) and (d) yield
	\[
	\ov{K_1} = \langle \ov{Q_1}^{\ov{L_1F_1Q_1Y_2}} \rangle = \langle \ov{Q_1} ^{\ov{L_1}} \rangle = [\ov{L_1},\ov{Q_1}]\ov{Q_1} = \ov{L_1}\ov{Q_1}.
	\]
	Since $X$ centralizes $\ov{L_1}$ and $Q_1$ centralizes $\ov{F_1}$, we have $[X,\ov{Q_1}] \le C_{\ov{Q_1}}(\ov{L_1F_1})$; by \ref{lemma4.5}(e), $C_{\ov{M_1}}(\ov{L_1F_1})$ is a $p'$-group, so $[X,\ov{Q_1}]=1$ and, as $[O_p(\ov{H_1}),\ov{L_1F_1}] \le O_p(\ov{L_1F_1}) = 1$, we also get $O_p(\ov{H_1}) = 1$. In particular, we have
	\begin{equation*}
		\label{eq1-lemma4.6}
		\tag{$\star$}
		X=C_{Y_2}(\ov{L_1Q_1}) = C_{Y_2}(\ov{K_1}).
	\end{equation*}
	Suppose now that $X=C_{Y_2}(\ov{L_1}) < Y_2$ and set $\wde{Y_2} := Y_2C_{M_1}(Y_1)/C_{M_1}(Y_1)$, so $\wde{Y_2}\ne 1$, $\ov{L_1} \ne 1$ and $\ov{Q_1} \ne 1$, or otherwise one gets $\ov{K_1} = 1$, whence $X = Y_2$.\\
	We can now verify the hypothesis of \cite[Lemma A.57]{Meier-Stell-Stroth:2012} with $(Y_1, \ov{Q_1},\ov{H_1},\wde{Y_2})$ in place of $(V,Q,H,Y)$. We have already shown that $O_p(\ov{H_1}) = 1$ and this implies $\ov{Q_1} \not \norm \ov{H_1}$, as $\ov{Q_1} \ne 1$. Now \ref{lemma1.57} shows that $Y_1$ is a faithful $Q!$-module for $\ov{H_1}$ with respect to $\ov{Q_1}$. By \ref{lemma4.4} $Y_2$ acts quadratically on $Y_1$, so $1 \ne \wde{Y_2} = C_{\wde{Y_2}}([\ov{Y_1}, \wde{Y_2}])$; moreover $1 \ne [\wde{Y_2},\ov{L_1}] \le [\wde{Y_2},\ov{K_1}]$. Now \cite[Lemma A.57]{Meier-Stell-Stroth:2012} yields $C_{\wde{Y_2}}(\ov{K_1}) =1$, therefore $C_{Y_2}(\ov{K_1}) = C_{Y_2}(Y_1) = X$ by \eqref{eq1-lemma4.6}.
\end{proof}

Recall that we have defined $A_i:= C_{Y_i}(Q_i)$.
\begin{lemma}[4.7]\label{lemma4.7}
	Suppose that a subgroup $U \le S_i$ is such that $[\ov{L_i},U]=1$ and $[Y_i,U] \ne 1$, then $[A_i,U] \ne 1$.
\end{lemma}
\begin{proof}
	The proof is precisely as in \cite[Lemma 4.7]{Meier-Stell-Stroth:2012}.	
\end{proof}

\begin{lemma}[4.8]\label{lemma4.8}
	Both $[\ov{L_1},Y_2] \ne 1$ and $[\ov{L_2},Y_1] \ne 1$ hold. In particular, $C_{Y_2}(\ov{L_1}) = C_{Y_2}(Y_1)$.
\end{lemma}
\begin{proof} 
	Note that by our definitions also $Y_1^{u\inv} \in \Delta$ and satisfies the same properties of $Y^u$ with respect to $Y$. Since $[\ov{L_2}, Y_1] = [\ov{L_1^{u}}, Y_1] = [\ov{L_1}, Y_1^{u\inv}]$, we only need to prove that $[\ov{L_1},Y_2] \ne 1$ for $Y_2$.\\
	The argument used in \cite[Lemma 4.8]{Meier-Stell-Stroth:2012} is a completely local argument, apart from point $(5^\circ)$. Here, as $F \le M_2$, we may replace $G$ with $M_2$ and consider
	\[
	F \le N_{M_2}(C_U(Q_2)) \le N_{M_2}(Q_2);
	\]
	this holds since $Q$ being large in $\L$ implies that $Q$ is large in $M$ and, therefore, that $Q_2$ is large in $M_2$.	Whence the same reasoning applies to our case, proving the assertion. The last statement follows then from \eqref{lemma4.6}.
\end{proof}

\begin{lemma}[4.9] \label{lemma4.9}
	The following hold.
	\begin{enumerate}[(a)]
		\item $[\ov{M_1^\circ},Y_2] \ne 1$ and $[\ov{M_1^\circ},Y_1^{u\inv}] \ne 1$; thus $\ov{M_1^\circ} \ne 1$.
		\item $Y_i$ is a $Q!$-module for $\ov{M_i}$ with respect to $\ov{Q_i}$.
		\item One of $Y_2$ or $Y_1^{u\inv}$ is a non-trivial quadratic offender on $Y_1$.
		\item The triple $(\ov{M_i},Y_i,\ov{Q_i})$ fulfills the hypothesis of the \nameref{thm:Q!FF}. 
	\end{enumerate}
\end{lemma}
\begin{proof}
	The proof follows the lines of \cite[Lemma 4.9]{Meier-Stell-Stroth:2012}; we provide all the details since it constitutes the main reason for considering symmetric pairs.\\
	Since $\ov{L_i}=[F^*(\ov{M_i},Q_i]$, we get $\ov{L_i} \le \ov{M_i^\circ}$; by \ref{lemma4.8} $[\ov{L_1},Y_2] \ne 1$ and $[\ov{L_2},Y_1] \ne 1$, hence $[\ov{M_1^\circ},Y_2] \ne 1$ and $[\ov{M_2^\circ},Y_1] \ne 1$. Since conjugation by $u\inv$ is defined on $M_2$ and $M_2^{u\inv} = M_1$ we have also $[\ov{M_1^\circ},Y_1^{u\inv}] \ne 1$; this is (a).\\
	Note the $\ov{Q_i}=1$ implies $\ov{M_i^\circ}=1$, contradicting (a); hence $\ov{Q_i} \ne 1$. In addition, $Q_i \not \norm M_i$ since $O_p(\ov{M_i}) =1$, so \ref{lemma1.57}(b) yields that $Y_i$ is a faithful $Q!$-uniqueness module for $\ov{M_i}$ (with respect to $\ov{Q_i}$), proving (b).\\
	By \cite[Lemma 4.4]{Meier-Stell-Stroth:2012} $Y_1$ and $Y_2$ mutually act quadratically on each other; suppose that $Y_2$ is not an offender on $Y_1$, thus
	\[
	\frac{|Y_2|}{|C_{Y_2}(Y_1)|} < \frac{|Y_1|}{|C_{Y_1}(Y_2)|}.
	\]
	Conjugating now by $u\inv$ one gets
	\[
	\frac{|Y_1|}{|C_{Y_1}(Y_1^{u\inv})|} < \frac{|Y_1^{u\inv}|}{|C_{Y_1^{u\inv}}(Y_1)|},
	\]
	showing (c).\\
	By (c) $Y_1$ admits a non-trivial quadratic offender $Y_3 \in \{Y_2, Y_1^{u\inv}\}$ with $[\ov{M_1^\circ},Y_3] \ne 1$ by (a); since $Y_1$ is a $Q!$-module by (b) and $O_p(\ov{M_1})=1$, the Hypothesis of the \nameref{thm:Q!FF} are fulfilled.
\end{proof}

As a consequence, up to replacing $u$ with $u\inv$, we may always consider $Y_2$ a non-trivial, quadratic offender of $Y_1$.\\

We are now left to analyze the possible outcomes of the \nameref{thm:Q!FF}; to get this done, we first need to analyze occurrences of the orthogonal module for the prime $2$. If $V$ is a vector space equipped with a form $f$, denote the set of singular vectors of $V$ by $Sing(V)$.

\begin{lemma}[4.10] \label{lemma4.10}
	Suppose that the following hold:
	\begin{itemize}
		\item[(i)] $M^\circ / C_{M^\circ}(Y) \cong \Omega_{2n}^\epsilon (2)$ for $2n \ge 4$ and $\ov{M} \ne \ov{M^\circ}$,
		\item[(ii)] $[Y,M^\circ]$ is a natural $\Omega_{2n}^\epsilon(2)$-module for $M^\circ$,
		\item[(iii)] $[Y_1,Y_2] \not \subseteq Sing([Y_1,M_1^\circ]) \cap Sing([Y_2,M_2^\circ])$.
	\end{itemize}
	Then:
	\begin{itemize}
		\item[(a)] if $(2n,\epsilon) \ne (4,+)$, then Theorem \ref{theorem DL}(5) holds,
		\item[(b)] if $(2n,\epsilon) = (4,+)$, then Theorem \ref{theorem DL}(3) holds.
	\end{itemize}
	In both cases, if $Y = Y_M$, then $\< y \> \not \in \Delta$ for all $y \in [Y_M,M^\circ] \setminus Sing([Y_M,M^\circ])$.
\end{lemma}
\begin{proof}
	Let $H := \Aut_{\mathbb{F}_2}([Y,M^\circ])$,then \cite[Lemma B.35(d)]{Meier-Stell-Stroth:2012} shows that $N_H(\ov{M^\circ}) \cong \O_{2n}^\epsilon(2)$. Since $\ov{M} \ne \ov{M^\circ}$ we conclude that $\ov{M} \cong \O_{2n}^\epsilon(2)$ and $[Y,M^\circ]$ is the corresponding natural orthogonal module. Either $Y = [Y,M^\circ]$, hence $[Y,M^\circ] \le [Y,M] \le Y = [Y,M^\circ]$, or $ [Y,M^\circ] < Y$. In the latter case, \cite[Theorem C.22]{Meier-Stell-Stroth:2012} gives $\ov{M} \cong \O_6^+(2)$ and $|Y:[Y,M^\circ]|=2$, hence we always have $[Y,M^\circ]=[Y,M]$.\\
	
	Without loss of generality, suppose that $[Y_1,Y_2]$ contains a non-singular element $t$ of $[Y_1,M_1]$. Since $M_1^\dagger = M_1 C_{\L}(Y_1)$, $C_{M_1^\dagger}(t)$ doesn't contain a Sylow $2$-subgroup of $M_1^\dagger$.\\
	Assume first that $t \in Sing([Y_2,M_2])$; then $C_{M_2}(t)$ contains a Sylow $2$-subgroup of $M_2$. Since $|M_1|_2 = |M_2|_2$, if $C_\L(t) \subseteq M_1^\dagger$, then $C_{M_2}(t) \le M_1^\dagger$ and $C_{M_1^\dagger}(t)$ would contain a Sylow $2$-subgroup of $M_1^\dagger$, a contradiction. Thus $C_\L(t) \not \subseteq M_1^\dagger$.\\
	Assume now $t \not \in Sing([Y_2,M_2])$; since $M_2=M_1^u$ and $t$ is non-singular for $M_1$ in $[Y_1,M_1]$, also $t^u \not \in Sing([Y_2,M_2])$. Since $M_2$ acts transitively over the non-singular vectors, there is an element $m \in M_2$ such that $(t^u)^m =t$. Moreover, $(u,m) \in D(\L)$ via $(Y_1,Y_2,Y_2)$, hence we may write $t^{um} = t$. As $Y_1 \ne Y_1^{um} = Y_2$, $um \not \in N_\L(Y_1) = M_1^\dagger$, hence again $C_\L(t) \not \subseteq M_1^\dagger$.\\
	By \ref{lemma4.9} $Y_i$ is a $Q_i!$-module for $\ov{M_i}$, hence $Q_i \not \norm M_i$ and the pair $(M_1,Y_1)$ satisfies the hypothesis of Theorem~\ref{theorem CL}. In particular, if $\ov{M} \cong \O_4^+(2)$, then $\ov{M} \cong \SL_2(2) \wr C_2$, so Theorem~\ref{theorem DL}(3) holds. Otherwise Theorem~\ref{theorem DL}(5) holds.
\end{proof}

\begin{notation}[4.11]
	Set $\{i,j\} = \{1,2\}$,
	\[
	J_i := J_{M_i}(Y_i) \qquad \text{and} \qquad \ov{R_i} := F^*(\ov{J_i})
	\]
	and for $R_i$ the inverse image of $\ov{R_i}$ in $M_i$
	\[
	W_i :=[Y_i,R_i] \qquad \text{and} \qquad T_i:=Y_jR_i.
	\]
	Note that $T_i$ is well-defined as $Y_j \le M_i$ and $R_i \norm M_i$.
\end{notation}

\begin{lemma}[4.12] \label{lemma4.12}
	Suppose that case (1) of the \nameref{thm:Q!FF} holds; then Theorem~\ref{theorem DL}(3) holds.
\end{lemma}
\begin{proof}
	By hypothesis there exists an $\ov{M_1}$-invariant set $\mathcal{K}$ of subgroups of $\ov{M_1}$ such that $Y_1$ is a natural $\SL_2(q)$-wreath product module with respect to $\mathcal{K}$ and, moreover, $\ov{M_1^\circ} = O^p(\< \mathcal{K} \>) \ov{Q_1}$ and $Q_1$ acts transitively on $\mathcal{K}$.\\
	In addition, $J_i = \bigtimes_{K \in \mathcal{K}} K$ and $(M_i, Y_i, J, \mathcal{K})$ satisfies the hypothesis of \cite[Lemma A.27]{Meier-Stell-Stroth:2012} in place of $(H, V, P, \mathcal{K})$, so such a set $\mathcal{K}$ is unique.\\
	Suppose that $Y=Y_1=Y_M$; by \cite[Lemma 2.2(b)]{MS:2009} $Y_{M^\circ S} \le Y_M$ and the simplicity of $Y_M$ as a $M^\circ S$-module implies that $Y_{M^\circ S} =Y_M$. This shows D(3).
\end{proof}

\vspace{0.4cm}
We now look at the situation where case (2) of the \nameref{thm:Q!FF} holds; in particular, $W_1$ is a semisimple $J_1$-module and one needs here to distinguish the case where $W_1$ is an $R_1$-simple module from the case where it is not $R_1$-simple.\\
As usual, the local analysis goes along with the analysis undertaken in \cite{Meier-Stell-Stroth:2012}, as expected; however, here we see how arguments involving $G$-conjugates of $Q$ used to exclude some of the possible modules need further reasoning. Because of the complexity of the next proofs, we mainly continue recording the full argument adjoining to them all the required changes that the locality requires.

\begin{lemma}[4.13]\label{lemma4.13}
	Suppose that case (2) of the \nameref{thm:Q!FF} holds with $W_1$ not a simple $R_1$-module. Then~\ref{theorem DL}(4:4) holds.
\end{lemma}
\begin{proof}
	By \ref{lemma4.9}(c) either $Y_2=Y_1^u$ or $Y_1^{u\inv}$ is a quadratic non-trivial offender on $Y_1$; thus up to replacing $u$ with $u\inv$, we may suppose without loss of generality that $Y_2$ is a quadratic offender on $Y_1$ with $[Y_1,Y_2] \ne 1$. Select a minimal non-trivial offender $A \le Y_2$ on $Y_1$, then $A$ is a quadratic best offender on $Y_1$ (for example by \cite[Lemma A.39]{Meier-Stell-Stroth:2012}), thus $A \le J_1$.\\
	Recall that $Y_1Y_2 \le M_1 \cap M_2$ and let $I_i$ be a simple $R_i$-submodule of $W_i$, so that the \nameref{thm:Q!FF} may be applied. We follow the enumeration of \cite[Lemma 4.13]{Meier-Stell-Stroth:2012}; in particular, points \eqref{4.13-1} to \eqref{4.13-11} deal exclusively with the structure of local subgroups and, therefore, are  the same as in \cite[Lemma 4.13]{Meier-Stell-Stroth:2012}. We simply restate the findings, with little changes to the notation.
	
	\begin{subresult}\label{4.13-1}
		\begin{itemize}
			\item[(a)] $\ov{R_i}$ is quasisimple and contained in $\ov{M_i^\circ}$,
			\item[(b)] $C_{Y_i}(R_i) = 1$, $W_i$ is semisimple as a $J_i$-module and $\ov{M_i}$ acts faithfully on $W_i$.
		\end{itemize}
	\end{subresult}

	Let now $\{i.j\}=\{1,2\}$.
	\begin{subresult}\label{4.13-2}
		If $x \in M_j$ is such that $C_{I_j}(x) \ne 1$, then $I_j^x = I_j$.
	\end{subresult}

	\begin{subresult}\label{4.13-3}
		Let $X_i \le Y_i$ and $\cU_j := \{ I \le W_j \mid \text{ $I$ is a simple $R_j$-submodule of $W_j$} \}$. Suppose that $[X_i,Y_j] \ne 1$ but $X_i$ acts trivially on $\cU_j$. Then
		\[
		[\ov{R_j}, X_i] = \ov{R_j} \ne 1, \quad 1 \ne [I_j, X_i] \le [I_j,Y_i] \quad \text{and} \quad \text{$Y_i$ acts trivially on $\cU_j$.}
		\]		
	\end{subresult}

	\begin{subresult}\label{4.13-4}
		$T_i = Y_jR_i$ acts trivially on $\cU_i$; in particular, $W_i$ is a faithful and semisimple $\ov{T_i}$-module and $O_p(\ov{T_i}) = 1$.
	\end{subresult}

	\begin{subresult}\label{4.13-5}
		$C_{Y_i}(I_j) = C_{Y_i}(Y_j)$; in particular, $[I_1,I_2] \ne 1$.
	\end{subresult}

	\begin{subresult}\label{4.13-6}
		$W_i$ is not a self-dual $T_i$-module.
	\end{subresult}

	\begin{subresult}\label{4.13-7}
		Let $\bK_i := \End_{R_i}(I_i)$, then $\bK_i$ is a field (as $I_i$ is $R_i$-irreducible) and $T_i$ acts $\bK_i$-linearly on $I_i$.
	\end{subresult}
	
	\begin{subresult}\label{4.13-8}
		One of the following holds, with $q$ a power of $p$.
		\begin{itemize}[leftmargin=5em]
			\item[1)] \begin{itemize}
				\item[(a)] $\ov{R_i} = \ov{J_i} \cap \ov{M_i^\circ} \in \{ \SL_n(q), n \ge 3; \, \Sp_{2n}(q), n\ge 3; \, \SU_n(q), n \ge 8; \, \Omega^{\pm}_n(q), n \ge 10 \}$.
				\item[(b)] $W_1$ is the direct sum of at least two isomorphic natural modules for $\ov{R_1}$.
				\item[(c)] $\ov{M_i^\circ} = \ov{R_i} C_{\ov{M_i^\circ}}(\ov{R_i})$.
				\item[(d)] If $W_i < Y_i$, then $\ov{R_i} \cong \Sp_{2n}(q)$ with $p=2$ and $n \ge 4$.
			\end{itemize}
			\item[2)] $\, p=2$, $\ov{J_i} = \ov{R_i} \cong \SL_4(q)$ and $Y_i$ is the direct sum of two non-isomorphic natural modules for $\ov{R_i}$.
		\end{itemize}
	\end{subresult}
	
	\begin{subresult}\label{4.13-9}
		$I_i$ is a faithful $\bF_q T_i$-module.
	\end{subresult}
	
	\begin{subresult}\label{4.13-10}
		$R_i=J_i$, $\ov{J_i} \cong \SL_n(q)$ with $n \ge 3$, $W_i = Y_i$, $\ov{M_i^\circ} = \ov{J_i}C_{\ov{M_i^\circ}}(\ov{J_i})$ and $Y_i$ is the direct sum of $m \ge 2$ isomorphic natural modules for $\ov{J_i}$. In particular, $\bK_i \cong \bF_q$.
	\end{subresult}
	
	\begin{subresult}\label{4.13-11}
		$Y_j \le J_i$.
	\end{subresult}
	
	From this point on the general argument remains the same, but in the locality $\L$ we need to justify steps and add some more specific reasoning in multiple occasions. We therefore provide a full proof from now on.\\
	The idea developed by the authors of \cite[Lemma 4.13]{Meier-Stell-Stroth:2012} here consists in first proving \ref{theorem DL}(4:4) for all primes and afterwards deriving a contradiction to $p=2$. This is mainly carried by using the structure of $\ov{J_i}$ and looking at the interaction between the two $\ov{J_i}$-modules $I_1$ and $I_2$. The link between the two modules is provided by a suitable conjugate $Q_0$ of both $Q_1$ in $M_1$ and $Q_2$ in $M_2$; the same argument applies to our case with a little extra work.\\
	By \eqref{4.13-10} $Y_i$ is a homogeneous $\ov{J_i}$-module, so \cite[Lemma 5.2(d)]{MS:2008} shows that there exists an $S_i$-invariant simple $J_i$-submodule in $Y_i$; let $I_1$ and $I_2$ denote such submodules, with $I_2 = I_1^u$.\\
	Denote by $C_i$ the inverse image of $C_{\ov{M_i^\circ}}(\ov{J_i})$ in $M_i^\circ$ and, since $[I_1,I_2] \ne 1$ by \eqref{4.13-5}, pick $1 \ne x \in [I_1,I_2]$ and set $X_i := \bK_i x$.\\
	
	Recall that $I_i$ is a natural $\ov{J_i} \cong \SL_n(q)$-module and note that $J_iC_i/C_i \cong \PSL_n(q)$. Thus we have the following properties.
	\begin{enumerate}[i.]
		\item $\ov{J_i}$ is transitive on $I_i$.
		\item $O^{p'}(N_{\ov {J_i}}(X_i)) = C_{\ov{J_i}}(x)$.
		\item $O_p(C_{\ov{J_i}}(x))$ induces $\Hom_{\bK_i}(I_i/X_i,X_i)$ on $I_i$; thus $C_{I_i}(O_p(C_{\ov{J_i}}(x))) = X_i = [I_i, O_p(C_{\ov{J_i}}(x))]$.
		\item $O_p(C_{J_i}(x) C_i / C_i)$ is a natural $\SL_{n-1}(q)$-module for $C_{J_i}(x)$; as $n-1 \ge 2$, it is a non-central simple $C_{J_i}(x)$-module.
		\item If $W \le I_i$ is a $\bK_i$-subspace, then $O^{p'}(N_{\ov{J_i}}(W) / C_{\ov{J_i}}(W)) \cong \SL_{\bK_i}(W)$ and $N_{\ov{J_i}}(W)$ acts transitively on $W$.
	\end{enumerate}
	Since $I_i$ is $S_i$-invariant, thus also $Q_i$-invariant, $C_{I_i}(Q_i) \ne 1$; then (i.) implies that there exist elements $y_i \in J_i$ such that $x \in C_{I_i}(Q_i)^{y_i} = C_{I_i}(Q_i^{y_i}) \le Z(Q_i^{y_i})$, where $Q_i^{y_i}$ is well-defined as $Q_i, J_i \le M_i$. Since $Q_1=Q$ and it is large, $x^{y_1\inv} \in C_{I_1}(Q_1) \le Z(Q)$, which together with $Q$-repleteness implies that $\< x^{y_1\inv} \> \in \Delta$; since $I_j \le S$ for both $j \in \{1,2\}$, the transitivity of the action of $J_j$ shows that all the non-trivial elements of $I_1 \cup I_2$ generate a cyclic subgroup in $\Delta$. Thus $(y_1,y_2\inv,u\inv) \in D(\L)$ via $(\< x ^{y_1\inv}\>, \< x \>, \< x^{y_2\inv} \>, \< (x^{y_2\inv} )^{u\inv} \> ) $ and, therefore, also $(u,y_2, y_1\inv) \in D(\L)$. Since $x \in I_2 \cap S$ is centralized by $Q_2^{y_2}$ and $x^{y_1\inv} \in S$, conjugation by $y_1^{\inv}$ is well-defined on $Q_2^{y_2}$; whence $x^{y_1\inv} \in Z(Q_1) \cap Z(Q_2^{y_2 y_1\inv})$ and \ref{lemma: L containing Q}(c) applied to $Q_2^{y_2y_1\inv}= (Q_1^u)^{y_2y_1\inv}$ implies that $uy_2y_1\inv \in N_\L(Q_1)$, so $Q_1^{y_1} = Q_2^{y_2} =: Q_0 \le M_1 \cap M_2$.
	
	\begin{subresult}\label{4.13-12}
		$\displaystyle O_p \left( \frac{C_{J_i}(x)C_i}{C_i} \right) = \frac{Q_0C_i}{C_i} = \frac{(Q_0 \cap J_i)C_i}{C_i}.$
	\end{subresult}
	By \eqref{4.13-10} $\ov{M_i^\circ} = \ov{J_iC_i}$, so $Q_0 \le J_iC_i$ and $Q_0 \not \le C_i$. Since $C_{J_i}(x) C_i$ contains a Sylow $p$-subgroup of $J_iC_i$, $Q_0 \le C_{J_i}(x)C_i$ and $Q_0C_i / C_i$ is a non trivial $p$-subgroup of $C_{J_i}(x)C_i/C_i$, normal as $(Q!)$ implies $C_{J_i}(x) \le N_{J_i}(Q_0)$. Then (iv.) yields
	\[
	O_p \left( \frac{C_{J_i}(x)C_i}{C_i} \right) = \frac{Q_0C_i}{C_i} = \frac{[Q_0,C_{J_i}(x)]C_i}{C_i} \le\frac{(Q_0 \cap J_i)C_i}{C_i} \le O_p \left( \frac{C_{J_i}(x)C_i}{C_i} \right),
	\]
	which is \eqref{4.13-12}.
	
	\begin{subresult}\label{4.13-13}
		$Q_0 = (Q_0 \cap J_i)(Q_0 \cap C_i) = (Q_0 \cap J_i)C_{Q_0}(I_i) \le N_{M_i}(I_i).$
	\end{subresult}
	Indeed by \eqref{4.13-12} $Q_0C_i =(Q_0 \cap J_i)C_i$, which implies $Q_0 = (Q_0  \cap J_i)(Q_0 \cap C_i)$; now $Q_0 \le J_iC_i$ shows that $Q_0$ normalizes $I_i$ whereas $J_i$ centralizes the $p$-group $\ov{Q_0 \cap C_i}$ (by the definition of $C_i$) and has $I_i$ as a simple module. Suppose $C_{I_i}(\ov{Q_0 \cap C_i}) < I_i$; since $\ov{Q_0 \cap C_i}$ is a $p$-group, $C_{I_i}(\ov{Q_0 \cap C_i}) \ne 1$ and, therefore, it is a proper, non-trivial $J_i$-submodule of $I_i$, a contradiction. Thus $[Q_0 \cap C_i,I_i] =1$ and one gets $Q_0 = (Q_0 \cap J_i)(Q_0 \cap C_i) \le (Q_0 \cap J_i)C_{Q_0}(I_i) \le Q_0$, proving \eqref{4.13-13}.\\
	
	\begin{subresult}\label{4.13-14}
		$O_p(C_{\ov{J_i}}(x)) = \ov{Q_0 \cap J_i}$.
	\end{subresult}
	This follow from \eqref{4.13-12}, since $\ov{J_i}\cap \ov{C_i} \le Z(\ov{J_i})$ and it is a $p'$-subgroup (as $\ov{J_i} \cong \SL_n(q)$).
	
	\begin{subresult}\label{4.13-15}
		$X_i = [I_1,I_2] = [I_i,Q_0]=C_{I_i}(Q_0)$ and it is a $1$-dimensional $\bK_i$-subspace; in particular, $|[I_1,I_2]|=q$.
	\end{subresult}
	By (iii.) and \eqref{4.13-14} one then has $X_i = C_{I_i}(Q_0 \cap J_i) = [I_i, Q_0 \cap J_i]$; using \eqref{4.13-13} we get $X_i = C_{I_i}(Q_0) = [I_i,Q_0]$.\\
	The only equality that we are missing is that with $[I_1,I_2]$; note that this is a $\bK_i$-subspace of $I_i$ and, since $x \in [I_1,I_2]$, $[I_i,Q_0] = X_i \le [I_1,I_2]$. Thus $Q_0$ normalizes $I_i$, thus also $[I_1,I_2]$. Recall that $[I_1,I_2] \in \Delta$, so $H:=N_\L([I_1,I_2])$ is a subgroup of $\L$ and $H_i:= N_{J_i}([I_1,I_2])Q_0 \le H$. Set $I:= [I_1,I_2]$; since $I^{y_1\inv} \le I_1^{y_1\inv} = I_1 \le S$, conjugation by $y_1\inv$ induces a well-defined group isomorphism $N_\L(I) \fto N_\L(I^{y_1\inv})$; as $Q=Q_1 = Q_0^{y_1\inv} \le N_\L(I^{y_1\inv})$, \ref{lemma:weakly closed subgroups in localities}(a) shows that $Q$ is a weakly closed subgroup of $N_\L(I^{y_1\inv})$. Conjugating back by $y_1$ we see that $Q_0$ is weakly closed in $H$ and therefore in $H_1$ and also in $H_2$, since $Q_0=Q_1^{uy_2}$. Set $H_i^\circ := \< Q_0 ^{M_i} \>$: then \cite[Lemma 1.46(c)]{Meier-Stell-Stroth:2012} gives $H_i^\circ = \< Q_0 ^{H_i} \> = \< Q_0^{H_i^\circ} \>$. 
	Since $[I_i,Q_0]  = X_i \le [I_1,I_2]$ and both $[I_1,I_2]$ and $I_i$ are normalized by $H_i$, one has $[I_i, H_i^\circ] \le [I_1,I_2]$.
	By (v.) $H_i$ is transitive on $I$, so $I$ is a simple $H_i$-module; hence $[I_i,H_i^\circ] = I$; by the transitivity we may apply \cite[Lemma 1.57(c)]{Meier-Stell-Stroth:2012} considering $I$ a $Q_0!$-module for $H$ and obtain $Q_0^{H_i} = Q_0^{H}$, so $H_i^\circ = \< Q_0^H \> =: H^\circ$. Then $[I_i, H^\circ, I_j] = 1$ and the Three Subgroups Lemma yields $[I,H^\circ]=1$; hence from $Q_0 \le H^\circ$, $C_{I_i}(Q_0) = X_i$ and $X_i \le I$ follows that $I=X_i$, which is \eqref{4.13-15}.\\
	
	Set $Z_j:=[I_i,Y_j]$.
	
	\begin{subresult}\label{4.13-16}
		$\dim_{\bK_i}(Z_j) = m \ge 2$, where $m$ is given by the decomposition $Y_i=\displaystyle \bigoplus_{h=1}^m A_h$ with the $\ov{J_i}$-modules $A_h$ are all isomorphic natural modules for $J_i$ (by \eqref{4.13-10} $\ov{J_i} \cong \SL_n(q)$).
	\end{subresult}
	The structure of $Y_i$ together with $m \ge 2$ is given by \eqref{4.13-10}; clearly here $A_h \cong I_i$ for every direct summand $A_h$. By \eqref{4.13-11} $Y_i \le J_j$ and $Y_j \le J_i$, so $Y_j$ and $I_i$ act one on each other; thus \eqref{4.13-15} implies $|[Y_j,I_i]|=|I|^m = q^m$ and the $\bK_i$-linear action of $Y_j$ on $I_i$ shows that $Z_j$ is a $\bK_i$-subspace of $I_i$.\\
	For future use, note also that $I \le Z_j$, so $Z_j \in \Delta$ and $K_j:=N_\L(Z_j)$ is a subgroup of $\L$. 
	
	\begin{subresult}\label{4.13-17}
		$Y_i \not \le Q_i$.
	\end{subresult}
	Assume $Y_i \le Q_i$ for a contradiction. Since $Q_i$ and $Q_0$ are $M_i$-conjugate and since $Y_i \norm M_i$ one has $Y_i \le Q_0$, hence $Y_i \le J_i \cap Q_0 \le O_p(C_{J_i}(x))$. By (iii.) $[I_j,Y_i] \le X_j$, so $Z_i \le X_j$, a contradiction since $m \ge 2$ by \eqref{4.13-10} and $X_j$ is $1$-dimensional over $\bK_j$.
	
	\begin{subresult}\label{4.13-18}
		$Q_0C_j \le K_j$ and $\ov{C_j}$ acts faithfully on $Z_j$.
	\end{subresult}
	$Q_0$ normalizes $I_i$ by \eqref{4.13-13} and $Y_j$ as $Q_0 \le M_j$; then $Q_0 \le K_j$. By \eqref{4.13-11} $I_i \le Y_i \le  J_j \le C_{M_j}(\ov{C_j})$, thus for $c \in C_j$ one has 
	\[
	[I_i,Y_j]^c = [I_i^c,Y_j^c] \le [I_iC_{M_j}(Y_j), Y_j] \le [C_{M_j}(Y_j),Y_j]^{I_i} [I_i,Y_j] =[I_i,Y_j],
	\]
	proving that $C_j \le K_j$.\\
	We now prove $C_{\ov{C_j}}(Z_j) \le C_{\ov{C_j}}(Y_j) = 1$. The simplicity of $I_j$ as $J_j$-module yields $\< [I_i,I_j]^{J_j} \> = I_j$; since $Y_j=\bigoplus_{h=1}^m A_h$ for $\ov{J_j}$-modules $A_h \cong I_j$, $\< Z_j^{J_j} \> =\< [I_i, Y_j]^{J_j} \>  = Y_j$ and the assertion follows as $[\ov{C_j},J_j] =1$.\\
	
	For the reminder of the proof set $K_j^\circ := \wcl_{K_j}(Q_0,K_j)$.
	
	\begin{subresult}\label{4.13-19}
		The following equalities hold.
		\begin{multline*}
			\< Q_0^{J_i \cap K_j} \> = K_j^\circ = \wcl_{K_j}(Q_0,(J_i \cap K_j)Q_0) = \\
			= \wcl_{K_j}(Q_0,(J_i \cap K_j)C_{M_i}(Z_j)) = \wcl_{M_i}(Q_0,(J_i \cap K_j)C_{M_i}(Z_j)) .
		\end{multline*}
	\end{subresult}
	We know that $Q_0 = (Q_0 \cap J_i)C_{Q_0}(I_i)$ by \eqref{4.13-13}; as $Q_0 \le K_j$ and $Z_j \le I_i$ we have
	\[
	Q_0 = (Q_0 \cap J_i)C_{Q_0}(Z_j) \le (K_j \cap J_i)C_{Q_0}(Z_j).
	\]
	Since by \eqref{4.13-16} $Z_j$ is a $\bK_i$-subspace of $I_i$, (v.) implies that $J_i \cap K_j$ is transitive on $Z_j$. Since $Z_j \le Y_j$ and $Q_0$ is a $M_j$-conjugate of $Q_j$, there exists elements $y_1 \in M_1$ and $y_2 \in M_2u\inv$ such that $Q_0^{y_j}=Q$ and $Z_j^{y_j} \le S$; thus conjugation by $y_j$ extends to $K_j$ and we see that $Q \le K_j^{y_j}$, hence it is large in $K_j^{y_j}$. Thus $Q_0$ is large in both $K_1$ and $K_2$ and $Z_j$ is a $Q_0!$-module for $K_j$ by \ref{lemma1.57}. The same conclusion also holds now for $(J_i \cap K_j)Q_0$ and for $(J_i \cap K_j)C_{M_i}(Z_j)$ in place of $K_j$, since both are contained in $K_j$, and for $(J_i \cap K_j)C_{M_i}(Z_j) \le M_i$ since $Q_0$ is also large in $M_i$ and $Z_j$ is, therefore, a $Q_0!$-module for $(J_i \cap K_j)C_{M_i}(Z_j)$. \\
	Now four applications of \ref{lemma1.57}(c) to $(Q_0, K_j, Z_j)$, $(Q_0, (J_i \cap K_j)Q_0, Z_j)$ and $(Q_0, (J_i \cap K_j)C_{M_i}(Z_j), Z_j)$ (twice) in place of $(Q,M,V)$, with $J_i \cap K_j$ in place of $U$, yield
	\[
	\< Q_0^{K_j} \> = \<Q_0^{J_i \cap K_j} \> = \<Q_0^{M_i}\>
	\] which implies all the equalities.
	
	\begin{subresult}\label{4.13-20}
		$Z_j$ is a natural $\bK_i \SL_m(q)$-module for $K_j^\circ$.
	\end{subresult}
	By \eqref{4.13-16} $Z_j$ is a $\bK_i$-subspace of $I_i$. Set $\widetilde{K_j} := K_j / C_{K_j}(Z_j)$; by (v.) $Z_j$ is a natural $\SL_m(q)$-module for $O^{p'}(J_i \cap K_j)$ and \eqref{4.13-19} implies $\widetilde{K_j^\circ} \le O^{p'}(\widetilde{J_i \cap K_j}) \cong \SL_m(q)$. Since $K_j^\circ$ is generated by $p$-elements and $\SL_m(q)$ has no non-trivial, proper, normal subgroup generated by $p$-elements we obtain $\widetilde{K_j^\circ} = O^{p'}(\widetilde{J_i \cap K_j})$.
	
	\begin{subresult}\label{4.13-21}
		$Q_0 \cap C_j$ induces $\Hom_{\bK_i}(Z_j/X_i, X_i)$ on $Z_j$.
	\end{subresult}
	Recall that $I=[I_1,I_2]$ and, by \eqref{4.13-15}, $[Q_0,I]=1$; as $Z_j=[I_i,Y_j]$ and $Y_j$, as $J_j$-module,  is a direct sum of copies of $I_j$ we get $[Q_0 \cap J_j, Z_j] =1$. Thus, using \eqref{4.13-13}, $Q_0=(Q_0\cap C_j)(Q_0 \cap J_j) = (Q_0 \cap C_j)C_{Q_0}(Z_j)$.\\
	By \eqref{4.13-14} and (iii.) $\ov{Q_0 \cap J_i} = O_p(C_{\ov{J_i}}(x))$ and it induces $\Hom_{\bK_i}(I_i/X_i,X_i)$ on $I_i$; since $X_i \le Z_j$ by \eqref{4.13-15} and $Z_j$ is a $\bK_i$-subspace by \eqref{4.13-16} we conclude that $Q_0 \cap J_i$ induces $\Hom_{\bK_i}(Z_j/X_i,X_i)$ on $Z_j$. Now \eqref{4.13-21} follows since $(Q_0 \cap C_j)C_{Q_0}(Z_j) = Q_0 =(Q_0 \cap J_i)C_{Q_0}(Z_j)$.\\
	
	Recall that, according to our notation convention,
	\begin{align*}
		\ov{X} := \frac{X C_{M_j}(Y_j)}{C_{M_j}(Y_j)} \qquad & \text{for all $X \le M_j$,} \\
		\widetilde{L} := \frac{L C_{K_j}(Z_j)}{C_{K_j}(Z_j)} \qquad  & \text{for all $L \le Z_j$,}
	\end{align*}
	and set
	\[
	J_j^\star := \< (Q_0 \cap J_j)^{J_j} \> \qquad \text{and} \qquad C_j^\star := \< (Q_0 \cap C_j)^{C_j} \>. 
	\]
	
	\begin{subresult}\label{4.13-22}
		$\ov{M_j^\circ} = \ov{J_j}\ov{C_j^\star}$, $C_j^\star \le K_j^\circ$ and $\widetilde{C_j^\star} = \widetilde{K_j^\circ} \cong \SL_m(q)$; thus $Z_j$ is a faithful, natural $\SL_m(q)$-module for $\widetilde{C_j^\star}$.
	\end{subresult}
	By \eqref{4.13-10} we have $R_i = J_i$ and $\ov{J_i} \cong \SL_n(q)$ with $n \ge 3$; by \eqref{4.13-14} also $\ov{Q_0 \cap J_j} = O_p(C_{\ov{J_i}}(x))$. Again, since $\SL_n(q)$ contains no non-trivial, proper normal subgroup generated by $p$-elements, we get $\ov{J_j^\star}=\ov{J_j}$ and, together with \eqref{4.13-8}(1:c), $\ov{M_j^\circ} = \ov{J_jC_j}$. By the definition of $C_j$, $[\ov{J_j}, \ov{C_j}]=1$ and by \eqref{4.13-13} $\ov{Q_0} = (\ov{Q_0 \cap J_j})(\ov{Q_0 \cap C_j})$, whence
	\begin{equation*}
		\label{eq-4.13-20}
		\tag{$\dagger$}
		\ov{M_j^\circ} = \< \ov{Q_0}^{\ov{M_j^\circ}} \> = \ov{J_j^\star}\ov{C_j^\star} \qquad \text{and} 	\qquad [\ov{J_j^\star}, \ov{C_j^\star}] = 1.
	\end{equation*}
	Since by \eqref{4.13-18} $\ov{C_j}$ acts faithfully on $Z_j$ we get $C_{C_j^\star}(Z_j) = C_{C_j^\star}(Y_j)$, hence $\ov{C_j^\star} = \widetilde{C_j^\star}$. We also have $O_p(\ov{C_j^\star}) \le O_p(\ov{C_j}) \le O_p(\ov{M_j}) =1$ and clearly $C_j^\star \le K_j^\circ$. Now by \eqref{4.13-20} and \eqref{4.13-21} $Z_j$ is a natural $\SL_m(q)$-module for $\widetilde{K_j}$ and $Q_0 \cap C_j$ induces $\Hom_{\bK_i}(Z_j/X_i,X_i)$ on $Z_j$, thus \cite[Lemma 7.2]{MS:2008} implies \eqref{4.13-22}.\\
	
	We can now complete the proof: as usual, the properties concerning the modules and related arguments are as in \cite[Lemma 4.13]{Meier-Stell-Stroth:2012}, however we will need to work harder to prove that $p \ne 2$.\\
	Recall that, seen as a $J_j$-module, $Y_j = \bigoplus_{h=1}^m A_h$ with all the $A_h$ natural $\SL_n(q)$-modules isomorphic to $I_j$; thus \cite[Lemma 5.2]{MS:2008} shows that, as a $\ov{M_j^\circ}=\ov{J_jC_j^\star}$-module, $Y_j \cong I_j \otimes_{\bK_i} U_j$ for some $\bK_i C_j^\star$-module $U_j$. Being $[I_1,I_2]$ $1$-dimensional over $\bK_i$ we have
	\[
	U_j \cong [I_1,I_2] \otimes U_j = [I_j \otimes U_j, I_i] \cong [Y_j,I_i] =Z_j
	\]
	when seen as a $C_j^\star$-module; thus $U_j$ is a natural $\SL_m(q)$-module for $C_j^\star$. By \eqref{eq-4.13-20} and the structure of $\ov{J_j}$ and $\ov{C_j^\star}$, $\ov{M_j^\circ}$ is a central product of these two and $Y_j$ a corresponding tensor product module.\\
	
	We are now left to prove that $p$ is odd.\\
	By \eqref{4.13-16} and since $Z_2 \le Y_1$, $Z_2^u$ is an $m$-dimensional $\bK_2$-subspace of $I_2$ and $I^u = [I_1,I_2]^u \le Z_2^u$; in particular, both $I$ and $I^u$ are $\bK_2$-subspaces of $I_2$. Also $Z_1$ is, again by \eqref{4.13-16}, a $\bK_2$-subspace of $I_2$ of dimension $m$. Since $I_2$ is a natural $\SL_m(q)$-module for $J_2$, $J_2$ is transitive on pairs of incident $1$- and $m$-dimensional subspaces of $I_2$. Thus there exists $v \in J_2$ with
	\[
	(Z_2^u)^v = Z_1 \qquad \text{and} \qquad (I^u)^v = I.
	\]
	Since $Z_2 \le Y_1$ and $Z_2 \in \Delta$, we have $Z_2^u \le Y_2  \le S$ and, therefore, $(u,v) \in D(\L)$ via $(Z_2,Z_2^u, (Z_2^u)^v)$. Set $g:=uv$, so that
	\[
	I^g=I, \quad I_1^g = I_2^v = I_2, \quad Y_1^g = Y_2, \quad Z_2^g=Z_1
	\]
	and, since $I_i \le Y_i \le S \le M$, one gets
	\[
	Z_1^g = [I_2^g,Y_1^g] = [I_2^g, Y_2] \qquad \text{and} \qquad [I_2^g,I_2] = [I_2,I_1]^g = I^g =I.
	\]
	Now as $I_2 \le Y_2 \le J_1$, $I_2^g \le J_1^g=J_2$; also $I_1 \le J_2$ and, being $Y_2$ a direct sum of copies of $I_2$ and $[I_2^g,I_2]=I$, we get $Z_1^g = [I_2,Y_1]^g =[I_2^g,Y_2] = [I_1,Y_2] = Z_2$. Hence we have found $g \in N_\L(I)$ acting non-trivially on $\{Z_1,Z_2\}$; in particular, all the powers of $g$ are defined, so up to replacing $g$ by a suitable power we may suppose $o(g)=2$.\\
	
	By the structure of the natural $\SL_n(q)$-module, if we set
	\[
	\mathcal{X} := \{(x_0,Z_0) \mid \textit{$Z_0$ is an $m$-dimensional $\bK_1$-subspace of $I_1$ and $0 \ne x_0 \in Z_0$}\},
	\]
	then $	p \not \mid |\mathcal{X}|$ and $J_1$ acts transitively on $\mathcal{X}$. Since $N_{M_1}(I_1)$ contains $S_1J_1$, we conclude that every element of $\mathcal{X}$ is fixed by some Sylow $p$-subgroup of $N_{M_1}(I_1)$, which is a parabolic subgroup of $M_1$. As $(x,Z_2) \in \mathcal{X}$, there exists $S_0 \in \Syl_p(N_{M_1}(I_1))$ such that $S_0 \le N_{M_1}(Z_2) \cap C_{M_1}(x) \cap N_{M_1}(I_1)$. By \eqref{4.13-10} we have
	\[
	Z_1 = [Y_1,I_2] = \bigoplus_{h=1}^m [A_h,I_2]
	\]
	where the $A_h$ are natural $\ov{J_1} \cong \SL_n(q)$-natural modules isomorphic to $I_1$. In particular, we may choose $A_1=I_1$. Set
	\[
	\M := \{A_h \mid Y_1=\bigoplus_{h=1}^m A_h\} \quad \text{and} \quad U:= C_{\ov{J_1}}(I_1 / I) \cap C_{\ov{J_1}}(I).
	\]
	Note that $I_2 \le U$ and therefore $ [I_1,I_2]=[I_1,U]$; if $\phi: I_1 \fto A_h$ is a $J_1$-isomorphism, then $\phi([I_1,U]) = [\phi(I_1),U] =[A_h,U]$ and we obtain
	\[
	Z_1 = \bigoplus_{A \in \M} [A,U].
	\]
	Let now $x \in S_0$ and set $\psi:= c_{x} \circ \phi \circ c_{x\inv}$, that is the map making the following diagram commutative:
	\begin{center}
		\begin{tikzcd}
			I_1 \ar[r,"\phi"]		& A_h \ar[d,"c_x"] \\
			I_1^x \ar[r,"\psi"] \ar[u,"c_{x\inv}"]		& A_h^x
		\end{tikzcd}
	\end{center}
	Then $\psi$ is a $J_1$-isomorphism, since for any $t \in J_1$ and $m \in I_1$, $\psi((m^x)^t) = c_x(\phi(m^{xtx\inv})) = c_x((\phi(m))^{xtx\inv}) = \phi(m)^{xt} = (\psi(m^x))^t$. As this holds for any two modules in $\M$, we obtain that the elements of $\M^x:=\{A^x \mid A \in \M\}$ are all isomorphic $J_1$-modules and $I_1 = I_1^x \in \M^x$. Moreover,
	\[
	Y_1 = Y_1^x =\bigoplus_{A^x \in \M^x} A^x \quad \text{and therefore} \quad Z_1 = \bigoplus_{A^x \in \M^x} [A^x,U].
	\]
	Since $S_0$ normalizes $U$, we conclude that 
	\[
	Z_1^x = \bigoplus_{A \in \M} [A,U]^x = \bigoplus_{A \in \M} [A^x,U] = \bigoplus_{A^x \in \M^x} [A^x,U] = Z_1;
	\]
	as it holds for any $x \in S_0$, we have $S_0 \le N_\L(Z_1)$.\\
	Recall that $I=\bK_1 x$ where $\bK_1 = \End_{R_1}(I_1)$ and $\ov{R_1} = F^*(\ov{J_1})$, thus $R_1 \norm M_1$. Then $S_0$ normalizes $\bK_1$, hence $S_0 \le N_\L(I)$ and we have $B:= \< S_0,g \> \le N_\L(I)$ and also $\<Z_1,Z_2\> \le N_\L(I)$. In particular, $B$ acts by conjugation on $\{Z_1,Z_2\}$, with $S_0$ acting trivially and $g$ as a transposition.\\
	
	Assume that $p=2$; since $S_0$ is a conjugate of $S$, it is a Sylow $2$-subgroup of $B$, hence there exists $b \in B$ such that $g^b \in S_0$. Then $g^b$ acts trivially on $\{Z_1,Z_2\}$, a contradiction to $g$ acting as a transposition.
\end{proof}

\vspace{0.4cm}
We now deal with the last case, i.e. when $W_1$ is $R_1$-simple. Prior to this we need to address a separate issue.
Consider $Q \le S$ (or any weakly closed subgroup of $\L$ contained in $S$); $Q$ is then weakly closed in $\F$, however in the locality $\L$ we may have the possibility to conjugate $Q$ through a chain of elements $(x_1, \dots, x_k) \in \L^k$ that carry $Q$ ``out of $S$ and back into $S$". In details, setting $Q=:Q_0$ and $Q_l := (Q_{l-1})^{x_l}$, we might have $Q_j \not \le S$ for some $j \in \{1,\dots,l-1\}$ and $Q_k \le S$; now, if $(x_1, \dots, x_k) \not \in \W(\L)$, the fact that $Q$ is weakly closed is no longer sufficient to guarantee that $Q_k = Q$.\\
The next lemma addresses this issue; the formulation here presented is a generalization of ideas of the author obtained together with Ellen Henke.

\begin{lemma}\label{lemma:conj of p-subgr and fus systems}
	Let $(\L,\Delta,S)$ be a locality over the saturated fusion system $\F$. Define $\wht{\F} := \widehat{\F_S(\L)}$ to be the fusion system on $S$ generated by the maps of the form
	\[
	P_0 \overset{c_1}{\fto} P_1 \overset{c_2}{\fto} \dots \overset{c_{n}}{\fto} P_n
	\]
	where
	\begin{enumerate}[(i)]
		\item each $P_i$ is a $p$-subgroup of $\L$,
		\item $P_0$ and $P_n$ are subgroups of $S$,
		\item for every $i \in \{1,\dots,n\}$ there exist elements $f_i \in \L$ and subgroups $Y_{i} \in \Delta$ such that $Y_{i} \le S_{f_i}$, $P_{i-1} \le N_\L(Y_{i})$ and $c_i = c_{f_i}$.
	\end{enumerate}
	Then $\F = \wht{\F}$.
\end{lemma}
\begin{proof}
	Note that, for every $i \in \{1,\dots,n\}$, conjugation by $f_i$ extends to $c_i=c_{f_i} \colon N_\L(Y_i) \fto N_\L(Y_i^{f_i})$; as $P_{i-1} \le N_\L(Y_i)$, it also induces a group homomorphism $c_i \colon P_{i-1} \fto P_{i-1}^{f_i} = c_i(P_{i-1}) \le P_i$. Hence $f_i$ induces a group isomorphism, via conjugation, on $P_{i-1}$ and, without loss of generality, we may assume $P_i = (P_{i-1})^{f_i}$. \\
	
	The inclusion $\F \subseteq \wht{\F}$ is obvious, so we are only concerned with the reverse one.\\
	By our hypothesis we have $R_i:= Y_i^{f_i} \in \Delta$; in particular we get
	\begin{equation*}
		\label{eq:inclusion poperty of the P_i}
		\tag{$\star$}
		P_{i-1}^{f_i} = P_i \le N_\L(R_i) \cap N_\L(Y_{i+1}) \quad \textit{for all $i \in \{1,\dots,n\}$. }
	\end{equation*}
	Therefore $P_i \le N_\L(\< R_i,Y_{i+1} \>)$, where $\<R_i,Y_{i+1} \> \in \Delta$ since it is a subgroup of $S$ containing, for example, $Y_{i+1}$. Since $P_i$ normalizes $R_i$ and $Y_{i+1}$, $P_i \, \<R_i,Y_{i+1} \>$ is a $p$-subgroup of $\L$, hence there exists $g_i \in \L$ such that $(P_i \, \<R_i,Y_{i+1}\>)^{g_i} \le S$; in particular, $R_i^{g_i}, \, Y_{i+1}^{g_i} \in \Delta$.\\
	Note now that we have
	\begin{align*}
		(f_1,g_1) \in D(\L) & \textit{ via } (Y_1, R_1, R_1^{g_1}) \\
		(g_{i-1}\inv,f_i,g_i) \in D(\L) & \textit{ via } (Y_i^{g_{i-1}},Y_i, Y_i^{f_i}=R_i, R_i^{g_i} ) \; \textit{ for all $i \in \{2,\dots,n-1\}$} \\
		(g_{n-1}\inv, f_n) \in D(\L) & \textit{ via } (Y_n^{g_{n-1}}, R_n)
	\end{align*}
	Set $\xi := (c_{f_n} \circ \dots \circ c_{f_2} \circ c_{f_1})|_{P_0}$; then $\xi$ is the composition of restrictions of conjugation maps $c_{f_i}: N_\L(Y_i) \fto N_\L(R_i)$, where $Y_i,R_i \in \Delta$, in detail
	\[
	\xi =  c_{f_n}|_{P_{n-1}} \circ \dots \circ c_{f_2}|_{P_1} \circ c_{f_1}|_{P_0} .
	\]
	Let now $x \in P_0 \le N_\L(Y_1)$; then we have $(g_1\inv, f_1 \inv , x, f_1, g_1) \in D(\L)$ via $(R_1^{g_1}, R_1, Y_1, Y_1, Y_1^{f_1} = R_1, R_1^{g_1})$.\\
	For every $i \in \{2,\dots,n-1\}$ we have seen that $w_i :=(g_i\inv, f_i,g_i) \in D(\L)$, and also $w_1:=(f_1,g_1), \, w_n:=(g_{n-1}\inv,f_n) \in D(\L)$. Setting $u_i:=\Pi(w_i)$ we get
	\begin{align*}
		c_{u_1}& \colon 	 P_0  				  \fto 		 P_1^{g_1}, \\
		c_{u_2}& \colon 	 P_1^{g_1} 		  \fto  		 P_2^{g_2}, \\
		&			\qquad			  \vdots \\
		c_{u_n}& \colon	 P_{n-1}^{g_{n-1}}   \fto 		 P_n\\
	\end{align*}
	where all the subgroups $P_0, \, P_1^{g_1}, \, \dots, \,P_{n-1}^{g_{n-1}}, \, P_n \le S$. By construction, the following diagrams, where elements of $\L$ denote conjugations and $i \in \{2,\dots,n-1\}$, are commutative:
	\begin{center}
		\begin{tikzcd}
			P_0 \ar[r,"f_1"] \ar[rd,"u_1",swap] 	& P_1 \ar[d,"g_1"] 				&[30pt] P_{i-1} \ar[r,"f_i"]	& P_i \ar[d,"g_i"]  &[30pt] P_{n-1} \ar[r,"f_n"]	& P_n \\
			& P_1^{g_1} &  P_{i-1}^{g_{i-1}} \ar[u,"g_{i-1}\inv"] \ar[r,"u_i"]	& P_i^{g_i}	& P_{n-1}^{g_{n-1}} \ar[u,"g_{n-1}\inv"] \ar[ur,"u_n",swap]
		\end{tikzcd}
	\end{center}
	
	Thus we have 
	\[
	(c_{u_n} \circ \, \dots \, \circ c_{u_2} \circ c_{u_1})|_{P_0} = c_{u_n}|_{P_{n-1}^{g_{n-1}}} \circ \, \dots \, \circ c_{u_2}|_{P_1^{g_1}} \circ c_{u_1}|_{P_0}  = c_{f_n}|_{P_{n-1}} \circ \, \dots \, \circ c_{f_2}|_{P_1} \circ c_{f_1}|_{P_0} = \xi,
	\]
	proving that $\xi \in \Hom_\F(P_0,P_n)$.
\end{proof}

\begin{lemma}[4.14]\label{lemma4.14}
	Suppose that case (2) of the  \nameref{thm:Q!FF} holds with $W_1$ a simple $R_1$-module. Then Theorem~\ref{theorem DL} holds.
\end{lemma}
\begin{proof}
	As usual, the proof mimics that of \cite[Lemma 4.14]{Meier-Stell-Stroth:2012} as it deals almost exclusively with the local structure. The main difference will appear when the properties of $Q$ and of its conjugates become relevant; we will show that the $M_1 \cup uM_2$-conjugates of $Q$ share the same behavior and properties as in \cite[Lemma 4.14]{Meier-Stell-Stroth:2012} at least with respect to the $M_i$-modules $Y_i$. This will afterwards reveal to be sufficient to undertake a similar exclusion process during the case-analysis.\\
	For this lemma we will need to follow a different numbering with respect to \cite[Lemma 4.14]{Meier-Stell-Stroth:2012}.
	
	\begin{subresult}\label{4.14-1}
		Looking at the conclusion of \nameref{thm:Q!FF} we find out that the following hold.
		\begin{enumerate}[(a)]
			\item $\ov{R_i}$ is quasisimple, $\ov{R_i} \le \ov{M_i^\circ}$ and one concludes that:
			\begin{itemize}
				\item[i)] $\ov{R_i} =\ov{J_i}$;
				\item[ii)] $p=2$ and $\ov{J_i} \in \{ \O^\pm_{2n}(q), \Sp_4(2), G_2(2)\}$.
			\end{itemize}
			\item $C_{Y_i}(R_i) =1$ and the action of $\ov{M_i}$ on $W_i$ is faithful; thus $C_{M_i}(W_i) = C_{M_i}(Y_i)$.
			\item One of the following holds.
			\begin{itemize}
				\item[i)] $\ov{M_i^\circ} = \ov{R_i} = \ov{M_i^\circ} \cap \ov{J_i}$.
				\item[ii)] $\ov{M_i^\circ} \in \{\Sp_4(2), 3 \cdot \Sym(6), \SU_4(q).2, G_2(2)\}$ for $p=2$ and if $\ov{M_i^\circ} \cong \SU_4(q).2 \cong \O^-_6(q)$, then $W_i$ is a natural $\SU_4(q)$-module.
			\end{itemize}
			\item The pair $(\ov{J_i}, W_i)$ is one among the cases (1)-(9),(12) of \cite[Theorem C.3]{Meier-Stell-Stroth:2012}, where $n \ge 3$ in case (1), $n \ge 2$ in case (2) and $n=6$ in case (12).
		\end{enumerate}
	\end{subresult}
	
	Recall that $M_{i\circ} = O^p(M_i^\circ)$.
	\begin{subresult}\label{4.14-2}
		$\ov{R_i}=F^*(\ov{M_i^\circ})=\ov{M_{i\circ}}$, so $W_i=[Y_i,M_{i\circ}]$.
	\end{subresult}
	We use \eqref{4.14-1}(c). Suppose first $\ov{R_i} = \ov{M_i^\circ}$; then \eqref{4.14-2} follows immediately since $\ov{R_i}$ is quasisimple by \eqref{4.14-1}(a). The equality with $\ov{M_{i \circ}}$ follows since $O^p(\ov{R_i}) < \ov{R_i}$ implies $O^p(\ov{R_i}) \le Z(\ov{R_i})$, so $\ov{R_i}$ is nilpotent and therefore $\ov{R_i} \cong O^p(\ov{R_i}) \times C_p = O_{p'}(\ov{R_i}) \times C_p$, which is abelian, a contradiction.  \\
	If $\ov{R_i} \ne \ov{M_i^\circ}$, then \eqref{4.14-1}(c:ii) shows that $F^*(\ov{M_i^\circ})$ is quasisimple of index $2$ in $\ov{M_i^\circ}$ and $p=2$ in all cases; thus $F^*(\ov{M_i^\circ}) = O^2(\ov{M_i^\circ}) =\ov{M_{i\circ}}$. Since $\ov{R_i} \le \ov{M_i^\circ}$ and $\ov{R_i}$ is quasisimple and normal in $\ov{M_i}$ we obtain $\ov{R_i} = F^*(\ov{M_i^\circ})$, proving \eqref{4.14-2}.\\
	
	Observe next that one of the following holds:
	\begin{itemize}
		\item[i.] $\displaystyle \left| \frac{W_1}{C_{W_1}(W_2)} \right| \le \left| \frac{W_2}{C_{W_2}(W_1)} \right|$,
		\vspace{0.1cm}
		\item[ii.] $\displaystyle \left| \frac{W_2}{C_{W_2}(W_1)} \right| \le \left| \frac{W_1}{C_{W_1}(W_2)} \right|$.
	\end{itemize}
	If (ii.) holds, then $|W_1/C_{W_1}(W_1^{u\inv})| \le |W_1^{u\inv} / C_{W_1^{u\inv}}(W_1) |$ and, up to replacing $W_2$ with $W_1^{u\inv}$ (i.e. $u$ with $u\inv$), we may suppose that $W_2$ is an offender on $W_1$. Lemma \ref{lemma4.2} shows that conjugation by $u\inv$ doesn't affect our setting. \\
	
	Set $Z:=[W_1,W_2]$ and, against the usual convention adopted until now, let the bar notation denote quotients of subgroups of $M_i$ over $C_{M_j}(Y_j)$, so
	\[
	\ov{W_i} := \frac{W_i C_{M_j}(Y_j)}{C_{M_j}(Y_j)}.
	\]
	Then \eqref{4.14-1}(b) gives $C_{M_j}(Y_j)=C_{M_j}(W_j)$, hence $\ov{W_i} \cong W_i/C_{W_i}(W_j)$ as an $M_1 \cap M_2$-module.
	
	\begin{subresult}\label{4.14-3}
		$W_2$ is a non-trivial, quadratic offender on $W_1$; thus $Z \ne 1$.
	\end{subresult}
	By \ref{lemma4.4} and \ref{notation4.3} $Y_2$ acts quadratically and non-trivially on $Y_1$, hence $W_2$ is a quadratic offender on $W_1$. By \eqref{4.14-1}(b) $C_{M_i}(Y_i) =C_{M_i}(W_i)$. Taking first $i=1$, then $[Y_1,Y_2] \ne 1$ implies $[W_1,Y_2] \ne 1$; taking $i=2$ we have that $[W_1,Y_2] \ne 1$ implies now $[W_1,W_2]\ne 1$.\\
	
	Set $\bK_i := \End_{R_i}(W_i)$ and recall that $T_j = Y_iR_j$.

	\begin{subresult}\label{4.14-7}
		$T_j$ acts $\bK_j$-linearly on $W_j$; hence $Z$ is a $\bK_j$-subspace of $W_j$.
	\end{subresult}
	Suppose that the action of $W_i$ on $W_j$ is not $\bK_j$-linear. Then $p < |\bK_j|=|\bK_i|$ and one can apply \cite[Lemma 1.22]{Meier-Stell-Stroth:2012} and obtain $\dim_{\bK_i}(W_i) =1$; this means that $R_i$ acts as scalar multiplications over $W_i$, thus $[R_i,R_i]$ acts trivially on $W_i$. However $\ov{R_i}$ is, by \eqref{4.14-1}(a), quasisimple, so $[\ov{R_i},\ov{R_i}]=\ov{R_i}$, yielding a contradiction since $W_i$ is a non-central $R_i$-module. We have shown that $W_i$ acts $\bK_j$-linearly on $W_j$; as $Y_i$ normalizes $R_j$, it induces field automorphisms of $\bK_j$, hence it acts semilinearly on $W_j$. Moreover, since $Y_i$ acts quadratically on $W_j$, it centralizes the non-trivial $\bK_j$-subspace $[W_j,W_i] \le W_j$, thus it must act linearly on the entire $W_j$.\\	
	
	We now turn our attention to the situation regarding $M_1$- and $uM_2$- conjugates of $Q$. 	
	\begin{subresult}\label{subres:4.14-conjugates of Q}
		The following hold.
		\begin{itemize} 
			\item[(i)] Let $x_i \in M_i$; if $Q_i^{x_i} \le M_j$, then there exists $x_j \in M_j$ such that $Q_i^{x_i} = Q_j^{x_j}$.
			\item[(ii)] If $[v,Q_1^{y_1}]=1=[v,Q_2^{y_2}]$ for elements $y_i \in M_i$ and for $1 \ne v \in Y_1 \cap Y_2$, then $Q_1^{y_1}=Q_2^{y_2}$.
		\end{itemize}
	\end{subresult}
	Suppose first that $i=1$ and $Q_1^{x_1} \le M_2$. Then there exists $t \in M_1$ such that we can write
	\begin{center}
		\begin{tikzcd}
			Q_1 \ar[r,"x_1"] &	Q_1^{x_1} \ar[r,"u\inv"] & (Q_1^{x_1})^{u\inv} \ar[r,"t"] & ((Q_1^{x_1})^{u\inv})^t \le S.
		\end{tikzcd}
	\end{center}
	We have $Q_1 \le N_\L(Y_1)$ with $Y_1 \in S_{x_1}$, $Q_1^{x_1} \le M_2 \le N_\L(Y_2)$ with $Y_2 \le S_{u\inv}$ and $(Q_1^{x_1})^{u\inv} \le N_\L(Y_1)$ with $Y_1 \le S_t$. Thus we may apply \ref{lemma:conj of p-subgr and fus systems}, and obtain that $((Q_1^{x_1})^{u\inv})^t$ is an $\L$-conjugate of $Q_1$; note also that $(u\inv,t)$ is in the domain of $\L$ via $(Y_2,Y_1,Y_1)$. Since $Q_1$ is weakly closed, $((Q_1^{x_1})^{u\inv})^t = Q_1$, hence $Q_1^{x_1} = (Q_1^{t\inv})^u = Q_1^{t\inv u}$. Setting $x_2 := (t\inv)^u \in M_2$ we obtain $Q_1^{x_1} = Q_1^{u x_2} = Q_2^{x_2}$.\\
	If instead $i=2$, then $Q_2^{x_2} \le M_1$ and again for a suitable $t \in M_1$ we reason in the same way over the sequence
	\begin{center}
		\begin{tikzcd}
			Q_1 \ar[r,"u"] &	Q_2 \ar[r,"x_2"] & Q_2^{x_2} \ar[r,"t"] & (Q_2^{x_2})^t \le S.
		\end{tikzcd}
	\end{center}
	This proves (i).\\
	For (ii) note that $v^{y_1 \inv} \le Z(Q_1) \le S$ and $Q$-repleteness gives $\<v^{y_1\inv}\> \in \Delta$. Then we have that $(y_1,y_2\inv,u\inv) \in D(\L)$ via $(\<v^{y_1\inv}\>, \< v\>, \< v^{y_2\inv} \>,  \< (v^{y_2\inv})^{u\inv} \>)$ since the first subgroup is in $\Delta$ and all the successive ones are contained either in $Y_1 \le S$ or in $Y_2 \le S$. In particular, also $(u,y_2,y_1\inv) \in D(\L)$. Moreover, $C_\L(v)$ is a subgroup of $\L$; as $Q_2^{y_2} \le C_\L(v)$ and $v^{y_1\inv} \in S$, conjugation by $y_1\inv$ extends to a well-defined homomorphism from $Q_2^{y_2}$. In particular, we get
	\[
	1 \ne v^{y_1\inv} \le Z(Q_1) \cap Z(Q_2^{y_2y_1\inv}) = Z(Q_1) \cap Z(Q_1^{uy_2y_1\inv})
	\]
	and \ref{lemma: L containing Q}(c) gives $uy_2y_1\inv \in N_\L(Q_1)$. Thus $Q_1^{y_1} = Q_2^{y_2}$, which is (ii). In particular, note that $Q_1^{y_1} = Q_2^{y_2} \le M_1 \cap M_2$.\\
	
	Define therefore
	\[
	\mathcal{Q} := \{Q_0 \in Q^{M_1} \cup Q^{uM_2} \mid  Q_0 \le M_1 \cap M_2 \};
	\]
	\eqref{subres:4.14-conjugates of Q}  shows that $\mathcal{Q} = Q_1^{M_1} \cap Q_2^{M_2}$ and if there exists $Q_0 \in \mathcal{Q}$ such that $[v,Q_0]=1$ for some $v \in Y_1\cap Y_2$, then such $Q_0$ is unique and we may therefore denote it by $Q_v$. Set then
	\begin{align*}
		X & := M_1 \cup uM_2 \\
		\V &  := \{v \in Z^\natural \mid \exists \, Q_v \in \mathcal{Q} \text{ with $[v,Q_v] =1$} \}
	\end{align*}
	and note that for every $v \in \V$, $Q_v \le M_1 \cap M_2$; thus $L := \< Q_v \mid v \in \V \>$ is a subgroup of $M_1 \cap M_2 \le \L$.

	\begin{subresult}\label{4.14-4}
		The following hold.
		\begin{itemize}
			\item[(a)] $W_1$, $W_2$ and $Z$ are normalized by $M_1 \cap M_2$; in particular $L\le N_{M_i}(Z)$.
			\item[(b)] $L = \wcl_{X}(Q,M_1 \cap M_2) \le \wcl_{X}(Q, N_{M_1}(Z)) \cap \wcl_{X}(Q,N_{M_2}(Z))$.
			\item[(c)] If $\V = Z^\natural$, then $L = \wcl_{X}(Q, N_{M_1}(Z)) = \wcl_{X}(Q,N_{M_2}(Z))$.
		\end{itemize}
	\end{subresult}
	
	\begin{itemize}
		\item[(a)] Clearly $M_1 \cap M_2$ normalizes $W_1$ and $W_2$, thus also $Z= [W_1,W_2]$, and (a) follows immediately from $L \le M_1 \cap M_2$.
		\item[(b)] Since $Z \le W_1 \cap W_2 \le Y_1 \cap Y_2$, \eqref{subres:4.14-conjugates of Q}\textit{(ii)} holds for the elements of $Z^\natural$. Note that $L$ is generated by $X$-conjugates of $Q$ that are subgroups of $M_1$ and $M_2$, therefore we have $L \le \wcl_{X}(Q,M_1 \cap M_2)$. Suppose that $x \in X$ and $Q^x \le M_1 \cap M_2$. Since $Q^x$ normalizes $Z$, there exists $1 \ne v \in C_Z(Q^x)$ and we get $Q_v=Q^x \le M_1 \cap M_2$, so $v \in \V$ and $Q_v \in L$, showing that $\wcl_{X}(Q,M_1 \cap M_2) \le L$. Thus they are equal. Since $M_1 \cap M_2 \le N_{M_i}(Z)$ one gets $L =  \wcl_{X}(Q,M_1 \cap M_2) \le \wcl_{X}(Q, N_{M_i}(Z))$, proving (b).
		\item[(c)] Suppose that $x \in X$; since $Q^x \le N_{M_i}(Z)$ for both $i \in \{1,2\}$, there exists an element $1 \ne z \in C_Z(Q^x)$.  As $z \in Z^\natural = \V$, we obtain  $Q_z=Q^x \in \mathcal{Q}$; hence $Q^x \le L$, showing that $\wcl_{X}(Q,N_{M_i}(Z)) \le L$ for both $i \in \{1,2\}$. This together with (b) implies (c).
	\end{itemize}
	
	\begin{subresult}\label{4.14-5}
		If $M_i$ is transitive on $W_i$, then $\V = Z^\natural$.
	\end{subresult}
	Since $M_i$ is transitive on $W_i$ and $1 \ne Z \le W_1 \cap W_2$, every $1 \ne v \in Z$ is centralized by some $M_1$-conjugate of $Q_1$ and by some $M_2$-conjugate of $Q_2$. By \eqref{subres:4.14-conjugates of Q} these two conjugates of $Q$ are equal, namely $Q_v \in Q_1^{M_1} \cap Q_2^{M_2}$, and contained in $M_1 \cap M_2$, thus $Q_v \in \mathcal{Q}$. This shows that $Z^\natural = \V$.
	
	\begin{subresult}\label{4.14-6}
		Let $1 \ne z \in C_Z(L)$ and $K_i \le M_i$ be a subgroup transitive on $W_i$. Then $L = Q_z$ and $Z^\natural \subseteq z^{N_{K_i}(L)}$.
	\end{subresult}
	First of all note that $\{z\} \cup K_i \cup L \subseteq M_i$, so all the conjugates in $z^{N_{K_i}(L)}$ are well-defined. The transitivity of $K_i$ implies, by \eqref{4.14-5}, $\V=Z^\natural$; hence for every $v \in Z^\natural$ one has $Q_v \le L$, so also $[z,Q_v]=1$. Then $Q_v = Q_z$ and one must have $L = Q_z$; let $k \in K_i$ be such that $z^k=v$ (such a $k$ exists by the transitivity of $K_i$), then $Q_z^k = Q_v = Q_z$, that is $k \in N_{K_i}(Q_z) = N_{K_i}(L)$.

	\vspace{0.60cm}
	
	The various cases of \cite[Theorem C.3]{Meier-Stell-Stroth:2012} can now be discussed. Here Case (12), that is the natural $\Sym(6)$- or $\Alt(6)$-module for $n=6$, can be treated uniformly with Case (2) since they are also natural $\Sp_4(2)$-, respectively $\Sp_4(2)'$-, modules.\\
	Now that the issues involving conjugates of $Q$ have been solved, the arguments used in each case reflect very closely those appearing in \cite[Lemma 4.14]{Meier-Stell-Stroth:2012} and, indeed, we will often directly refer to the proofs there even though the arguments for the locality $\L$ neew the adaptation worked out in points \eqref{4.14-1} to \eqref{4.14-6} now and then.

	\begin{caseresult}
		Case (1) holds with $n \ge 3$, so $\ov{J_1} \cong \SL_n(q)$ and $W_1$ is a corresponding natural module. Then Theorem~\ref{theorem DL}(1) holds.
	\end{caseresult}
	This is essentially \cite[Lemma 4.14, Case 1]{Meier-Stell-Stroth:2012}.
	
	\begin{caseresult}
		One of Case (2), with $n \ge 2$, and Case (12), with $n=6$, holds, hence $\ov{J_1} \cong \Sp_{2n}(q)$ or $\Sp_4(2)'$ and $W_1$ is a corresponding natural module. Then Theorem~\ref{theorem DL}(2) holds.
	\end{caseresult}
	This is essentially \cite[Lemma 4.14, Case 2]{Meier-Stell-Stroth:2012}, hoe

	Consider first the case when $p$ is odd. The argument uses the following statement, the proof of which is just as in \cite[Lemma 2.26]{Meier-Stell-Stroth:2012} up to replacing $N_G(Q)$ by $N_\L(Q)$, $C_G(Q)$ by $C_\L(Q)$ and recalling that we have chosen $Q$ such that $Q=Q^\bullet = O_p(N_\L(Q))$.
	\begin{lemma}[2.26]
		\label{lemma2.26}
		Let $X \le Y_M$ be an $M$-submodule such that $X$ is a natural $\Sp_{2m}(p^k)$-module for $M^\circ$ with $2m \ge 4$ and $p \ne 2$. Then $X \le Q$.
	\end{lemma}
	\begin{proof}
		Identical to the proof of \cite[Lemma 2.26]{Meier-Stell-Stroth:2012}.
	\end{proof}
	Using \ref{lemma2.26} one argues precisely as in the LST, obtaining the decomposition
		\[
	Y_1 = [Y_1,Z(\ov{J_1})] \times C_{Y_1}(Z(\ov{J_1})) = W_1 \times C_{Y_1}(J_1),
	\]
	which, together with \eqref{4.14-1}(b), gives $C_{Y_i}(J_i) =1$, hence $Y_1=W_1$ and Theorem~\ref{theorem DL}(2) holds.\\
	
	Let now $p=2$. The following lemma is then applied.
	\begin{lemma}[2.25]
		\label{lemma2.25}
		Let $p=2$ and $Y \le Y_M$ be an $M$-submodule such that $I:=[Y,M_\circ]$ is a natural $\Sp_{2m}(2^k)^\prime$-module for $M_\circ$, with $m\ge 1$.
		\begin{itemize}
			\item[(a)] If $L \le \L$ is a subgroup with $Y \le L$ and $A \norm L$ is a normal $p$-subgroup of $L$ such that $I \le A \le M$, then $Y \le O_2(L)$.
			\item[(b)] If $I \le Q$, then also $Y \le Q$.
		\end{itemize} 
	\end{lemma}
	\begin{proof}
		Up to few remarks, it is as in \cite[Lemma 2.25]{Meier-Stell-Stroth:2012}.\\
		We have $C_\L(M^\circ)=1$ by \ref{lemma:dichotomy for subgroups containing Q}, which implies $C_Y(M_\circ)=1$.\\
		That $[M^\circ,C_M(I)] \le C_M(Y)$ can still be obtained invoking \cite[Lemma 1.52(c)]{Meier-Stell-Stroth:2012}, since it is applied to $M$ and $Q$ is large in $M$.\\
		Note also that $I \le L \cap M$ by hypothesis and that $I \in \Delta$ since $I \le Y \le Y_M$ and it is normal in $M$ (thus normalized by $S$). This is sufficient to obtain (a).\\
		Now (b) follows from (a) with $L=N_\L(Q)$ and $A=Q=O_p(L)$.
	\end{proof}
	Now the reasoning is exactly as in \cite[Lemma 4.14, Case 2]{Meier-Stell-Stroth:2012}, with an application of \ref{lemma2.25}, proving that Theorem~\ref{theorem DL}(2) holds again.
	
	\setcounter{caseresult}{2}
	\begin{caseresult}
		Case (3), so $\ov{J_1} \cong \SU_n(q)$ with $n \ge 4$ and $W_1$ is a corresponding natural module, does not occur.
	\end{caseresult}
	The proof is as in \cite[Lemma 4.14, Case 3]{Meier-Stell-Stroth:2012}.
	
	\begin{caseresult}
		Case (4) holds, that is
		\[
		\ov{J_1} \cong \begin{cases*}
			\Omega_{2n}^+(q) \quad \text{for $2n \ge 6$,} \\
			\Omega_{2n}^-(q) \quad \text{for $p=2$ and $2n \ge 6$,} \\
			\Omega_{2n}^-(q) \quad \text{for $p$ odd and $2n \ge 8$,} \\
			\Omega_{2n+1}(q) \quad \text{for $p$ odd and $2n+1 \ge 7$,} \\
			\O_4^-(2), \\
			\O_{2n}^\epsilon(q) \quad \text{for $p=2$ and $2n \ge6$.}
		\end{cases*}
		\] 
		and $W_1$ is a corresponding natural module. Then Theorem~\ref{theorem DL}(5) or \ref{theorem DL}(4:1) holds.
	\end{caseresult}
	Here one needs to pay extra care that certain operations are possible in the locality $\L$; whence we reproduce the full argument.\\
	
	In all cases $\ov{R_i} = F^*(\ov{J_i}) \cong \Omega_m^\epsilon(q)$; now \eqref{4.14-1}(c) also gives $\ov{M_i^\circ} = \ov{R_i}$ and $W_i=[Y_i,M_i^\circ]$. Since $Y_i$ is a $p$-subgroup acting $\bK_j$-linearly on $W_j$ and it normalizes $\ov{R_j}$, \cite[Lemma B-35(d)]{Meier-Stell-Stroth:2012} shows that either $\ov{Y_iR_j}=\ov{T_j}=\ov{R_j} \cong \Omega_{2n}^\epsilon(q)$ or $p=2$ and $\ov{T_j} \cong \O_m^\epsilon(q)$, with $W_1$ a corresponding natural module.
	
	Suppose that $|Z| \le q$; \cite[Lemma B.9(c)]{Meier-Stell-Stroth:2012} gives $p=2$, $Z$ not singular in $W_1$ and $|\ov{W_2}| =2$. $W_2$ being an offender on $W_1$ implies then $q=2$ and so $\ov{T_1} \cong \O_{2n}^\epsilon(2)$. Hence \ref{lemma4.10} gives Theorem~\ref{theorem DL}(5).
	
	Suppose now that $W_1 < Y_1$. Then \cite[Theorem C.22]{Meier-Stell-Stroth:2012} gives $\ov{J_1} \cong \O_6^+(2) \cong \Sym(8)$ and \cite[Theorem C.4(h)]{Meier-Stell-Stroth:2012} shows that every offender in $\ov{T_i}$ on $Y_i$ is a best offender. Recall that $Y_2$ is an offender on $Y_1$, then always \cite[Theorem C.22]{Meier-Stell-Stroth:2012} shows that $\ov{Y_2}$ in $\ov{M_1}$ is generated by transpositions, so $[Y_1,Y_2]$ contains a non-singular vector of $Y_1$. We can now apply again \ref{lemma4.10}, hence Theorem~\ref{theorem DL}(5) holds in this case as well.
	
	Assume then that $W_1=Y_1$ and $|Z| > q$. If $\ov{J_1} \cong O_4^-(2) \cong \Sym(5)$, then $\ov{T_1} \in \{\Sym(5), \, \Alt(5)\}$ and \cite[Theorem C.4(g)]{Meier-Stell-Stroth:2012} shows that $\ov{W_2}$, an offender on $W_1$, is generated by transpositions of $\ov{T_1}$. Thus $Z=[W_1,W_2]$ contains non-singular vectors of $W_1$ and once again \ref{lemma4.10} yields the conclusion of Theorem~\ref{theorem DL}(5).
	
	Let then $\ov{J_1} \not \cong \O_4^-(2)$, so $2n \ge 6$. Looking for a contradiction, suppose that $Y \le Q$, so $Y_i \le Q_i$ for $i \in \{1,2\}$. By \cite[Lemma B.6(e)]{Meier-Stell-Stroth:2012} and by the quadratic action of $W_1$ on $W_2$, $Z$ is an isotropic subspace of $W_2$. Now \cite[Lemma B.5]{Meier-Stell-Stroth:2012} shows that the singular vectors of $W_2$ contained in $Z$ form a $\bK_2$-subspace of $Z$ with codimension at most $1$. Since $|Z| > q$, there exists then a singular vector $1 \ne v \in Z$. Then $v$ is $p$-central in $M_2$, so there exist $x \in M_2$ and $w \in W_1$ such that $w^{ux} =v$ and $Q_2^x = Q^{ux}$. Here note that $(u,x) \in \Delta$ via $(Y_1,Y_2,Y_2)$. Then $[w,Q]=1$, so $\<w\> \in \Delta$ and, therefore, also $\<v\> \in \Delta$ and $(Q!)$ implies $C_\L(w) \le N_\L(Q)$. Thus $Y \le O_p(C_\L(w))$ and we get $Y_2 \le O_p(C_\L(v))$ by conjugating back via $ux$ as a map $C_\L(w) \fto C_\L(v)$. In particular, $W_2 \le O_p(C_\L(v)) \cap C_{M_1}(v) \le O_p(C_{M_1}(v))$.\\
	Since $v \in Z$, it is also a vector of $W_1$; suppose it is singular, hence centralized by a Sylow $p$-subgroup of $M_1$. This contradicts the Point-Stabilizer Theorem \cite[Theorem C.8]{Meier-Stell-Stroth:2012}, $W_2$ being a non-trivial offender on $W_1$. Then $v$ is non-singular in $M_1$, so $|O_p(C_{\ov{M_1}}(v))| \le 2$ and then also $|W_2 / C_{W_2}(W_1)| \le 2$. Since $Z \ne 1$, this means that $q=2$ and $C_{W_2}(W_1)$ is a hyperplane in $W_2$, hence $|Z| = |C_{W_2}(W_1)^\perp| =2 = q$, yielding a contradiction to $|Z| > q=2$. hence to $Y \le Q$. Thus Theorem~\ref{theorem DL}(4:1) holds.
	
	\begin{caseresult}
		Case (5), that is $p=2$, $\ov{J_1}\cong G_2(q)$ and $W_1$ is a corresponding natural module, does not occur.
	\end{caseresult}
	This is \cite[Lemma 4.14, Case 5]{Meier-Stell-Stroth:2012}.
	
	\begin{caseresult}
		Case (6) holds, so $\ov{J_1} \cong \SL_n(q) / \<(-id)^{n-1}\>$ for $n \ge 5$ and $W_1$ is the corresponding exterior square of a natural $\SL_n(q)$-module. Then Theorem~\ref{theorem DL}(4:2) holds.
	\end{caseresult}
	The reasoning given in \cite[Lemma 4.14, Case 6]{Meier-Stell-Stroth:2012} works for our locality $\L$ as well.
	
	\begin{caseresult}
		Case (7), that is $\ov{J_1} \cong \Spin_{7}(q)$ and $W_1$ is the corresponding spin-module, does not occur.
	\end{caseresult}
	This is \cite[Lemma 4.14, Case 7]{Meier-Stell-Stroth:2012}.
	
	\begin{caseresult}
		Case (8) holds, that is $\ov{J_1} \cong \Spin_{10}^+(q)$ and $W_1$ is the corresponding halfspin-module. Then Theorem~\ref{theorem DL}(4:3) holds.
	\end{caseresult}
	The proof is as in \cite[Lemma 4.14, Case 8]{Meier-Stell-Stroth:2012}.
	
	\begin{caseresult}
		Case (9), that is $\ov{J_1} \cong 3 \cdot \Alt(6)$ and $|W_1|=2^6$, does not occur.
	\end{caseresult}
	The argument for the locality $\L$ requires a refinement with respect to that in \cite[Lemma 4.14, Case 9]{Meier-Stell-Stroth:2012}, since it needs the detection of a specific subgroup of $\L$ (namely $N_\L(Z)$ below), a step that is unnecessary when working within a group $G$ from the beginning.\\
	
	By \eqref{4.14-1}(a) and (c) $\ov{R_1} = \ov{J_1}$ and either $\ov{M_1^\circ} = \ov{J_1}$ or $\ov{M_1^\circ} \cong 3 \cdot \Sym(6)$; now $Z(\ov{J_1})$ has order $3$, so it acts coprimely on $Y_1$ and we get
	\[
	Y_1 = [Y_1,Z(\ov{J_1})] \times C_{Y_1}(Z(\ov{J_1})).
	\]
	As $[Y_1,Z(\ov{J_1})]$ is a non-trivial $J_1$-submodule of $W_1$, it equals $W_1$; moreover, $[J_1, C_{Y_1}(Z(\ov{J_1}))] \le W_1 \cap C_{Y_1}(Z(\ov{J_1})) =1$, so $Y_1 = W_1 \times C_{Y_1}(J_1)$. Since $C_{Y_1}(J_1) = 1$ by \eqref{4.14-1}(b), we get $Y_1 = W_1$, so also $Y_2 =W_2$. In particular, $Y_2$ is a non-trivial offender on $Y_1$ and, as $\bK_i=\End_{J_i}(Y_i)$, $\bK_i \cong \bF_4$. Applying \cite[Theorem C.4(e)]{Meier-Stell-Stroth:2012} we get $|Y_i|=|C_{Y_i}(Y_j)| = 4$, so also $Y_1$ is a non-trivial offender on $Y_2$, $Z=C_{Y_i}(Y_j)$ with $|Z|=2^4$ and a non-trivial offender is uniquely determined by its inclusion in a Sylow $p$-subgroup of $\ov{J_i}$; thus all the non-trivial offenders on $W_i$ in $\ov{J_i}$ are $\ov{J_i}$-conjugate and we also have $Y_i \le J_j$. Since $Y_2 \le S$, $Y_2^u$ is a well-defined subgroup of $M_2$ and a non-trivial offender on $Y_1^u = Y_2$. Thus in $\ov{M_2}$ we get $\ov{Y_1} = \ov{(Y_2^u)^h}$ for some $h \in M_2$; as $(u,h) \in D(\L)$ via $(Y_1, Y_2, Y_2^h=Y_2)$, we have $\ov{Y_1}=\ov{Y_2^{uh}}$. Set $g:=uh$ and note that $Y_1^g = Y_2$, $Y_2^g = Y_1$ and $Z^g=[Y_1^g,Y_2^g] =Z$. Since $Y_2 \le S$, it normalizes $Y_1 \cap Z(Q) \ne 1$, thus $1 \ne C_{Y_1 \cap Z(Q)}(Y_2) = Z \cap Z(Q)$ and $Q$-repleteness gives $Z \in \Delta$. In particular, $g \in N_\L(Z)$. Define
	\[
	\Gamma := \{ [y_1,y_2] \mid y_i \in Y_i \setminus Z \}
	\]
	and note that, since $Y_j \le J_i$, $[Y_i,y_j]$ is, for every $y_j \in Y_j \setminus Z$, a non-trivial $\bK_i$-subspace of $Y_i$. Consider the map $(y_j - \id): Y_i \fto Y_i$ given by $(y_j - \id)(v) = [v,y_j]$; then $ker(y_j - \id) = C_{Y_i}(y_j) \ge Z$. As $|Z|=2^4$ we get that $[Y_i,y_j]$ is $1$-dimensional over $\bK_i$. Since $|Y_i/Z|=2^2$, there are $3$ distinct spaces of the form $[Y_i,y_j]$ in $Y_i$.
	Therefore $|\Gamma|=9$ and we may build two (a priori not necessarily distinct) partitions of $\Gamma$:
	\[
	\Gamma_1:= \{ [Y_1,y_2]^\natural \mid y_2 \in Y_2 \setminus Z \} \qquad \text{and} \qquad 
	\Gamma_2:= \{ [Y_2,y_1]^\natural \mid y_1 \in Y_1 \setminus Z \}
	\]
	where $|[Y_i,y_j]^\natural | = 3$ and $|\Gamma_i|=3$.\\
	We obtain that the element $g \in N_\L(Z)$ acts in the following way:
	\[ \Gamma^g = \Gamma, \quad \Gamma_1^g = \Gamma_2 \quad \text{and} \quad \Gamma_2^g=\Gamma_1.
	\]
	As $\ov{Y_2}$ is a non-trivial offender on $\ov{Y_1}$, there exists a Sylow $2$-subgroup of $\ov{M_1}$ normalizing $\ov{Y_2}$, hence with preimage contained in $C_{N_\L(Z)}(\{ \Gamma_1, \Gamma_2 \})$. Since also $g \in N_\L(Z)$ and it is acting as a non-trivial transposition on $\{\Gamma_1,\Gamma_2\}$, we may suppose $o(g)=2$ and we obtain the following situation:
	\[
	g \in N_{N_\L(Z)}(\{\Gamma_1,\Gamma_2\}) \qquad \text{and} \qquad |C_{N_\L(Z)}(\{ \Gamma_1, \Gamma_2 \})|_2 = |N_\L(Z)|_2.
	\]
	We conclude that $N_{N_\L(Z)}(\{\Gamma_1,\Gamma_2\}) = C_{N_\L(Z)}(\{ \Gamma_1, \Gamma_2 \})$, so $\Gamma_2 = \Gamma_1^g = \Gamma_1$. In particular, for any fixed $y_1 \in Y_1 \setminus Z$ we have $[Y_2,y_1]^\natural \in \Gamma_1$. However note that, for every $y_2 \in Y_2 \setminus Z$, $[Y_2,y_1]^\natural \cap [Y_1,y_2]^\natural \supseteq \{ [y_2,y_1] \}$; in other words, every element of the partition $\Gamma_1$ has non-trivial intersection with every other element of $\Gamma_1$, a contradiction.
\end{proof}

\begin{proof}[Proof of Theorem~\ref{theorem DL}]
	By Lemma \ref{lemma4.9} the triple $(\ov{M_1},Y_1, \ov{Q_1})$ satisfies the hypothesis of the $Q!FF$-Module Theorem, \cite[C.24]{Meier-Stell-Stroth:2012} and we have then covered all the possible outcomes: Lemma \ref{lemma4.12} covers C.24(1), Lemma \ref{lemma4.13} covers C.24(2) when $W_1$ is not a simple $R_1$-module and \ref{lemma4.14} covers C.24(2) for $W_1$ a simple $R_1$-module.
\end{proof}

\bibliographystyle{alpha.bst}
\bibliography{rep-coh-local}

\end{document}